\documentclass[10pt,a4paper]{article}

\usepackage[utopia]{mathdesign}

\usepackage{float}
\usepackage{graphicx}
\usepackage{longtable}

\usepackage{amsmath}

\usepackage{cite}

\usepackage{theorem}
\usepackage{enumerate}
\usepackage[dvipsnames]{xcolor}
\usepackage{tikz}
\usetikzlibrary{arrows.meta,calc}

\usepackage[left=2.5cm,right=2.5cm,top=2.5cm,bottom=2.5cm]{geometry}
\usepackage{microtype}
\usepackage{hyperref}
\usepackage[disable]{todonotes}

\definecolor{linkblue}{HTML}{1F3B73}  

\definecolor{urlblue}{HTML}{2E75B6}   

\definecolor{citeblue}{HTML}{1F3B73}   
\hypersetup{colorlinks, breaklinks,
  linkcolor   = linkblue,
  urlcolor    = urlblue,
  anchorcolor = linkblue,
  citecolor   = citeblue}
\definecolor{labelkey}{rgb}{0.6,0.6,0.6}
\definecolor{refkey}{rgb}{0,0.6,0.0}
\definecolor{Brown}{rgb}{0.45,0.0,0.05}
\definecolor{lime}{rgb}{0.00,0.8,0.0}
\definecolor{lblue}{rgb}{0.5,0.5,0.99}

\usepackage{titlesec}
\titlespacing*{\section}{0pt}{8pt}{4pt}
\titlespacing*{\subsection}{0pt}{6pt}{3pt}
\titlespacing*{\paragraph}{0pt}{4pt}{0.5em}

\usepackage{array}
\newcolumntype{L}[1]{>{\raggedright\let\newline\\\arraybackslash\hspace{0pt}}m{#1}}
\newcolumntype{C}[1]{>{\centering\let\newline\\\arraybackslash\hspace{0pt}}m{#1}}
\newcolumntype{R}[1]{>{\raggedleft\let\newline\\\arraybackslash\hspace{0pt}}m{#1}}
\renewcommand{\arraystretch}{1.2}

\def\disp{\displaystyle}
\def\h{\hfill\Box}
\def\la{\left\langle}
\def\ra{\right\rangle}

\def\eq{\begin{equation}}
\def\eeq{\end{equation}}

\def\al{\alpha}
\def\be{\beta}
\def\gg{\gamma}

\def\ve{\varepsilon}
\def\th{\theta}

\def\vTh{\vartheta}

\def\lm{\lambda}
\def\Lm{\Lambda}
\def\O{\Omega}

\def\ph{\varphi}

\def\Lto{\Longrightarrow}

\def\tto{\rightrightarrows}
\def\dn{\downarrow}

\def\R{\mathbb{R}}
\def\N{\mathbb{N}}

\def\B{\mathbb{B}}

\def\Hat{\widehat}
\def\Tilde{\widetilde}
\def\Bar{\overline}

\def\oa{\overline{a}}

\def\ox{\overline{x}}
\def\oy{\overline{y}}

\def\ou{\overline{u}}
\def\ov{\overline{v}}

\def\oT{\overline{T}}
\def\T{\overline{T}}

\def\cA{\mathcal{A}}
\def\cU{\mathcal{U}}

\DeclareMathOperator{\co}{co}
\DeclareMathOperator{\gph}{gph}
\DeclareMathOperator{\epi}{epi}

\DeclareMathOperator{\dom}{dom}

\DeclareMathOperator{\dist}{dist}

\DeclareMathOperator{\bd}{bd}
\DeclareMathOperator{\inter}{int}

\DeclareMathOperator{\var}{var}
\DeclareMathOperator{\supp}{supp}

\DeclareMathOperator*{\Limsup}{Lim\,sup}

\newtheorem{theorem}{Theorem}[section]
\newtheorem{lemma}[theorem]{Lemma}
\newtheorem{corollary}[theorem]{Corollary}

\newtheorem{proposition}[theorem]{Proposition}
\newtheorem{definition}[theorem]{Definition}
\theoremstyle{plain}{\theorembodyfont{\rmfamily}
}
\theoremstyle{plain}{\theorembodyfont{\rmfamily}
}
\theoremstyle{plain}{\theorembodyfont{\rmfamily}
}
\theoremstyle{plain}{\theorembodyfont{\rmfamily}
\newtheorem{example}[theorem]{Example}}

\theoremstyle{plain}{\theorembodyfont{\rmfamily}
\newtheorem{remark}[theorem]{Remark}}

\allowdisplaybreaks

\title{Optimal Control of Prox-Regular Sweeping Processes\\
with Variable Time}
\author{%
Tan H.\ Cao\thanks{Department of Applied Mathematics and Statistics,
Stony Brook University -- SUNY Korea (Asian Campus), Yeonsu-Gu, Incheon,
Republic of Korea. Email: \texttt{tan.cao@stonybrook.edu}.}
\and
Boris  S.\ Mordukhovich\thanks{Department of Mathematics and Institute for Artificial Intelligence and Data Science, Wayne State University, Detroit, Michigan, USA.
Email: \texttt{boris@math.wayne.edu}. Research of this author was partly supported by the US National Science Foundation under grant DMS-2204519 and by the Australian Research Council under Discovery Projects DP-190100555 and DP-250101112.}
\and
Trang Nguyen\thanks{Department of Mathematics and Statistics,
South Dakota State University, Brookings, South Dakota 57007, USA.
Email: \texttt{trang.nguyen@sdstate.edu}. Research of this author was partly supported by CAPITAL Services Inc.} 
\and
Thi Phung\thanks{Department of Mathematics, Wayne State University,
Detroit, Michigan, USA.
Email: \texttt{thi.pn@wayne.edu}.}
}
\date{\today}

\begin{document}
\maketitle
{\let\thefootnote\relax}

\begin{abstract}

The paper is devoted to a broad class of optimal control problems governed by constrained sweeping processes with controlled dynamics, controlled prox-regular moving sets, and variable time included in the optimization process. Employing an advanced version of the method of discrete approximations, we verify its well-posedness with appropriate strong convergence of feasible and optimal solutions. This approach leads us to considering a parametric family of discrete-time controlled sweeping processes, which are reduced in turn to nondynamic mathematical programs with special types of geometric and functional constraints. Invoking powerful tools of first-order and second-order variational analysis and generalized differentiation, we derive necessary optimality conditions for discrete-time and then---by passing to the limit over the diminishing step of discretization---continuous-time controlled sweeping processes. A crucial role in our device is played by explicit coderivative evaluations of set-valued mapping associated with the sweeping dynamics. In this way, we establish novel necessary optimality conditions for local minimizers of continuous-time controlled sweeping processes that contain absolutely continuous and BV adjoint arcs, local and global maximization conditions, $L^2$ active set multipliers, vector measures generated by sweeping and control constraints, and free-time transversality condition expressed through the averaged Hamiltonian. For suitably selected multipliers, we also obtain nondegenerate forms of these conditions when the optimal trajectory starts on the boundary of the moving set. The efficiency of the obtained results is illustrated by some examples related to practical modeling.
\end{abstract}
\vspace*{-0.1in}

\section{Introduction and Problem Formulation}\label{S1}
\setcounter{equation}{0}

The sweeping process, introduced by Moreau in the 1970s \cite{Moreau1974}, is the discontinuous differential inclusion
\[
\dot{x}(t)\in -N_{C(t)}\bigl(x(t)\bigr)\quad\text{a.e. }t\in[0,T],
\]
where $C(t)$ is a convex moving set and the normal cone is taken in the sense of convex analysis. Originally motivated by applications to elastoplasticity and quasistatic mechanics, the sweeping process model is now used well beyond those settings; namely, in hysteresis, electrical circuits, traffic equilibria, crowd motion, etc.; see, e.g., \cite{Adly2017,MauryVenel2011,ColomboHenrionHoangMordukhovich2016} and the references therein. The characteristic feature of sweeping processes is that the state is not steered toward a target but \emph{swept} by the boundary of $C(t)$: the trajectory evolves freely while it lies in the interior, and the normal reaction acts only upon contact.

The normal cone mapping in the description of the sweeping dynamics is heavily {\em discontinuous} and unbounded, which places sweeping processes outside the scope of well-developed results for differential inclusions with Lipschitzian right-hand sides. Furthermore, the state constraint $x(t)\in C(t)$ is not imposed but \emph{generated} by the dynamics. When the moving set is fixed, the Cauchy problem is well-posed  while admitting a {\em unique} solution; thus it doesn't allow any optimization. A meaningful optimization problem requires therefore extra {control} actions centering either the moving set, or additive perturbations of the dynamics, or both.

Optimal control problems for sweeping processes with control actions entering polyhedral moving sets were introduced and analyzed in \cite{ColomboHenrionHoangMordukhovich2012,ColomboHenrionHoangMordukhovich2016}. Polyhedral sweeping processes with controls in the additive perturbations of the dynamics were studied in \cite{CaoMordukhovich2017,ColomboMordukhovichNguyen2020,HoangMordukhovich2019}. The aforementioned papers employed the {\em method of discrete approximations} to derive necessary optimality conditions of the {\em generalized Euler-Lagrange} and {\em Hamiltonian} types for local minimizers of controlled sweeping processes by passing to the limit from their well-posed finite-difference counterparts by using advanced tools of variational analysis and generalized differentiation. The established optimality conditions 
were subsequently applied to {\em controlled models} in robotics, traffic flows, crowd motions, nonsmooth mechanics, etc.; see \cite{CaoMordukhovich2017,ColomboMordukhovichNguyen2019,HoangMordukhovich2019}. A nonpolyhedral sweeping control problem  with control functions acting simultaneously in an additive perturbation/drift and in a nonconvex {\em prox-regular} moving set was investigated in \cite{CaoMordukhovich2019} and applied therein to a planar crowd motion model. Subsequent developments along these lines appeared in, e.g.,  \cite{CaoMordukhovich2019a,CaoColomboMordukhovich2021a,CaoColomboMordukhovich2021b,CaoKhalilMordukhovich2022b,CaoEtAl2025}. Yet another version of the discrete approximation method was developed in \cite{ColomboGidoni2021} to derive necessary optimality conditions for a sweeping control problem with an uncontrolled moving set of smooth boundary under controlled perturbations. The results of \cite{ColomboGidoni2021} were applied to optimal control of soft crawlers.

Complementing to the method of discrete approximations, some {\em continuous approximation} methodologies were developed in a number of publications; see, e.g., \cite{ArroudColombo2018,DePinhoEtAl2019,DePinhoEtAl2022,PinhoEtAl2023,HermosillaPalladino2022,ChamounZeidan2025,Bettiol2025,NourZeidan2022,NourZeidan2023a}. They are based on {\em exact penalization} procedures using either the distance function, or the exponential penalization, or both. In this way, certain sweeping control counterparts of the Pontryagin maximum principle (PMP) were established in the aforementioned publications under various assumptions. In contrast to local maximization conditions obtained by using the method of discrete approximations that are global under some convexity, the PMP versions established by continuous approximations are global without imposing convexity. On the other hand, employing discrete approximations with the subsequent passage to the limit led us to necessary optimality conditions for controlled sweeping processes that were not available via using continuous approximations.

The above discussions concerned sweeping control problems on the {\em fixed time} interval $[0,T]$. The main attention of this paper is paid to optimal control of the sweeping dynamics on the {\em variable time} interval (labeled as {\em free-time} problems), where the horizon $T$ is included in the optimization process. While it has been well understood in classical optimal control theory that free-time problems can be reduced to fixed-time ones by a certain transformation, this is not the case for sweeping optimal control; see a detailed discussion in \cite{CaoMordukhovich2026}. Therefore, free-time sweeping optimal control problems require their own consideration.

The method of discrete approximations was employed in \cite{Mordukhovich1996} to derive generalized Euler-Lagrange and Hamiltonian necessary conditions for free-time optimal control problems governed by nonconvex Lipschitzian differential inclusions. Then modifications of this method were developed in \cite{CaoKhalilMordukhovich2022a,ColomboMordukhovichNguyen2024,Mordukhovich2026} to establish optimality conditions in free-time sweeping control problems with polyhedral moving sets. The polyhedral structure appears to be rather restrictive from both viewpoints of variational theory and applications. In particular, it forces constraint surfaces to be flat, which significantly simplifies
the structure of optimality conditions and results in missing important relationships between ingredients of the optimality system, especially for free-time problems; see below. 

This paper addresses the following Bolza-type dynamic optimization problem $(P)$ for {\em free-time controlled sweeping systems} with controls acting in both additive perturbations and {\em nonconvex moving sets}:
\begin{equation}\label{e1}
\mbox{minimize }\;J[x,u,a,T]:=\varphi\big(x(T),T\big)+\int_0^T \ell\big(t,x(t),u(t),a(t),\dot x(t),\dot u(t)\big)\,dt
\end{equation}
over the two groups of control functions $u(\cdot)$ and $a(\cdot)$ {and the corresponding trajectories $x(\cdot)$} together with the terminal time $T\ge 0$ subject to the constrained sweeping differential inclusion
\begin{equation}\label{e2}
\begin{cases}
\dot x(t)\in -N_{C_u(t)}\big(x(t)\big)+f\big(x(t),a(t)\big) & \text{ a.e. } t\in[0,T],\\[0.4ex]
x(0)=x_0\in C_u(0),\\[0.4ex]
(x(T),T)\in \Xi_x\times \Xi_T\subset \R^n\times[0,\infty).
\end{cases}
\end{equation}
where the moving set $C_u$ is structured as
\begin{equation}\label{e3}
C_{{u}}(t)=C+u(t)\;\mbox{ with }\;C=\bigcap_{i=1}^m\big\{x\;\big|\;g_i(x)\ge 0\big\},\quad t\in[0,T],
\end{equation}
and where the normal cone in \eqref{e2} is understood in an appropriate sense of variational analysis; see Section~\ref{S2}. 

The admissible controls in the moving sets and dynamic perturbations are given by, respectively,
\begin{equation}\label{e4}
\cU:=\bigcup_{T\ge0}\left\{ u:[0,T]\to \R^n\,\left|\,
\begin{array}{l}
u(\cdot)\text{ is absolutely continuous and}\\
\nu_i(u(t))\le 0\ \text{for all\ }t\in[0,T],\ i=1,\ldots,s
\end{array}\right.\right\},
\end{equation}
\begin{equation}\label{e5}
\cA:=\bigcup_{T\ge0}\left\{ a:[0,T]\to \R^d\;\; \text{is measurable and}\;\; a(t)\in A \;\;\text{a.e. } t\in [0,T] \right\}.
\end{equation}
Observe that the sweeping dynamics in \eqref{e2} automatically generates the mixed state-control constraints
\begin{equation}\label{e6}
g_i\big(x(t)-u(t)\big)\ge 0\quad\text{for all }\quad t\in[0,T],\quad i=1,\ldots,m.
\end{equation}

The main contribution of the paper is a comprehensive analysis of problem $(P)$ by constructing its well-posed discrete approximations with deriving necessary optimality conditions for discrete-time problems and then, by passing to the limit from them, necessary optimality conditions for appropriate local minimizers of the original sweeping control problem. In this way, by applying advanced tools of variational analysis and (first- and second-order) generalized differentiations, we arrive at novel optimality conditions for both discrete-time and continuous-time controlled sweeping processes with prox-regular moving sets. The obtained nondegenerate conditions of the generalized Euler-Lagrange and modified Hamiltonian types reflect the interaction between various ingredients of the primal and adjoint systems, measure multipliers, and the optimal horizon. They are new even for fixed-interval and polyhedral controlled sweeping processes.

The rest of the paper is organized as follows. Section~\ref{S2} collects the needed preliminaries from variational analysis and generalized differentiation. In Section~\ref{S2a}, we formulate and discuss the standing assumptions on the data of $(P)$. Section~\ref{S2b} constructs a relaxed version of problem $(P)$ and defines the class of local minimizers for $(P)$ studied in the paper. Section~\ref{S3}, based on the preceding publication \cite{CaoMordukhovich2026}, presents the design, convergence results, and necessary optimality conditions for discrete approximations of the original sweeping control problem $(P)$ and its designated local minimizer. In Section~\ref{S4}, we formulate the main results of the paper---necessary optimality conditions for the original continuous-time sweeping optimal control problem---and discuss their new content. Section~\ref{S6a} is devoted to the proofs of the main necessary conditions for sweeping optimal solutions. In Section~\ref{S6b}, we provide important consequences of the main results related to the extended Hamiltonian maximization and related conditions. Section~\ref{S6c} presents major specifications of the developed necessary optimality conditions and discusses their advances in various settings including fixed-time, polyhedral problems, etc. Section~\ref{S9b} establishes the nondegeneracy of the optimality conditions at the endpoints for suitably selected multipliers. In Section~\ref{S7}, we illustrate applications of optimality conditions for controlled sweeping processes to solve a minimum-time navigation problem around a circular keep-out zone and a two-agent corridor crowd motion model. The concluding Section~\ref{S8} summarizes the obtained results and discusses some topics of our future research.\vspace*{-0.05in}

\section{Prox-Regularity and Generalized Differentiation}\label{S2}

In this paper, we work in a finite-dimensional space $\R^n$. Unless otherwise stated, $\|\cdot\|$ denotes, respectively, the Euclidean norm, $\la\cdot,\cdot\ra$ the standard inner product, and
\[
B(x,\ve):=\big\{y\in\R^n\;\big|\;\|y-x\|<\ve\big\}
\]
the open ball centered at $x\in\R^n$ with radius $\ve>0$. For a nonempty set $\O\subset\R^n$, $\bd\O$ signifies its boundary.

We start with recalling the notion of {\em uniformly prox-regular sets} that plays an important role in our study of controlled sweeping  processes. Given a nonempty set $\O\subset\R^n$, consider the associated {\em distance function}
\[
\dist(x;\O):=\inf_{y\in\O}\|x-y\|,\quad x\in\R^n,
\]
and the {\em Euclidean projection} of $x$ onto $\O$ given by
\[
\Pi(x;\O):=\big\{w\in\O\;\big|\;\|x-w\|=\dist(x;\O)\big\},\quad x\in\R^n.
\]
If the set $\O$ is closed around $\ox\in\O$, then the Euclidean projection $\Pi(x;\O)$ is nonempty for all $x$ near $\ox$. Using this, we define the {\em proximal normal cone} to $\O$ at $\ox\in\O$ by
\begin{equation*}
N^P_\O(\ox):=
\big\{v\in\R^n\ \big|\ \exists\,\al>0\ \text{ such that }\ \ox\in\Pi(\ox+\al v;\O)\big\}
\end{equation*}
with $N^P_\O(\ox):=\emptyset$ when $\ox\notin\O$, and then proceed with the definition of uniformly prox-regular sets.

\begin{definition}
\label{Def3}
{\rm\textbf{(uniform prox-regularity)}}
Given $\eta>0$, a closed set $\O\subset\R^n$ is said to be {\sc $\eta$-uniformly prox-regular} if for every $x\in\bd\O$ and every $v\in N^P_\O(x)$ with $\|v\|=1$, we have
\[
B(x+\eta v,\eta)\cap\O=\emptyset,
\]
{\rm which can be equivalently written as}
\[
\la v, y-x\ra\le \frac{\|v\|}{2\eta}\|y-x\|^2
\quad\text{for all }\;y\in\O,\;x\in\bd\,\O,\; v\in N^P_\O(x).
\]
\end{definition}

Uniform prox-regularity is a quantitative nonconvex substitute for convexity: an exterior ball of fixed radius rolls along the boundary of $\O$. Every closed convex set is uniformly prox-regular. Moreover, on a sufficiently small tubular neighborhood of a uniformly prox-regular set, the metric projection is single-valued. In the sweeping framework, the resulting projection stability and normal cone regularity are precisely what allow the nonconvex moving set $C_u(t)$ to be handled within the same variational calculus as the convex case.

The proximal normal cone reflects local metric projection geometry and is always convex while lacking robustness needed for variational calculus in the nonconvex setting. The following construction, known as the {\em limiting/Mordukhovich normal cone}, is always robust and enjoys comprehensive calculus rules  based on {\em variational/extremal principles} of variational analysis. We refer the reader to the books \cite{Mordukhovich2006I,Mordukhovich2026,Rockafellar1998} 
and the bibliographies therein for equivalent representations of $N_\O$, detailed studies and a variety of applications.

\begin{definition}
\label{DefLimNormal}
\textbf{\rm(\bf limiting normal cone)} Let $\O\subset\R^n$ be a nonempty set that is closed around $\ox\in\O$. The $($limiting$)$ {\sc normal cone} to $\O$ at $\ox$ is defined by
\begin{equation}\label{lnc}
N_\O(\ox):=
\big\{v\in\R^n\ \big|\ \exists\,x_k\to\ox,\ w_k\in\Pi(x_k;\O),\ \al_k\ge 0\ \text{ such that }\ \al_k(x_k-w_k)\to v\big\}.
\end{equation}
We put $N_\O(\ox):=\emptyset$ when $\ox\notin\O$.
\end{definition}

Due to the aforementioned advantages of the limiting normal cone, we use it as the basic construction in this paper. It stands as the normal cone in the description of the sweeping dynamics \eqref{e2}. Observe to this end that, in the case of uniformly prox-regular moving sets $C_u$ in \eqref{e2} (in particular, for convex sets), the limiting normal cone agrees with the proximal (and, respectively, classical convex) one.

Next we recall the subdifferential and coderivative constructions for functions and multifunctions/set-valued mappings associated with the limiting normal cone \eqref{lnc}. 

Given an extended-real-valued 
function $\phi:\R^n\to\Bar\R:=(-\infty,\infty]$ lower semicontinuous (l.s.c.) around $\ox\in\dom\phi:=\{x\in\R^n\mid \phi(x)<\infty\}$, the (limiting) {\em subdifferential} of $\phi$ at $\ox$ is defined by
\begin{equation}\label{sub}
\partial\phi(\ox):=\big\{v\in\R^n\ \big|\ (v,-1)\in N_{\epi\phi}\big(\ox,\phi(\ox)\big)\big\},
\end{equation}
where $\epi\phi:=\{(x,\al)\in\R^{n+1}\mid \al\ge \phi(x)\}$ is the epigraph of $\phi$.

Let $F:\R^n\tto\R^m$ be a set-valued mapping whose graph is closed around $(\ox,\oy)\in\gph F$, i.e., with $\oy\in F(\ox)$. The (limiting) {\em coderivative} of $F$ at $(\ox,\oy)$ is
\begin{equation}\label{cod}
D^\ast F(\ox,\oy)(u):=\big\{v\in\R^n\ \big|\ (v,-u)\in N_{\gph F}(\ox,\oy)\big\},\qquad u\in\R^m,
\end{equation}
where $\oy$ is skipped if ${F(\ox)=\{\oy\}}$ is a singleton. If $F:\R^n\to\R^m$ is single-valued and continuously differentiable around $\ox$, then $D^\ast F(\ox)(u)=\{\nabla F(\ox)^\ast u\}$ for all $u\in\R^m$, where $^*$ on the right-hand side stands for the matrix transposition/adjoint operator. The reader is referred to the books \cite{Mordukhovich2006I,Mordukhovich2026,Rockafellar1998} for various analytic representations, comprehensive calculus rules and numerous applications of the subdifferential and coderivative constructions in \eqref{sub} and \eqref{cod} to optimization, control, and other areas.

Yet another basic construction of variational analysis is broadly used in the paper. The {\em second-order subdifferential/generalized Hessian} of a function $\phi:\R^n\to\Bar\R$ at $\ox,\ov$ with $\ov\in\partial\phi(\ox)$ is defined as the coderivative \eqref{cod} of the first-order subdifferential mapping \eqref{sub} by
\begin{equation}\label{2nd}
\partial^2\phi(\ox,\ov)(u):=(D^\ast\partial\phi)(\ox,\ov)(u),\qquad u\in\R^n.
\end{equation}
If the function $\phi$ is twice continuously differentiable ($C^2$-smooth) around $\ox$, then the generalized Hessian in \eqref{2nd} reduces to the standard Hessian mapping
$$
\partial^2\phi(\ox,\ov)(u)=\big\{\nabla^2\phi(\ox)u\big\},\qquad u\in\R^n.
$$
If $\phi$ is of class $C^{1,1}$ around $\ox$, i.e., it is $C^1$-smooth with locally Lipschitzian gradients around $\ox$, then we have 
$$
\partial^2\phi(\ox)(u)=\partial\la u,\nabla\phi(\ox)\ra,\qquad u\in\R^n.
$$
In general, the second-order subdifferential \eqref{2nd} is a positively homogeneous set-valued mapping, which satisfies extended calculus rules and is expressed in terms of the given data for broad classes of functions appearing in variational analysis and its applications; see the book \cite{Mordukhovich2024} with the vast bibliographies therein.

Note that the second-order subdifferential \eqref{2nd} naturally arises in sweeping processes, which are described by the normal cone mapping, i.e., by the subdifferential mapping of the set {\em indicator function} $\delta_C$ equal to 0 on $C$ and $\infty$ otherwise. As we see below, the adjoint system for the sweeping dynamics in \eqref{e2} is expressed via the coderivative of $N_C$, i.e., via the second-order subdifferential of the indicator function. Explicit second-order subdifferential calculations for $\delta_C$ play a crucial role in our subsequent developments.

\section{Standing Assumptions}\label{S2a}

Here we formulate and discuss the standing assumptions on the given data of the control problem $(P)$. 

\medskip
\noindent\textbf{(H1)} The perturbation mapping $f\colon\R^n\times\R^d\to\R^n$ in \eqref{e2} is {\em globally Lipschitz continuous}, i.e., there exists a constant $L_f>0$ such that
\begin{equation}\label{e10}
\|f(x_1,a_1)-f(x_2,a_2)\|\le L_f\big(\|x_1-x_2\|+\|a_1-a_2\|\big)
\end{equation}
for all $(x_1,a_1),(x_2,a_2)\in\R^n\times\R^d$. Furthermore, there exists a constant $M>0$ such that the {\em growth condition}
\begin{equation}\label{e11}
\|f(x,a)\|\le M(1+\|x\|)\qquad\text{for all }x\in\R^n\ \text{ and all } a\in A
\end{equation}
is satisfied, where $A\subset\R^d$ is the control set taken from \eqref{e5}.

\medskip
\noindent\textbf{(H2)} The constraint functions $g_i\colon\R^n\to\R$, $i=1,\ldots,m$, defining the sets
\[
C=\bigcap_{i=1}^m C_i,\qquad C_i:=\big\{x\in\R^n\;\big|\;g_i(x)\ge 0\big\},
\]
satisfy the following conditions:

\smallskip
\noindent\hspace*{0.2in}\textbf{(H2.1)} Each $g_i$ is twice continuously differentiable ($C^2$-smooth) on $\R^n$.

\smallskip
\noindent\hspace*{0.2in}\textbf{(H2.2)} There exist constants $c,M_1,M_2,M_3>0$ and open sets $V_i\supset C_i$, $i=1,\ldots,m$, ensuring the estimates
\[
\dist\!\big(C_i,\,\R^n\setminus V_i\big)>c,
\]
\begin{equation}\label{e12}
M_1\le \|\nabla g_i(x)\|\le M_2,\qquad \|\nabla^2 g_i(x)\|\le M_3
\end{equation}
for all $x\in V_i$ and all $i=1,\ldots,m$, where $\dist(\O_1,\O_2):=\inf\{\|y_1-y_2\|\mid y_1\in\O_1,\,y_2\in\O_2\}$ denotes the usual distance between two sets $\O_1,\O_2\subset\R^n$.

\smallskip
\noindent\hspace*{0.2in}\textbf{(H2.3)} There are constants $\be>0$ and $\rho>0$ such that

\begin{equation}\label{e13}
\sum\limits_{i\in I_\rho(x)}\lm_i\|\nabla g_i(x)\|
\le
\be\Big\|\sum\limits_{i\in I_\rho(x)}\lm_i\nabla g_i(x)\Big\|
\end{equation}
for all $x\in C$ and all $\lm_i\ge 0$, where
\begin{equation}\label{e14}
I_\rho(x):=\big\{i\in\{1,\ldots,m\}\;\big|\;g_i(x)\le \rho\big\}.
\end{equation}

\medskip
\noindent\textbf{(H3)} The admissible control set $U\subset\R^n$ is described by continuously differentiable inequality constraints. Specifically, there exist $C^1$-smooth functions $\nu_i\colon\R^n\to\R$, $i=1,\ldots,s$, for which
\[
U:=\bigcap_{i=1}^s U_i,\qquad U_i:=\nu_i^{-1}\!\big((-\infty,0]\big)=\big\{u\in\R^n\;\big|\;\nu_i(u)\le 0\big\}.
\]

\smallskip
\noindent\hspace*{0.2in}\textbf{(H3.1)} There exists a constant $L_\nu>0$ such that
\begin{equation}\label{e15}
|\nu_i(u_2)-\nu_i(u_1)|\le L_\nu\|u_2-u_1\|
\end{equation}
for all $u_1,u_2\in\R^n$ and all $i=1,\ldots,s$. Moreover, $\max_{1\le i\le s}\|\nabla\nu_i(u)\|\le L_\nu$ for all $u\in\R^n$.

\smallskip
\noindent\hspace*{0.2in}\textbf{(H3.2)} If $u\in U$ and $v_i\in N_{U_i}(u)$, $i=1,\ldots,s$, satisfy $\sum\limits_{i=1}^s v_i=0$, then $v_i=0$ for all $i=1,\ldots,s$.

\smallskip
\noindent\hspace*{0.2in}\textbf{(H3.3)} There exists $\delta_\nu>0$ such that $\nabla\nu_i(u)\ne0$ whenever $|\nu_i(u)|\le\delta_\nu$, $i=1,\ldots,s$.

\medskip
\noindent\textbf{(H4)} The set $A$ from \eqref{e5} is compact in $\R^d$, and for each $x\in\R^n$ the set $f(x,A):=\{f(x,a)\mid a\in A\}$ is convex.

\medskip
\noindent\textbf{(H5)} The terminal cost $\ph\colon\R^n\times[0,\infty)\to\Bar\R$ is lower semicontinuous. The running cost $\ell\colon[0,\infty)\times\R^{4n+d}\to\Bar\R$ in \eqref{e1} --- defined on the half-line in $t$ since the horizon $T$ is itself a decision variable---is such that the section $\ell(t,\cdot)$ is lower semicontinuous and bounded from below on bounded sets for a.e.\ $t\ge 0$, and that the composition $t\mapsto\ell(t,x(t),u(t),a(t),\dot x(t),\dot u(t))$ is integrable on $[0,T]$ for every feasible quadruple $(x(\cdot),u(\cdot),a(\cdot),T)$. The endpoint constraint sets $\Xi_x\subset\R^n$ and $\Xi_T\subset[0,\infty)$ in \eqref{e2} are nonempty and closed. 

\medskip

Assumptions \textbf{(H1)}--\textbf{(H5)} ensure well-posedness of the sweeping system \eqref{e2} under admissible controls and guarantee the regularity and stability properties required for the discrete approximation scheme developed in Section~\ref{S3}. Their individual roles are as follows. Assumption \textbf{(H2)} is the principal geometric condition of the model. The $C^2$-smoothness in \textnormal{(H2.1)} and the uniform gradient and Hessian bounds in \textnormal{(H2.2)} together yield {\em uniform prox-regularity} for each $C_i$ and for their intersection $C$. The inverse-triangle inequality \textnormal{(H2.3)}, imposed on the enlarged near-active index set $I_\rho(x)$ in \eqref{e14}, is a robust {\em positive linear independence constraint qualification} (PLICQ) of the Mangasarian--Fromovitz type; it preserves prox-regularity of the intersection and the calculus of the normal cone mapping $x\mapsto N_C(x)$ used throughout.

Observe further that assumption \textbf{(H1)} governs the perturbation term $f$. The Lipschitz property \eqref{e10} yields the continuous dependence of trajectories on the state and control, while the linear growth bound \eqref{e11} provides uniform a priori estimates that preclude finite-time blow-up and supply the compactness needed for the approximating sequences. Assumption \textbf{(H3)} concerns the structure of the admissible set $U$ for the moving set controls $u(\cdot)$. The Lipschitz condition \eqref{e15} guarantees the required regularity of the generators $\nu_i$, whereas \textnormal{(H3.2)} is the parallel positive linear independence constraint qualification on the admissible control side ruling out {\em degeneracies} among the active control constraints. Together with the inverse triangle inequality \textnormal{(H2.3)} on the moving set side, the two PLICQ-type conditions make {\em normal cone calculus} available for $U$ and for the translated sets $C_u(t)=C+u(t)$. Note that \textnormal{(H3.2)} alone does not give a gradient form of the normals to each $U_i$. This is the role of \textnormal{(H3.3)}: it ensures that $N_{U_i}(u)=\{t\nabla\nu_i(u)\mid t\ge0\}$ at every boundary point of $U_i$. The same holds for the relaxed sets $U^i_k$ in \eqref{e:ua-dc} once $L_\nu\delta_k\le\delta_\nu$. This gradient representation is used for the control multipliers $\psi^u$ below.

Assumptions \textbf{(H4)} and \textbf{(H5)} are compactness--semicontinuity conditions: the compactness of the set $A$ in \eqref{e5} combined with the convexity of the image set $f(x,A)$ in \textbf{(H4)} enables the passage to the limit for the perturbation term, and the lower semicontinuity and summability of $\ell$ in \textbf{(H5)} yield the existence of optimal solutions and stability under convergence of feasible processes.

Under these hypotheses, the moving set $C_u(t)=C+u(t)$ is {\em uniformly prox-regular} for every $t\ge 0$; see \cite{AdlyNacryThibault2016,Venel2011}. {Indeed, the Hessian bound in \textnormal{(H2.2)} gives each set $C_i$ the exterior ball property, and \textnormal{(H2.3)} preserves it for the intersection $C$. No convexity of the functions $g_i$ is needed.} Therefore, the sweeping inclusion \eqref{e2} is well posed for every admissible control, and the proximal and limiting normal cones of $C_u(t)$ agree along feasible trajectories.

\section{Local Minimizers and Sweeping Relaxation}\label{S2b}

The main goal of this paper is to establish necessary optimality conditions for the sweeping control problem $(P)$. To proceed, let us specify the type of local minimizers for which we intend to derive such conditions.

Following the scheme originated for Lipschitzian differential inclusions in  \cite{Mordukhovich1995} and \cite[Chapter~6]{Mordukhovich2006I} and then extended to controlled sweeping processes in \cite{CaoMordukhovich2026,ColomboMordukhovichNguyen2020,Mordukhovich2026} among other publications, we define now the notion of intermediate local minimizers for the problem $(P)$ under consideration and its relaxed version.

A feasible solution $(\ox,\ou,\oa,\T)$ to $(P)$ is an {\em intermediate local minimizer} (i.l.m.) for this problem if there exists $\ve>0$ such that $J(\ox,\ou,\oa,\T)\le J(x,u,a,T)$ for all feasible quadruples $(x,u,a,T)$ in $(P)$ satisfying
\begin{equation}\label{local}
\int^{\T}_0\left(\left\| \dot x(t)-\dot\ox(t)\right\|^2+\left\| \dot u(t)-\dot\ou(t)\right\|^2 + \left\| a(t)-\oa(t)\right\|^2 \right)dt + |T-\T|^2< \ve.
\end{equation}
{Since $x(0)=\ox(0)$ and $u(0)=\ou(0)$, condition \eqref{local} also controls $\int_0^{\T}(\|x(t)-\ox(t)\|^2+\|u(t)-\ou(t)\|^2)\,dt$, which is at most $\T^2\ve$. Thus adding these terms to \eqref{local}, as in \cite{CaoMordukhovich2026}, gives the same class of i.l.m.\ with a different $\ve$.}
{If $T<\T$, then the functions $x,u,a$ in \eqref{local} are extended to $(T,\T]$ by $x(t):=x(T)$, $u(t):=u(T)$, and $a(t):=0$. Thus the velocities vanish there.}
If the $L^2$-norm of $\| \dot x(t)-\dot\ox(t)\|$ is replaced in \eqref{local} by the $C$-norm of $\|x(t)-\ox(t)\|$, we are talking about {\em strong local minimizers} for $(P)$. The latter notion clearly implies that $(\ox,\ou,\oa,\T)$ is an i.l.m.\ for $(P)$ but not vice versa.

In what follows, we need a certain {\em relaxation} of problem $(P)$ in the vein of Bogolyubov and Young in the calculus of variations. {Following \cite{CaoMordukhovich2026}, we restrict the running cost to the graph of the sweeping dynamics. Define $\ell_F(t,x,u,a,\dot x,\dot u):=\ell(t,x,u,a,\dot x,\dot u)$ if $-\dot x\in F(x,u,a)$ and $\ell_F(t,x,u,a,\dot x,\dot u):=\infty$ otherwise, where $F(x,u,a)=N_C(x-u)-f(x,a)$ is the mapping from \eqref{F-def} below. Let $\Hat\ell_F(t,x,u,a,\dot x,\dot u)$ denote the {\em convexification} of $\ell_F$ with respect to the {\em velocity} variables $(\dot x,\dot u)$, i.e., for each fixed $(t,x,u,a)$, the largest lower semicontinuous function of $(\dot x,\dot u)$ that is convex and majorized by $\ell_F(t,x,u,a,\cdot,\cdot)$.} Along with the original problem $(P)$ formulated in \eqref{e1}--\eqref{e5}, consider the
\textit{relaxed sweeping optimal control problem $(R)$} defined by:
\begin{equation}
\label{e7}
\text{minimize } \Hat J[x,u,a,T] := \ph(x(T),T) + \int^T_0{\Hat\ell_F}(t,x(t),u(t),a(t),\dot x(t),\dot u(t))\, dt
\end{equation}
over $(x(\cdot), u(\cdot), a(\cdot),T)\in  W^{1,2}([0,T];\R^n\times\R^n)\times L^2([0,T];\R^d)\times [0,\infty)$ satisfying
\begin{equation}
\label{e8}
\begin{cases}
\dot x(t) \in -N_{C_u(t)}(x(t)) + { f(x(t), A)}\; \text{ a.e. }\;t \in [0,T],\\
x(0)=x_0\in C_u(0),\\
(x(T), T)\in \Xi_x\times \Xi_T\subset \R^n\times[0,\infty), \\
\nu_i(u(t)) \leq 0\;\text{ for all }\;t\in [0,T] \text{ and } i=1,\ldots, s,\\
a(t) \in A, \text{ a.e. } t\in [0,T],
\end{cases}
\end{equation}
The main notion of local minimizers for $(P)$ of our study below is formulated as follows.

\begin{definition} {\rm\textbf{(relaxed intermediate local minimizers)}}
\label{Def2}
Let $(\ox,\ou,\oa,\T)$ be a feasible solution to $(P)$.
We say that it is a \textsc{relaxed intermediate local minimizer} (r.i.l.m.) for problem $(P)$ if there exists $\ve>0$ such that
 \[
 J[\ox,\ou,\oa,\T] = \Hat J[\ox,\ou,\oa,\T] \leq \Hat J[x,u,a,T]
 \]
whenever $(x,u,a,T)$ satisfies \eqref{e8} and the localization condition \eqref{local}.
\end{definition}

Observe that the cost \eqref{e7} is finite only if $-\dot x(t)\in F(x(t),u(t),a(t))$ for a.e.\ $t$. Note that $F(x,u,a)$ is convex, since the normal cone to the prox-regular set $C$ is convex. Thus the relaxation convexifies the running cost in the velocity variables, while the control $a$ stays coupled with the velocity through the original dynamics. In particular, the relaxed problem $(R)$ agrees with $(P)$, and relaxed minimizers are the same as the original ones, whenever $\ell$ is convex and lower semicontinuous in $(\dot x,\dot u)$. This is the case for both models in Section~\ref{S7}. Actually, this phenomenon has been observed for broad classes of nonconvex variational problems due to the ``hidden convexity" that is exhibited by differential systems on continuous-time intervals (nonatomic measures); see more discussions in \cite{ColomboMordukhovichNguyen2020,Mordukhovich2006I,Mordukhovich2026} and the references therein.

\section{Discrete Approximations in Controlled Sweeping Processes}\label{S3}
\setcounter{equation}{0}

This section summarizes the main results for discrete approximations of the sweeping control problem $(P)$ obtained in \cite{CaoMordukhovich2026}. These results play an important role in the subsequent limiting procedure to establish necessary optimality conditions for relaxed intermediate local minimizers of problem $(P)$.

Denote the set-valued mapping $F:\R^n\times\R^n\times\R^d\tto\R^n$ by
\begin{equation}\label{F-def}
F(x,u,a):=N_C(x-u)-f(x,a)
\end{equation}
and the index set of active constraints by
\begin{equation}\label{a-index}
I(y):=\big\{i\in\{1,\ldots,m\}\;\big|\;g_i(y)=0\big\}.
\end{equation}
Using Motzkin's theorem of the alternative, we get the representation
\begin{equation}\label{F-rep}
F(x,u,a)=\Big\{-\sum\limits_{i\in I(x-u)}\lm_i\nabla g_i(x-u)\;\Big|\;\lm_i\ge 0\Big\}-f(x,a).
\end{equation}
Then the original sweeping differential inclusion can be equivalently written as
\begin{equation}\label{F-dyn}
-\dot x(t)\in F\big(x(t),u(t),a(t)\big)\quad\text{for a.e. }\;t\in[0,T].
\end{equation}

Recall that the assumptions in \textbf{(H2)} ensure that the translated set $C_u(t):=C+u(t)$ is uniformly prox-regular. The following theorem contains well-posedness results, which are consequences of \cite[Theorem~1]{EdmondThibault2005}.

\begin{theorem}[well-posedness and estimate for sweeping trajectories]
\label{prop:wellposed_sweeping}
Consider the perturbed sweeping process \eqref{e2} controlled by the functions
\[
u(\cdot)\in W^{1,2}([0,T];\R^n),\qquad a(\cdot)\in L^2([0,T];\R^d)\ \text{ with }\ a(t)\in A\quad\text{for a.e. }\;t\in[0,T]
\]
under the assumptions in {\rm(\bf H1)}--{\rm(\bf H3)}, and let the constant $M>0$ be taken from \eqref{e11}. Then, for every initial condition $x_0\in C_u(0)$, the sweeping differential inclusion
\[
\dot x(t)\in -N_{C_u(t)}\big(x(t)\big)+f\big(x(t),a(t)\big)
\quad\text{for a.e. }\;t\in[0,T],\qquad x(0)=x_0,
\]
admits a unique solution $x(\cdot)\in W^{1,2}([0,T];\R^n)$. Moreover, we have the estimate
\begin{equation}\label{e31}
\|x(t)\|\le M_x := \|x_0\| + e^{2MT}\left(2MT(1+\|x_0\|) + \int_0^T \|\dot u(s)\|\,ds \right)
\quad \text{for all }\;t\in[0,T].
\end{equation}
\end{theorem}

Now we fix the (uniform) {\em discrete mesh}, which is used throughout the paper. For each $k\in\N$ and each horizon $T^k>0$, consider the partition 
\begin{equation}\label{DP}
\Delta_k(T^k):=\big\{0=t^k_0<t^k_1<\cdots<t^k_{k-1}<t^k_k=T^k\big\},
\qquad h_k:=t^k_{j+1}-t^k_j=\frac{T^k}{k},\ \ j=0,\ldots,k-1.
\end{equation}
The next strong convergence theorem is taken from \cite[Theorem~3.1]{CaoMordukhovich2026}.

\begin{theorem}[strong discrete approximation of feasible sweeping trajectories]\label{ThDA}
Under the assumptions in {\rm\textbf{(H1)}--\textbf{(H4)}}, let $(x(\cdot),u(\cdot),a(\cdot),\T)$ be a feasible solution to the original problem \eqref{e2}--\eqref{e5}. On the fixed mesh $\Delta_k(\T)$ from \eqref{DP},
write $t_j^k=j\T/k$ and $h_k=\T/k$. Suppose in addition that $\dot x(\cdot)$, $\dot u(\cdot)$, and $a(\cdot)$ are of bounded variation on $[0,\T]$, and that there exists a constant $\mu>0$ such that we have the estimates
\begin{equation}\label{e:var}
\begin{aligned}
\max\bigl\{\var(\dot x(\cdot);[0,\T]),\ \var(\dot u(\cdot);[0,\T]),\ \var(a(\cdot);[0,\T])\bigr\} &\le \mu,\\[0.3ex]
\max\!\left\{\left\|\dfrac{u(t_1^k)-u(0)}{h_k}\right\|,\ \left\|\dfrac{u(\T)-u(t_{k-1}^k)}{h_k}\right\|\right\} &\le \mu
\end{aligned}
\end{equation}
for all $k\in\N$ sufficiently large. Then there exist piecewise linear functions $(x^k(\cdot),u^k(\cdot))$ and a piecewise constant function $a^k(\cdot)$ on $[0,\T]$ enjoying the following properties.

\smallskip
\noindent
{\bf(i)} For every mesh point $t_j^k$, $j=0,\ldots,k$, we have the relations
\begin{equation}\label{e:ua-dc}
\begin{cases}
u^k(t_j^k)\in U_k:=\bigcap_{i=1}^s U_k^i
:=\bigcap_{i=1}^s \nu_i^{-1}\!\big((-\infty,\,L_\nu\delta_k]\big),\\[0.8ex]
\big(x^k(\T),\T\big)\in \big(\Xi_x+\mu_k^x \overline{B}(0,1)\big)\times \Xi_T
\end{cases}
\end{equation}
with $\delta_k\dn 0$ and $\mu_k^x\dn 0$ as $k\to\infty$. Moreover, it holds that $a^k(t_j^k)\in A$ for $j=0,\ldots,k-1$, while the terminal value of the piecewise constant representative may be chosen arbitrarily in $A$.

\smallskip
\noindent
{\bf(ii)} For every $j=0,\ldots,k-1$ and every $t\in[t_j^k,t_{j+1}^k]$, we have that
\begin{equation}\label{e:disc-dyn}
x^k(t)=x^k(t_j^k)+(t-t_j^k)\,v_j^k,\qquad
-v_j^k\in F\!\big(x^k(t_j^k),u^k(t_j^k),a^k(t_j^k)\big).
\end{equation}

\smallskip
\noindent
{\bf(iii)} The approximating sequences converge strongly on $[0,\T]$ in the sense that
\begin{equation}\label{e:strongconv2}
\begin{array}{ll}
(x^k(\cdot),u^k(\cdot))\to (x(\cdot),u(\cdot))
\quad\text{uniformly on }[0,\T],\\
(\dot x^k(\cdot),\dot u^k(\cdot),a^k(\cdot))\to (\dot x(\cdot),\dot u(\cdot),a(\cdot))
\quad\text{in }L^2([0,\T]).
\end{array}
\end{equation}

\smallskip
\noindent
{\bf(iv)} For all $k\in\N$, we get the estimates
\begin{equation*}
\var(\dot u^k(\cdot);[0,\T])\le \Tilde\mu,\quad\max\!\left\{\left\|\dfrac{u^k(t_1^k)-u^k(0)}{h_k}\right\|,\ \left\|\dfrac{u^k(\T)-u^k(t_{k-1}^k)}{h_k}\right\|\right\}\le \Tilde\mu,
\end{equation*}
where the positive constant $\Tilde\mu$ is calculated by
\begin{equation}\label{e34}
\Tilde\mu := \max\!\left\{
\begin{array}{l}
3\mu + 1 + 4L_f\T\mu\,e^{L_f\T}(2L_f+1) + 2L_f\mu,\\[0.4ex]
2e^{L_f\T}(2L_f+1)(\mu+1) + \mu
\end{array}\right\}.
\end{equation}

\smallskip
\noindent
{\bf(v)} We have the velocity estimates
\begin{equation}\label{e:rate-L2}
\|\dot x^k-\dot x\|_{L^2([0,\T];\R^n)}+\|\dot u^k-\dot u\|_{L^2([0,\T];\R^n)}\;=\;\mathcal O\!\left(h_k^{1/2}\right),\quad k\in\N.
\end{equation}
\end{theorem}

Next we construct a sequence of discrete approximations
of problem $(P)$ associated with its given r.i.l.m.\ $(\ox(\cdot),\ou(\cdot),\oa(\cdot),\T)$. For each $k\in\N$, let $T^k>0$ be a decision variable approximating $\T$ on the uniform mesh $\Delta_k(T^k)$ from \eqref{DP}. Fix $\ve>0$ as in Definition~\ref{Def2}{. Without loss of generality, $\ve\le1$. Choose a horizon radius $\rho\in(0,\min\{\ve/2,\T/2\})$ so small that
\[
\int_{\T-\rho}^{\T}\Bigl(\|\dot\ox(t)\|^2+\|\dot\ou(t)\|^2+\|\oa(t)\|^2\Bigr)dt<\frac{\ve}{4}.
\]
This is possible by the absolute continuity of the integral. Further,} let $\Tilde\mu>0$, $\delta_k\dn 0$, $\mu^{x}_k\dn 0$ be the constants taken from Theorem~\ref{ThDA}. {Here Theorem~\ref{ThDA} is applied to the reference process. Thus we assume throughout that $\dot\ox$, $\dot\ou$, and $\oa$ are of bounded variation on $[0,\T]$ and that \eqref{e:var} holds for $(\dot\ox,\dot\ou,\oa)$ and $\ou$ in place of $(\dot x,\dot u,a)$ and $u$ with some $\mu>0$. In particular, $\var(\oa;[0,\T])\le\mu$. When $T^k>\T$, the reference functions in \eqref{e:Pk-loc} and \eqref{e:discconv} are extended to $(\T,\T+\ve]$ by $\ox(t):=\ox(\T)$, $\ou(t):=\ou(\T)$, and $\oa(t)=\dot\ox(t)=\dot\ou(t):=0$. The extension of $\ox$ and $\ou$ agrees with \cite{CaoMordukhovich2026}, and the choice for $\oa$ is harmless since the extra interval shrinks to zero length.} Denote the discrete velocity variables by
\begin{equation}\label{e:disc-vel}
V^x_j:=\frac{x^k_{j+1}-x^k_j}{h_k},\qquad
V^u_j:=\frac{u^k_{j+1}-u^k_j}{h_k},\qquad j=0,\ldots,k-1,
\end{equation}
and collect the full discrete decision vector into
\begin{equation}\label{e:disc-var}
\sigma^k:=\bigl(x^k_0,\ldots,x^k_k,\,u^k_0,\ldots,u^k_k,\,a^k_0,\ldots,a^k_{k-1},\,T^k\bigr).
\end{equation}

For each $k\in\N$, the {\em discrete approximation problem}  $(P_k)$ of sweeping optimal control is defined by
\begin{equation}\label{e:Pk-cost}
\begin{aligned}
\mbox{minimize }\;J_k[\sigma^k]
:={}& \ph\bigl(x^k_{k},T^k\bigr)
+ h_k\sum\limits_{j=0}^{k-1}\ell\bigl(t^k_j,x^k_j,u^k_j,a^k_j,V^x_j,V^u_j\bigr)
+ \tfrac{1}{2}\bigl(T^k-\T\bigr)^2\\
& + \tfrac{1}{2}\sum\limits_{j=0}^{k-1}\int_{t^k_j}^{t^k_{j+1}}\!
\Bigl\|{\bigl(a^k_j,V^x_j,V^u_j\bigr)
-\bigl(\oa,\dot\ox,\dot\ou\bigr)\bigl(t\T/T^k\bigr)}\Bigr\|^2\,dt
+ \dist^2\!\Bigl(\textstyle\sum\limits_{j=1}^{k-1}\bigl\|V^u_j-V^u_{j-1}\bigr\|;\,(-\infty,\Tilde\mu\,]\Bigr)
\end{aligned}
\end{equation}
subject to the four families of constraints described as follows. The discrete sweeping inclusion
\begin{equation}\label{e:Pk-dyn}
x^k_{j+1}\in x^k_j - h_k\,F\bigl(x^k_j,u^k_j,a^k_j\bigr),\qquad j=0,\ldots,k-1,
\end{equation}
is the Euler-type discretization of \eqref{F-dyn}. The boundary and pointwise admissibility conditions are
\begin{equation}\label{e:Pk-bc}
\begin{cases}
\bigl(x^k_0,u^k_0\bigr)=\bigl(x_0,\ou(0)\bigr),\\[0.3ex]
x^k_k-u^k_k\in C,\\[0.3ex]
\bigl(x^k_{k},T^k\bigr)\in\bigl(\Xi_x+\mu^x_k \overline{B}(0,1)\bigr)\times\bigl(\Xi_T+\mu^x_k[-1,1]\bigr),\\[0.3ex]
{|T^k-\T|\le\rho},\\[0.3ex]
u^k_j\in U_k\;\mbox{ and}\;a^k_j\in A\;\mbox{ as }\;j=0,\ldots,k-1.
\end{cases}
\end{equation}
Fix the initial data, relax the endpoint constraint by the tolerance $\mu^x_k\dn 0$, localize the horizon around $\T$ from both sides, and enforce the relaxed control admissibility \eqref{e:ua-dc}. The $L^2$-localization constraints
\begin{equation}\label{e:Pk-loc}
\begin{cases}
\displaystyle\sum\limits_{j=0}^{k-1}\int_{t^k_j}^{t^k_{j+1}}\bigl\|(x^k_j,u^k_j,a^k_j)-(\ox(t),\ou(t),\oa(t))\bigr\|^2\,dt\le{\tfrac{\ve}{4}},\\[1.0ex]
\displaystyle\sum\limits_{j=0}^{k-1}\int_{t^k_j}^{t^k_{j+1}}\bigl\|(V^x_j,V^u_j)-(\dot\ox(t),\dot\ou(t))\bigr\|^2\,dt\le{\tfrac{\ve}{4}}
\end{cases}
\end{equation}
confine $\sigma^k$ to an $L^2$-neighborhood of the reference trajectory and its velocity in accordance with the r.i.l.m.\ property of $(\ox,\ou,\oa,\T)$. Finally, the discrete counterparts of the BV bound \eqref{e:var} defined by
\begin{equation}\label{e:Pk-bv}
\sum\limits_{j=1}^{k-1}\bigl\|V^u_j-V^u_{j-1}\bigr\|\le\Tilde\mu+1{,\qquad
\sum\limits_{j=1}^{k-1}\bigl\|a^k_j-a^k_{j-1}\bigr\|\le\Tilde\mu+1}
\end{equation}
supply the outer feasibility constraint, while $\Tilde\mu$ itself is the proximity target penalized by the distance square term in the cost function expression \eqref{e:Pk-cost}. {The last inequality in \eqref{e:Pk-bv} is a discrete counterpart of the bound $\var(\oa;[0,\T])\le\mu$ in \eqref{e:var}. It yields the strong compactness of the discrete perturbation controls, which plays the role of the $W^{1,2}$ bound on controls in \cite{CaoMordukhovich2019}; see Step~2 in the proof of Theorem~\ref{conver}. Note that no separate bound on $V^u_0$ is needed. Indeed, the second constraint in \eqref{e:Pk-loc} tells us that $\|V^u\|_{L^2(0,T^k)}\le\|\dot\ou\|_{L^2}+(\ve/4)^{1/2}$, where $\|\dot\ou\|\le2\mu$ by \eqref{e:var}. Since $\|V^u_j\|\le\|V^u_i\|+\sum_{l=1}^{k-1}\|V^u_l-V^u_{l-1}\|$ for all $i,j$, averaging over $i$ \eqref{e:Pk-bv} gives us
\begin{equation}\label{e:CV}\tag{5.19$'$}
\|V^u_j\|\le C_V:=(\T-\rho)^{-1/2}\bigl(2\mu(\T+\rho)^{1/2}+1\bigr)+\Tilde\mu+1\quad\mbox{for all }j=0,\ldots,k-1.
\end{equation}}

{Observe that the localization term in \eqref{e:Pk-cost} compares the discrete velocities with the reference velocities {\em rescaled} to the horizon $T^k$. With the substitution $s=t\T/T^k$, this term equals
\[
\frac{T^k}{\T}\cdot\frac12\sum_{j=0}^{k-1}\int_{j\T/k}^{(j+1)\T/k}\bigl\|(a^k_j,V^x_j,V^u_j)-(\oa,\dot\ox,\dot\ou)(s)\bigr\|^2\,ds,
\]
so it depends {\em linearly} on $T^k$. This is the only difference from the discrete problems in \cite{CaoMordukhovich2026}, and the two terms agree when $T^k=\T$. The rescaling removes the nodal evaluations of the reference velocities from the free-time condition; see Lemma~\ref{lem:free-time-rates} and Remark~\ref{rem:Riemann-integrable}. Theorem~\ref{ThDiscNOC} below remains valid for this modification with the same proof, since the cost depends smoothly on $T^k$. For Theorem~\ref{conver}, we give a proof adapted to $(P_k)$ below. It shows where the two-sided horizon constraint and the thresholds $\ve/4$ in \eqref{e:Pk-loc} are used. An upper bound on $T^k$ alone is not enough. Indeed, a discrete optimal horizon may stay away from $\T$ when the terminal cost decreases faster than $\tfrac12(T^k-\T)^2$ grows. By \eqref{e:discconv}, the constraint $|T^k-\T|\le\rho$ and the localization constraints \eqref{e:Pk-loc} are inactive for large $k$. Thus they add no multipliers to Theorem~\ref{ThDiscNOC}.}\vspace*{0.05in}

The next theorem on explicit coderivative/second-order calculations for sweeping dynamics is taken from \cite[Theorem~6.2]{CaoMordukhovich2026} and plays a crucial role in deriving necessary optimality conditions in both discrete-time and continuous-time controlled sweeping processes.

\begin{theorem}[coderivative calculations for sweeping dynamics]\label{Prop1}
Let $F(x,u,a):=N_C(x-u)-f(x,a)$, where $C$ is the nonconvex set defined in \eqref{e3}, and where $f$ is $C^1$-smooth with respect to $(x,a)$. Denote
\[
g(\cdot):=\bigl(g_1(\cdot),\ldots,g_m(\cdot)\bigr),\qquad
Q_\lm:=\sum\limits_{i=1}^{m}\lm_i\,\nabla^2 g_i(x-u)
\]
and fix a point $(x,u,a,w)\in\R^n\times\R^n\times\R^d\times\R^n$ such that $x-u\in C$ and $w+f(x,a)\in N_C(x-u)$. Let
\[
\Lambda(x,u,a,w):=\bigl\{\lm\in\R^m_+\bigm|\,-\nabla g(x-u)^*\lm=w+f(x,a),\ \lm_i g_i(x-u) = 0 \text{ for } i=1,\ldots,m\bigr\}
\]
be the set of KKT multipliers compatible with $w$. Assume further that:
\begin{itemize}
\item[\bf(a)] The gradients of the active constraints $\{\nabla g_i(x-u)\}_{i \in I(x-u)}$ are linearly independent.
\item[\bf(b)] For every $\lm\in\Lambda(x,u,a,w)$, the multifunction
\[
M:\R^{2m}\tto\R^n\times\R^m\quad\mbox{defined by }\;
M(\vTh):=\bigl\{(y,\lm)\bigm|\bigl(-g(y),\lm\bigr)+\vTh\in\gph N_{\R^m_-}\bigr\},
\]
is calm at the point $\vTh=0,\ y=x-u,\ \lm\in\Lambda(x,u,a,w)$ in the sense of {\rm\cite{Rockafellar1998}}.
\end{itemize}
Then for every $y\in\dom D^*N_C\bigl(x-u,\,w+f(x,a)\bigr)$, we have the coderivative formula
\begin{equation}\label{e:prop-coderiv}
D^*F(x,u,a,w)(y)=\bigcup_{\lm\in\Lambda(x,u,a,w)}\ \bigcup_{\gg\in\Gamma(x,u,\lm,y)}\bigl\{\bigl(\alpha_x(\lm,\gg,y),\,\alpha_u(\lm,\gg,y),\,\alpha_a(y)\bigr)\bigr\},
\end{equation}
where the ingredients are given by
\begin{equation*}\label{e:prop-components}
\begin{aligned}
\alpha_x(\lm,\gg,y):=-\nabla_x f(x,a)^*y-Q_\lm\,y-\nabla g(x-u)^*\gg,\\
\alpha_u(\lm,\gg,y):=Q_\lm\,y+\nabla g(x-u)^*\gg,\quad
\alpha_a(y):=-\nabla_a f(x,a)^*y,\\
\Gamma(x,u,\lm,y):=D^*N_{\R^m_-}\bigl(-g(x-u),\lm\bigr)\bigl(-\nabla g(x-u)y\bigr),
\end{aligned}
\end{equation*}
and  the coderivative term for $\Gamma(x,u,\lm,y)$ is explicitly calculated by
\begin{equation}\label{e:prop-DNR}
D^*N_{\R^m_-}(\bar x,\bar v)(\bar y)=
\begin{cases}
\varnothing & \text{if }\;\;\exists\,i\text{ with }\bar v_i\bar y_i\neq 0,\\[0.6ex]
\bigl\{\gg\in\R^m\bigm|\gg_i=0\ \forall i\in I_1(\bar y),\ \gg_i\ge 0\ \forall i\in I_2(\bar y)\bigr\} & \text{otherwise},
\end{cases}
\end{equation}
with the index sets expressed as
\begin{equation}\label{e17}
I_1(\bar y):=\{i\mid\bar x_i<0\}\cup\{i\mid\bar v_i=0,\,\bar y_i<0\},\qquad
I_2(\bar y):=\{i\mid\bar x_i=0,\,\bar v_i=0,\,\bar y_i>0\}.
\end{equation}
Moreover, the coderivative domain is evaluated by
\begin{equation*}
\begin{aligned}
\dom D^*N_C\!\bigl(x-u,\,w+f(x,a)\bigr)\subset\Bigl\{y\,\Bigm|\;
&\exists\,\lm\in\Lambda(x,u,a,w)\ \text{with}\\[-0.2ex]
&\lm_i\,\la\nabla g_i(x-u),y\ra=0,\ i=1,\ldots,m\Bigr\}.
\end{aligned}
\end{equation*}
\end{theorem}

Substituting $\bar x=-g(x-u)$, $\bar v=\lm$, and $\bar y=-\nabla g(x-u)y$ into \eqref{e:prop-DNR} and \eqref{e17} yields the following pointwise description of the set $\Gamma(x,u,\lm,y)$ in the coderivative formula 
\eqref{e:prop-coderiv}: the component $\gg_i$  of $\gg$ is admissible only when $\lm_i\,\la\nabla g_i(x-u),y\ra=0$, in which case we have
\begin{equation*}
\gg_i\in
\begin{cases}
\{0\} & \text{if }g_i(x-u)>0,\\[0.3ex]
\{0\} & \text{if }g_i(x-u)=0,\ \lm_i=0,\ \la\nabla g_i(x-u),y\ra>0,\\[0.3ex]
[\,0,\infty) & \text{if }g_i(x-u)=0,\ \lm_i=0,\ \la\nabla g_i(x-u),y\ra<0,\\[0.3ex]
\R & \text{if }g_i(x-u)=0\ \text{and}\ \bigl(\lm_i>0\ \text{or}\ \la\nabla g_i(x-u),y\ra=0\bigr).
\end{cases}
\end{equation*}

Considering the sequence of discrete approximation problems $(P_k)$, we now formulate two results concerning an appropriate convergence of discrete optimal solutions and necessary conditions for them. The first two-part statement is taken from \cite[Proposition~5.1 and Theorem~5.2]{CaoMordukhovich2026}, respectively.

\begin{theorem}[\bf existence and strong convergence of discrete optimal solutions]\label{conver} For each $k\in\N$, let $(P_k)$ be the discrete approximation problem constructed above, where $(\ox(\cdot),\ou(\cdot),\oa(\cdot),\T)$ is the designated r.i.l.m.\ for the original problem $(P)$ under the standing assumptions in ${\bf(H1)-(H5)}$. {Suppose, as in Theorem~{\rm\ref{ThDA}}, that $\dot\ox(\cdot)$, $\dot\ou(\cdot)$, and $\oa(\cdot)$ are of bounded variation on $[0,\T]$, and that there exists a constant $\mu>0$ such that the estimates \eqref{e:var} hold for $(\dot\ox,\dot\ou,\oa)$ and $\ou$ in place of $(\dot x,\dot u,a)$ and $u$ for all $k\in\N$ sufficiently large.} The following assertions hold:

{\bf(i)} For all $k\in\N$ sufficiently large, each problem $(P_k)$ admits an optimal solution.

{\bf(ii)} Let $\{(\ox^k,\ou^k,\oa^k,\T^k)\}$ be an arbitrary sequence of optimal solutions to discrete problems $(P_k)$ such that $(\ox^k(\cdot),\ou^k(\cdot))$ and $\oa^k(\cdot)$ are extended as piecewise linear and piecewise constant functions on $[0,\T^k]$, respectively. Suppose that the assumptions in ${\bf(H1)-(H5)}$ are satisfied locally around the optimal quadruples $(\ox^k,\ou^k,\oa^k,\T^k)$ and that the terminal cost $\ph$ is continuous around {$(\ox(\T),\T)$. Suppose also that the following conditions hold:
\begin{itemize}
\item[\rm(c1)] $\ell$ is lower semicontinuous jointly in all its arguments, and for every bounded set $K$, there are $c_K\in\R^n$ and $\beta_K\in\R$ such that $\ell(t,x,u,a,\dot x,\dot u)\ge\la c_K,\dot x\ra-\beta_K$ whenever $(t,x,u,a,\dot u)\in K$;
\item[\rm(c2)] $\ell$ is continuous at the reference points, and the sampled integrands of the approximations from Theorem~{\rm\ref{ThDA}} are dominated by a summable function.
\end{itemize}} Then  we have the strong convergence of extended discrete optimal solutions in the following sense: 
\begin{equation}\label{e:discconv}
\begin{aligned}
&\T^k\to\T\;\mbox{ as }\;k\to\infty,\\[0.2ex]
&\max_{t\in[0,\T^k]}\bigl\|(\ox^k(t),\ou^k(t))-(\ox(t),\ou(t))\bigr\|\to 0\;\mbox{ as }\;k\to\infty,\\[0.2ex]
&\int_0^{\T^k}\Bigl\|\bigl(\dot{\ox}^k(t),\dot{\ou}^k(t),\oa^k(t)\bigr)-\bigl(\dot\ox(t),\dot\ou(t),\oa(t)\bigr)\Bigr\|^2dt\to 0\;\mbox{ as}\;k\to\infty
\end{aligned}
\end{equation}
together with the upper limit estimate
\begin{equation}\label{e18}
\limsup_{k\to\infty}\sum\limits_{j=1}^{k-1}\left\|\dfrac{\ou^k_{j+1}-2\ou^k_j+\ou^k_{j-1}}{h_k}\right\|\le\Tilde\mu.
\end{equation}
\end{theorem}
{{\bf Proof}. The feasible set of $(P_k)$ is closed and bounded in a finite-dimensional space, and the cost \eqref{e:Pk-cost} is lower semicontinuous by \textbf{(H5)}. By Step~1 below, this set is nonempty for large $k$. This gives us {\bf(i)}. To prove {\bf(ii)}, fix an arbitrary sequence $z^k:=(\ox^k,\ou^k,\oa^k,\T^k)$ of optimal solutions to $(P_k)$. Split the cost as $J_k=F_k+Q_k$, where
\[
F_k[z^k]:=\ph\bigl(\ox^k_k,\T^k\bigr)+h_k\sum_{j=0}^{k-1}\ell\bigl(t^k_j,\ox^k_j,\ou^k_j,\oa^k_j,V^x_j,V^u_j\bigr)
\]
is the discrete counterpart of the original cost, and $Q_k\ge0$ collects the three penalty terms in \eqref{e:Pk-cost}. For convenience, we split the proof into the following four steps.

{\bf Step 1}. {\em Upper bound.} Let $(x^k,u^k,a^k)$ be the approximations of the reference process given by Theorem~\ref{ThDA} on the mesh $\Delta_k(\T)$, i.e., with $T^k=\T$. They satisfy \eqref{e:Pk-dyn} and \eqref{e:Pk-bc}. By \eqref{e:strongconv2}, they satisfy the localization constraints \eqref{e:Pk-loc} with the thresholds $\ve/4$ for large $k$. By Theorem~\ref{ThDA}{\bf(iv)}, they satisfy the first inequality in \eqref{e:Pk-bv}. They also satisfy the last one, since the construction in \cite{CaoMordukhovich2026} takes $a^k_j=\oa(\tau_j)$ with $\tau_j\in[t^k_j,t^k_{j+1})$, which gives us $\sum_j\|a^k_j-a^k_{j-1}\|\le\var(\oa;[0,\T])\le\mu$. Hence they satisfy \eqref{e:Pk-bv}, and the distance penalty vanishes on them. The horizon penalty also vanishes. Since $T^k=\T$, the rescaled localization term coincides with the original one, and it tends to zero by \eqref{e:strongconv2}. The continuity of $\ph$, condition (c2), and the dominated convergence theorem give us $F_k[x^k,u^k,a^k,\T]\to J[\ox,\ou,\oa,\T]$. The optimality of $z^k$ yields
\[
\limsup_{k\to\infty}J_k[z^k]\le\limsup_{k\to\infty}J_k[x^k,u^k,a^k,\T]=J[\ox,\ou,\oa,\T].
\]

{\bf Step 2}. {\em Compactness and locality.} Take any subsequence of $\{z^k\}$. Since $|\T^k-\T|\le\rho$, we may assume that $\T^k\to\Tilde T\in[\T-\rho,\T+\rho]$. Extend $z^k$ to $[0,\T+\rho]$ by the convention after \eqref{local}. The constraints \eqref{e:Pk-loc} and the compactness of $A$ give uniform $L^2$ bounds on $(\dot\ox^k,\dot\ou^k,\oa^k)$. Passing to a further subsequence, we get $(\dot\ox^k,\dot\ou^k,\oa^k)\to(\dot{\Tilde x},\dot{\Tilde u},\Tilde a)$ weakly in $L^2$ and $(\ox^k,\ou^k)\to(\Tilde x,\Tilde u)$ uniformly. The same holds for the piecewise constant interpolations of the nodal states. Moreover, the convergence of $\dot\ou^k$ is strong. Indeed, set $M:=\Tilde\mu+1$. By \eqref{e:Pk-bv} and \eqref{e:CV}, we have $\sum_{j=1}^{k-1}\|V^u_j-V^u_{j-1}\|\le M$ and $\|V^u_j\|\le C_V$ for all $j$. The extended velocity $\dot\ou^k$ is therefore bounded by $C_V$, and its total variation on $[0,\T+\rho]$ is at most $M+C_V$, including the jump at $\T^k$. By Helly's selection theorem, a subsequence converges pointwise. By the dominated convergence theorem, $\dot\ou^k\to\dot{\Tilde u}$ strongly in $L^2$. The convergence of $\oa^k$ is strong as well. Indeed, the last inequality in \eqref{e:Pk-bv} and the compactness of $A$ show that the extended controls $\oa^k$ are uniformly bounded and have uniformly bounded total variation, including the jump at $\T^k$. By Helly's selection theorem and the dominated convergence theorem, $\oa^k\to\Tilde a$ pointwise and strongly in $L^2$ along a subsequence. Since $A$ is closed, we have $\Tilde a(t)\in A$ for $t\in[0,\Tilde T)$. Since $f$ is continuous and the states converge uniformly, $f(\ox^k,\oa^k)\to f(\Tilde x,\Tilde a)$ strongly in $L^2$. Note that a weak limit of the controls would not suffice here, since it need not reproduce the limit drift when $f$ is nonlinear in $a$. Passing to the limit in the discrete inclusion \eqref{e:Pk-dyn} as in \cite{CaoMordukhovich2026}, by the closedness and hypomonotonicity of the normal cone graph of the prox-regular set $C$, we see that $(\Tilde x,\Tilde u,\Tilde a,\Tilde T)$ is feasible for $(P)$. Indeed, the endpoint and control constraints pass to the limit since $\mu^x_k\dn0$ and $\delta_k\dn0$. The weak lower semicontinuity of the norm and the thresholds $\ve/4$ in \eqref{e:Pk-loc} give us
\[
\int_0^{\Tilde T}\Bigl(\|\dot{\Tilde x}-\dot\ox\|^2+\|\dot{\Tilde u}-\dot\ou\|^2+\|\Tilde a-\oa\|^2\Bigr)dt\le\frac{\ve}{2}.
\]
If $\Tilde T<\T$, the convention after \eqref{local} adds at most the tail integral in the choice of $\rho$, which is less than $\ve/4$. Since $|\Tilde T-\T|^2\le\rho^2<\ve/4$, the left-hand side of \eqref{local} for $(\Tilde x,\Tilde u,\Tilde a,\Tilde T)$ is less than $\ve$.

{\bf Step 3}. {\em Lower bound and vanishing penalties.} By \eqref{e:Pk-dyn}, the discrete velocities satisfy $-V^x_j\in F(\ox^k_j,\ou^k_j,\oa^k_j)$. Hence the discrete running costs coincide with the values of $\ell_F$. Fix $R>C_V$ with $\|\dot\ou\|\le R$, and put $\ell_{F,R}:=\ell_F+\delta_{\overline{B}_R}(\dot u)$, where $\delta$ is the indicator function. Let $\Hat\ell_{F,R}$ denote the convexification of $\ell_{F,R}$ in the velocity variables $(\dot x,\dot u)$, defined as $\Hat\ell_F$. Since $\ell_F\le\ell_{F,R}$, we have $\Hat\ell_F\le\Hat\ell_{F,R}$. By \eqref{e:CV}, $\|V^u_j\|\le C_V<R$, so the discrete running costs are also values of $\ell_{F,R}$, and they are not less than the values of $\Hat\ell_{F,R}$. Let $K$ be a bounded set that contains the arguments $(t,x,u,a)$ of all the discrete and limit processes together with the ball $\|\dot u\|\le R$. By (c1), the affine function $\la c_K,\dot x\ra-\beta_K$ is majorized by $\ell_{F,R}$ for all velocities, since $\ell_{F,R}=\infty$ when $\|\dot u\|>R$. Hence it is also a minorant of $\Hat\ell_{F,R}$. Subtracting it, we obtain nonnegative integrands. The subtracted terms are $h_k\sum_j\bigl(\la c_K,V^x_j\ra-\beta_K\bigr)=\la c_K,\ox^k_k-x_0\ra-\beta_K\T^k$, and they converge. We now pass to the limit by using Young measures. Let $z^k$ collect the nonvelocity variables of the discrete running costs, i.e., the piecewise constant functions equal to $(t^k_j,\ox^k_j,\ou^k_j,\oa^k_j)$ on $[t^k_j,t^k_{j+1})$, and put $w^k:=(\dot\ox^k,\dot\ou^k)$. By Step~2, $z^k$ converges in measure to the function $\Tilde z$ that collects $(t,\Tilde x,\Tilde u,\Tilde a)$. Note that the time argument converges uniformly, since $|t-t^k_j|\le h_k$ on $[t^k_j,t^k_{j+1})$. The velocities $w^k$ converge only weakly in $L^2$ to $\Tilde w:=(\dot{\Tilde x},\dot{\Tilde u})$. Put $H:=\ell_{F,R}-\la c_K,\dot x\ra+\beta_K$. This function is nonnegative on the arguments considered. It is lower semicontinuous in all its variables, since $\ell$ is lower semicontinuous by (c1), the graph of $F$ is closed, and the ball $\overline B_R$ is closed. By the fundamental theorem for Young measures \cite{Ball1989}, a subsequence of $\{w^k\}$ generates a family of probability measures $\nu_t$, and the pairs $(z^k,w^k)$ generate the measures $\delta_{\Tilde z(t)}\otimes\nu_t$. Since $\{w^k\}$ is bounded in $L^2$, the barycenter of $\nu_t$ equals $\Tilde w(t)$ for a.e.\ $t$. Fix $b\in(0,\Tilde T)$. Then $b<\T^k$ for all large $k$. The nonnegativity and the lower semicontinuity of $H$ give us
\[
\liminf_{k\to\infty}\int_0^{\T^k}H\bigl(z^k(t),w^k(t)\bigr)\,dt\ge\int_0^b\int H\bigl(\Tilde z(t),w\bigr)\,d\nu_t(w)\,dt .
\]
For a.e.\ fixed $t$, Jensen's inequality for the convexification in the velocity variables yields
\[
\int H\bigl(\Tilde z(t),w\bigr)\,d\nu_t(w)\ge\Hat\ell_{F,R}\bigl(\Tilde z(t),\Tilde w(t)\bigr)-\la c_K,\dot{\Tilde x}(t)\ra+\beta_K .
\]
Let us now take $b\uparrow\Tilde T$ and add back the converging affine terms. The inequality $\Hat\ell_F\le\Hat\ell_{F,R}$ and the lower semicontinuity of $\ph$ give us $\liminf_kF_k[z^k]\ge\ph(\Tilde x(\Tilde T),\Tilde T)+\int_0^{\Tilde T}\Hat\ell_F\,dt=\Hat J[\Tilde x,\Tilde u,\Tilde a,\Tilde T]$. Observe that only the convexity and lower semicontinuity of $\Hat\ell_{F,R}$ in the velocity variables are used here. No convexity of $\ell$ itself is needed, since the relaxed cost already contains the convexification. By Step~2 and Definition~\ref{Def2}, we have $\Hat J[\Tilde x,\Tilde u,\Tilde a,\Tilde T]\ge J[\ox,\ou,\oa,\T]$. Combining this with the upper bound of Step~1 gives us
\[
\limsup_{k\to\infty}Q_k[z^k]\le\limsup_{k\to\infty}J_k[z^k]-\liminf_{k\to\infty}F_k[z^k]\le0.
\]
Thus $Q_k[z^k]\to0$ along the chosen subsequence. Since every subsequence also has a subsequence with this property, we get $Q_k[z^k]\to0$ for the whole sequence. In particular, it yields $\T^k\to\T$, and the rescaled localization term tends to zero as $k\to\infty$.

{\bf Step~4}. {\em Strong convergence and estimates.} By Step~3, the $L^2(0,\T^k)$-distance between $(\dot\ox^k,\dot\ou^k,\oa^k)$ and the rescaled reference $(\dot\ox,\dot\ou,\oa)(\cdot\,\T/\T^k)$ tends to zero. Since $\T^k\to\T$, the continuity of dilations in $L^2$ gives us the last convergence in \eqref{e:discconv}. Integrating the velocities from the common initial data $(x_0,\ou(0))$ gives us the uniform convergence of the states. Finally, the distance penalty tends to zero, which is equivalent to the upper limit estimate \eqref{e18}. Observe where the bounded variation assumption enters these estimates. By \eqref{e:var} and Theorem~\ref{ThDA}{\bf(iv)}, the comparison processes of Step~1 have total variation of the $u$-velocities at most $\Tilde\mu$. Thus the distance penalty vanishes on them, and the upper bound of Step~1 does not contain it. This is why Step~3 forces the penalty along $z^k$ to vanish in the limit.$\h$}

The next theorem follows from \cite[Theorem~7.2]{CaoMordukhovich2026} providing necessary conditions for optimal solutions to the discrete approximation problems $(P_k)$. Its notation is summarized in Table~\ref{tab:disc-notation}.

\begin{table}[ht]
\centering
\footnotesize
\renewcommand{\arraystretch}{1.25}
\begin{tabular}{@{}l@{\quad}l@{\quad}p{8.0cm}@{}}
\hline\noalign{\smallskip}
\textbf{Symbol} & \textbf{Range / defined at} & \textbf{Meaning} \\
\noalign{\smallskip}\hline\noalign{\smallskip}
$V^x_j,\,V^u_j$ & \eqref{e:disc-vel} & forward-difference velocities of $(x^k,u^k)$ at node $t^k_j$ \\
$\sigma^k$ & \eqref{e:disc-var} & full discrete decision vector of $(P_k)$ \\
$\th^{Xk}_j,\,\th^{Uk}_j,\,\th^{ak}_j$ & \eqref{d1} & averaged residuals on $[t^k_j,t^k_{j+1}]$: deviation of the discrete velocities and perturbation from the reference $(\dot\ox,\dot\ou,\oa)$ \\
$\lm^k$ & $\ge 0$ & cost normalization scalar (Lagrange multiplier of the running cost in $(P_k)$) \\
$\al^k$ & $\in\R^m_+$ & complementarity multiplier for the terminal moving-set constraint $x^k_k-u^k_k\in C$ \\
$p^k_j=(p^{xk}_j,p^{uk}_j)$ & $\in\R^n\times\R^n$ & discrete adjoint (costate) sequence dual to $(x^k,u^k)$ \\
$\psi^{uk}_j,\,\psi^{ak}_j$ & \eqref{e:psiu} and \eqref{e:psia} & multipliers for the control admissibility $u^k_j\in U_k$ and $a^k_j\in A$ \\
$\eta^k_j$ & $\in\R^m_+$ & KKT multipliers from $\Lambda(\ox^k_j,\ou^k_j,\oa^k_j,\,\cdot\,)$ representing the sweeping residual $V^x_j-f(\ox^k_j,\oa^k_j)$ as a nonnegative combination of active gradients, cf.\ \eqref{d4} \\
$\gg^k_j$ & $\in\R^m$ & coderivative multipliers from Theorem~\ref{Prop1}; sign pattern dictated by the active index set $I(\ox^k_j-\ou^k_j)$ \\
$w^k_j,\,v^k_j$ & \eqref{d2} & subgradient selections from $\partial\ell$ at the discrete reference points (the state/control-side and velocity-side components, respectively) \\
$\Lambda^k_j,\,\xi^k_j$ & \eqref{d9} & effective discrete costate entering the coderivative argument, and the Hessian coupling vector $\xi^k_j:=Q_{\eta^k_j}\Lm^k_j$ \\
$\mathcal H^k$ & \eqref{e:Hk} & averaged discrete Hamiltonian along the node sequence (cell-averaged form intrinsic to the BV/measure passage) \\
$\varrho^k$ & \eqref{e:varrho} & localization residual: minus the derivative of the localization penalty in the horizon $T^k$, i.e., minus the penalty divided by $T^k$; vanishes as $k\to\infty$ (Lemma~\ref{lem:free-time-rates}) \\
$\delta_k,\,\mu^x_k$ & $\dn 0$ & tolerances on the relaxed control admissibility and on endpoint localization in \eqref{e:Pk-bc} \\
$\Tilde\mu$ & \eqref{e34} & outer BV bound on the discrete velocity differences, penalized in \eqref{e:Pk-cost} \\
\noalign{\smallskip}\hline
\end{tabular}
\caption{Discrete-variable notation for the optimality system of $(P_k)$.}
\label{tab:disc-notation}
\end{table}

\begin{theorem}[necessary optimality conditions for discrete approximation problems]\label{ThDiscNOC}
Let $(\ox(\cdot),\ou(\cdot),\oa(\cdot),\T)$ be an r.i.l.m.\ for problem $(P)$  and let, for each large $k\in\N$, the quadruple $(\ox^k(\cdot),\ou^k(\cdot),\oa^k(\cdot),\T^k)$ be an optimal solution to $(P_k)$ extended on the entire interval $[0,\T_k]$ as in Theorem~{\rm\ref{conver}}. In addition to the assumptions of Theorem~{\rm\ref{ThDA}}, impose the qualification conditions of 
Theorem~{\rm\ref{Prop1}}, the $C^1$-smoothness of $f$ with respect to $(x,a)$, and the local Lipschitz continuity of terminal and running costs around the optimal points. Moreover, assume that the running cost $\ell$ does not depend on the time variable near the optimal points. This autonomy is not a mere simplification. Without it, differentiating the discrete cost with respect to the horizon $T^k$ produces an additional term with $\partial_t\ell$ in the free-time component of \eqref{d16}. Assume finally that the distance penalty term in \eqref{e:Pk-cost} vanishes at the optimal solution, and that the last constraint in \eqref{e:Pk-bv} is inactive there, i.e.,
\begin{equation}\label{e:BV-pen-zero}
\sum\limits_{j=1}^{k-1}\bigl\|V^u_j-V^u_{j-1}\bigr\|\le\Tilde\mu,\quad\sum\limits_{j=1}^{k-1}\bigl\|\oa^k_j-\oa^k_{j-1}\bigr\|<\Tilde\mu+1\quad\mbox{with}\quad V^u_j=\frac{\ou^k_{j+1}-\ou^k_j}{h_k}.
\end{equation} Introduce the averaged residuals
\begin{equation}\label{d1}
\begin{aligned}
\th^{Xk}_j&:=\int_{t^k_j}^{t^k_{j+1}}\!\left(\frac{\ox^k_{j+1}-\ox^k_j}{h_k}-{\dot\ox(t\T/\T^k)}\right)dt,\qquad
\th^{Uk}_j:=\int_{t^k_j}^{t^k_{j+1}}\!\left(\frac{\ou^k_{j+1}-\ou^k_j}{h_k}-{\dot\ou(t\T/\T^k)}\right)dt,\\
\th^{ak}_j&:=\int_{t^k_j}^{t^k_{j+1}}\!\bigl(\oa^k_j-{\oa(t\T/\T^k)}\bigr)dt
\end{aligned}
\end{equation}
for $j=0,\ldots,k-1$. Then there exist dual elements
\[
\lm^k\ge 0,\qquad
\al^k=(\al^k_1,\ldots,\al^k_m)\in\R^m_+,\qquad
p^k_j=(p^{xk}_j,p^{uk}_j)\in\R^n\times\R^n\ (j=0,\ldots,k),
\]
\[
\psi^{uk}_j=(\psi^{uk}_{j1},\ldots,\psi^{uk}_{js})\in\R^s_+\ (j=0,\ldots,k-1),\qquad
\psi^{ak}_j\in\R^d\ (j=0,\ldots,k-1),
\]
together with complementarity multipliers $\eta^k_j\in\R^m_+$, coderivative multipliers $\gg^k_j\in\R^m$ for $j=0,\ldots,k-1$, and subgradient selections for the running cost function
\begin{equation}\label{d2}
\bigl(w^{xk}_j,w^{uk}_j,w^{ak}_j,v^{xk}_j,v^{uk}_j\bigr)\in\partial_{x,u,a,X,U}\ell\!\left(t^k_j,\ox^k_j,\ou^k_j,\oa^k_j,\frac{\ox^k_{j+1}-\ox^k_j}{h_k},\frac{\ou^k_{j+1}-\ou^k_j}{h_k}\right),\quad j=0,\ldots,k-1,
\end{equation}
such that the following necessary optimality conditions are satisfied:

\smallskip
\noindent{\sc Nontriviality Condition}
\begin{equation}\label{d3}
\lm^k+\bigl\|\al^k\bigr\|+\sum\limits_{j=0}^{k}\bigl\|p^{xk}_j\bigr\|+\bigl\|p^{uk}_0\bigr\|+\sum\limits_{j=0}^{k-1}\bigl\|\psi^{uk}_j\bigr\|\neq 0.
\end{equation}

\smallskip
\noindent{\sc Primal--Dual Dynamic Relationships} 
\begin{subequations}\label{d:dyn}
\begin{align}
\label{d4}
&\frac{\ox^k_{j+1}-\ox^k_j}{h_k}-f\bigl(\ox^k_j,\oa^k_j\bigr)=\sum\limits_{i\in I(\ox^k_j-\ou^k_j)}\eta^k_{ji}\,\nabla g_i(\ox^k_j-\ou^k_j),\\[0.3ex]
\label{d5}
&\frac{p^{xk}_{j+1}-p^{xk}_j}{h_k}-\lm^k w^{xk}_j=-\nabla_x f(\ox^k_j,\oa^k_j)^*\Lm^k_j-\xi^k_j-\sum\limits_{i=1}^{m}\gg^k_{ji}\,\nabla g_i(\ox^k_j-\ou^k_j),\\[0.3ex]
\label{d6}
&\frac{p^{uk}_{j+1}-p^{uk}_j}{h_k}-\lm^k w^{uk}_j-\frac{1}{h_k}\sum\limits_{i=1}^{s}\psi^{uk}_{ji}\,\nabla\nu_i(\ou^k_j)=\xi^k_j+\sum\limits_{i=1}^{m}\gg^k_{ji}\,\nabla g_i(\ox^k_j-\ou^k_j),\\[0.3ex]
\label{d7}
&\lm^k w^{ak}_j+\frac{\lm^k}{h_k}\th^{ak}_j+\frac{\psi^{ak}_j}{h_k}=\nabla_a f(\ox^k_j,\oa^k_j)^*\Lm^k_j,\\[0.3ex]
\label{d8}
&p^{uk}_{j+1}=\lm^k\!\left(v^{uk}_j+\frac{\th^{Uk}_j}{h_k}\right)\quad\mbox{ for j=0,\ldots,k-1},
\end{align}
\end{subequations}
where the adjoint normal $\Lm^k_j$ and the 
Hessian coupling vector $\xi^k_j$ are defined by
\begin{equation}\label{d9}
\Lm^k_j:=p^{xk}_{j+1}-\lm^k\!\left(h_k^{-1}\th^{Xk}_j+v^{xk}_j\right),\qquad
\xi^k_j:=\!\left(\sum\limits_{i=1}^{m}\eta^k_{ji}\,\nabla^2 g_i(\ox^k_j-\ou^k_j)\right)\!\Lm^k_j,
\end{equation}
and  where the control multipliers $\psi^{uk}_j$ and $\psi^{ak}_j$ satisfy the adjoint inclusions
\begin{equation}\label{e:psiu}
\psi^{uk}_{ji}\in N_{(-\infty,L_\nu\delta_k]}\!\bigl(\nu_i(\ou^k_j)\bigr),\quad i=1,\ldots,s,\ j=0,\ldots,k-1,
\end{equation}
\begin{equation}\label{e:psia}
\psi^{ak}_j\in N_A(\oa^k_j),\quad j=0,\ldots,k-1.
\end{equation}

\smallskip
\noindent{\sc Transversality Conditions}
\begin{equation}\label{d16}
\begin{cases}
p^{uk}_k=-\disp\sum\limits_{i=1}^{m}\al^k_i\,\nabla g_i(\ox^k_k-\ou^k_k),\\[0.8ex]
\Bigl(-p^{xk}_k+\disp\sum\limits_{i=1}^{m}\al^k_i\,\nabla g_i(\ox^k_k-\ou^k_k),\ \mathcal H^k+\lm^k(\T-\T^k)+\lm^k\varrho^k\Bigr)\\[0.5ex]
\qquad\in\ \lm^k\,\partial\ph\bigl(\ox^k_k,\T^k\bigr)+N_{\Xi^k_x\times\Xi^k_T}\!\bigl(\ox^k_k,\T^k\bigr),
\end{cases}
\end{equation}
where $\Xi^k_x:=\Xi_x+\mu^x_k \overline{B}(0,1)$ and $\Xi^k_T:=\Xi_T+\mu^x_k[-1,1]$ are the relaxed endpoint sets of \eqref{e:Pk-bc}, and the averaged discrete Hamiltonian $\mathcal H^k$ and the auxiliary cost correction $\varrho^k$ are given by
\begin{equation}\label{e:Hk}
\mathcal H^k:=\frac{1}{k}\sum\limits_{j=0}^{k-1}\!\left[\la p^{xk}_{j+1},V^x_j\ra+\la p^{uk}_{j+1},V^u_j\ra
-\lm^k\,\ell\!\bigl(t^k_j,\ox^k_j,\ou^k_j,\oa^k_j,V^x_j,V^u_j\bigr)\right],
\end{equation}
\begin{equation}\label{e:varrho}
\varrho^k:=-\frac{1}{2\T^k}\sum\limits_{j=0}^{k-1}\int_{t^k_j}^{t^k_{j+1}}\Bigl\|(\oa^k_j,V^x_j,V^u_j)-(\oa,\dot\ox,\dot\ou)\bigl(t\T/\T^k\bigr)\Bigr\|^2\,dt
\end{equation}
with the vectors $V^x_j,V^u_j$ defined in \eqref{e:disc-vel}.

\smallskip
\noindent{\sc Complementary slackness conditions} meaning that for $j=0,\ldots,k-1$ and $i=1,\ldots,m$ we have
\begin{subequations}\label{d:cs}
\begin{align}
\label{d10}
g_i(\ox^k_j-\ou^k_j)>0&\ \Lto\ \eta^k_{ji}=0,\\[0.3ex]
\label{d11}
\bigl[i\in I_1(\beta^k_j)\bigr]&\ \Lto\ \gg^k_{ji}=0,\\[0.3ex]
\label{d12}
\bigl[i\in I_2(\beta^k_j)\bigr]&\ \Lto\ \gg^k_{ji}\ge 0
\end{align}
\end{subequations}
with $\beta^k_j:=-\nabla g(\ox^k_j-\ou^k_j)\Lm^k_j\in\R^m$ and the index sets $I_1(\cdot),I_2(\cdot)$ from \eqref{e17} evaluated at $\bar x=-g(\ox^k_j-\ou^k_j)$ and $\bar v=\eta^k_j$, together with the terminal and active  gradient implications
\begin{equation}\label{d14}
g_i(\ox^k_k-\ou^k_k)>0\ \Lto\ \al^k_i=0,\quad i=1,\ldots,m,
\end{equation}
\begin{equation}\label{d15}
\eta^k_{ji}>0\ \Lto\ \la\nabla g_i(\ox^k_j-\ou^k_j),\Lm^k_j\ra=0,\quad j=0,\ldots,k-1,\ i=1,\ldots,m.
\end{equation}

\smallskip
\noindent{\sc Enhanced nontriviality condition.} If the active constraint gradients $\{\nabla g_i(\ox^k_j-\ou^k_j)\}_{i\in I(\ox^k_j-\ou^k_j)}$ are linearly independent (equivalently, the active-constraint Jacobian has full row rank) for every $j=0,\ldots,k$, then
\begin{equation}\label{d17}
\lm^k+\sum\limits_{j=0}^{k-1}\bigl\|\psi^{uk}_j\bigr\|+\bigl\|p^{xk}_k\bigr\|+\bigl\|p^{uk}_0\bigr\|\neq 0.
\end{equation}
\end{theorem}\vspace*{-0.2in}

{\begin{remark}[on the BV penalty]\label{rem:BV-pen}
Assumption \eqref{e:BV-pen-zero} is technical. It concerns the optimal solutions of $(P_k)$, not the data of $(P)$. It removes the subgradient contribution of the distance penalty term in \eqref{e:Pk-cost} from relation \eqref{d8}. The estimate \eqref{e18} gives only an upper limit bound. It allows, for example, the sums in \eqref{e:BV-pen-zero} to equal $\Tilde\mu+1/k$. The first hard constraint in \eqref{e:Pk-bv} is inactive for all large $k$ by \eqref{e18}, since its upper bound is $\Tilde\mu+1$. The last one is inactive by the last condition in \eqref{e:BV-pen-zero}, so it adds no multiplier to Theorem~\ref{ThDiscNOC}. Note that the strong convergence of $\oa^k$ does not imply the convergence of their total variations, so this condition is not automatic. However, the penalty may remain positive. In this case, \eqref{d8} acquires the additional subgradient term
\[
\frac{2\lm^kd_k}{h_k}\bigl(n_j-n_{j+1}\bigr),\qquad d_k:=\max\Big\{\sum\limits_{i=1}^{k-1}\|V^u_i-V^u_{i-1}\|-\Tilde\mu,\,0\Big\},
\]
where $n_j\in\partial\|\cdot\|(V^u_j-V^u_{j-1})$ for $j=1,\ldots,k-1$, and $n_0=n_k:=0$. At a nonzero velocity difference, $n_j$ is its unit direction. At a zero difference, $n_j$ belongs to the closed unit ball. The factor $1/h_k$ prevents us from discarding this term merely because $d_k\to0$. In particular, the bounded variation estimate in Step~{\bf 4a} requires further control of this contribution. Assumption \eqref{e:BV-pen-zero} removes it. The condition on $u$ in \eqref{e:BV-pen-zero} holds automatically if the moving set control is fixed at zero in both $(P)$ and $(P_k)$, so that every $V^u_j$ is zero. This is the fixed moving set specialization used in both models of Section~\ref{S7}. A reference control $\ou\equiv0$ alone does not imply this property for the discrete optimal controls. Replacing the condition on $u$ in \eqref{e:BV-pen-zero} by a condition on the data of $(P)$ remains an open question. The difficulty is a mismatch of scales. The variation of the $u$-velocities is a quantity of order one, while the stationarity condition \eqref{d8} lives on the scale $h_k$. Hence any penalty or constraint on this variation enters \eqref{d8} with the factor $1/h_k$, and only its inactivity removes it. A bound of the form $d_k=O(h_k)$ would be enough, but Theorem~\ref{conver} gives no rate for $d_k\to0$.

The last condition in \eqref{e:BV-pen-zero} is not needed when the control set $A$ is convex. In this case, the multiplier rule gives $\zeta^k\ge0$ with $\zeta^k\bigl(\sum_{j=1}^{k-1}\|\oa^k_j-\oa^k_{j-1}\|-\Tilde\mu-1\bigr)=0$, and \eqref{d7} holds with $\psi^{ak}_j$ replaced by $\psi^{ak}_j+\zeta^k(m_j-m_{j+1})$. Here $m_j\in\partial\|\cdot\|(\oa^k_j-\oa^k_{j-1})$ for $j=1,\ldots,k-1$, and $m_0=m_k:=0$. Conditions \eqref{d3} and \eqref{d17} keep their form. Indeed, if all multipliers in \eqref{d3} vanish, then $\Lm^k=0$. Pairing \eqref{d7} with $\oa^k_j-c$ for a fixed $c\in A$ and summing over $j$, we get $\zeta^k\sum_{j=1}^{k-1}\|\oa^k_j-\oa^k_{j-1}\|\le0$ by the convexity of $A$. Thus $\zeta^k=0$ and then $\psi^{ak}=0$. Step~{\bf 5} of the proof of Theorem~\ref{ThNOC} shows how the term with $\zeta^k$ is handled in the limit. Both models of Section~\ref{S7} have $u\equiv0$ and a convex set $A$, so they need no part of \eqref{e:BV-pen-zero}.
\end{remark}}\vspace*{-0.2in}

\begin{remark}[on the sign of the effective costate]\label{rem:Lambda-sign}
In \cite[Theorem~7.2]{CaoMordukhovich2026}, the vector $\Lm^k_j$ in \eqref{d9} is written with the opposite sign. The sign used here is the correct one. Indeed, the discrete Euler--Lagrange inclusion in \cite{CaoMordukhovich2026} is a normal cone inclusion to $\gph F$ whose last component equals $\lm^k(h_k^{-1}\th^{Xk}_j+v^{xk}_j)-p^{xk}_{j+1}$. By the definition \eqref{cod} of the coderivative, the argument of $D^*F$ is the negative of this vector, i.e., the vector $\Lm^k_j$ from \eqref{d9}. A simple test confirms this choice. Let $n=d=1$, $f(x,a)=x+a$, $A=[-1,1]$, $\ell\equiv0$, $\varphi(x,T)=x$, and let $T$ be fixed. The optimal control is $\oa\equiv-1$. With the sign of \eqref{d9}, relations \eqref{d5} and \eqref{d7} give the classical adjoint equation and the inclusion $\psi^{ak}_j\in N_A(-1)$. With the opposite sign, they give the wrong growth of the adjoint and select the maximizer $\oa\equiv1$.
\end{remark}

{The degenerate multipliers of Example~\ref{ex:CQ0} below already appear at the discrete level. Their source is the first mesh cell. The proof of Theorem~\ref{ThDiscNOC} treats the fixed initial point $(x_0,\ou(0))$ as a variable inside the graph of $F$ at $j=0$, while a separate set fixes its value. The free normal of that set can cancel the coderivative term with an arbitrary $\gg^k_0$. The following lemma removes this source. We group the two sets before the multiplier rule is applied. Put $y_0:=x_0-\ou(0)$ and $I_0:=I(y_0)$.\vspace*{-0.1in}

\begin{lemma}[selection of multipliers at the initial cell]\label{lem:first-cell}
Let the assumptions of Theorem~{\rm\ref{ThDiscNOC}} hold, except that the last condition in \eqref{e:BV-pen-zero} is not required when $A$ is convex. Let the optimal solutions of $(P_k)$ belong to a sequence that converges as in Theorem~{\rm\ref{conver}}, and let the gradients $\{\nabla g_i(y_0)\}_{i\in I_0}$ be linearly independent. Then, for all large $k$, there exist multipliers satisfying all conclusions of Theorem~{\rm\ref{ThDiscNOC}}, with \eqref{d7} modified as in Remark~{\rm\ref{rem:BV-pen}} when $A$ is convex and the last constraint in \eqref{e:Pk-bv} is active. In addition, these multipliers satisfy:
\begin{itemize}
\item[\bf(a)] $\gg^k_0=0$ and $\psi^{uk}_0=0$.
\item[\bf(b)] $\la\nabla g_i(y_0),\Lm^k_0\ra\le0$ for all $i\in I_0$.
\item[\bf(c)] $p^{uk}_0=p^{uk}_1-h_k\lm^kw^{uk}_0-h_k\xi^k_0$.
\item[\bf(d)] $\lm^k+\|p^{xk}_k\|+\sum_{j=1}^{k-1}\|\psi^{uk}_j\|>0$.
\end{itemize}
\end{lemma}
{\bf Proof}. For convenience, we split the proof into the following four steps.

{\bf Step~1}. {\em Grouping the initial constraints.} The proof of \cite[Theorem~7.2]{CaoMordukhovich2026} writes $(P_k)$ as the minimization of a locally Lipschitz function over the intersection of finitely many closed sets. The variables are $(x^k_j,u^k_j,a^k_j,X_j,U_j,T^k)$, where $X_j$ and $U_j$ are the velocity variables. These sets are:
\begin{itemize}
\item the graphs of $F$ at the cells $j=0,\ldots,k-1$;
\item the control sets $u^k_j\in U_k$ and $a^k_j\in A$;
\item the evolution equations $x^k_{j+1}-x^k_j-h_kX_j=0$ and $u^k_{j+1}-u^k_j-h_kU_j=0$ with $h_k=T^k/k$;
\item the terminal moving set constraint and the relaxed endpoint set in \eqref{e:Pk-bc}, which also contain the initial condition $(x^k_0,u^k_0)=(x_0,\ou(0))$;
\item the localization constraints \eqref{e:Pk-loc}, the horizon constraint $|T^k-\T|\le\rho$, and the two constraints in \eqref{e:Pk-bv}.
\end{itemize}
For all large $k$, the localization constraints and the horizon constraint are inactive by \eqref{e:discconv}, since the optimal solutions converge as in Theorem~\ref{conver}. The first hard constraint in \eqref{e:Pk-bv} is inactive by \eqref{e18}. The second one is inactive by \eqref{e:BV-pen-zero}, or, when $A$ is convex, it keeps a scalar multiplier $\zeta^k\ge0$. The distance penalty in \eqref{e:Pk-cost} has zero subgradient by \eqref{e:BV-pen-zero}. Inactive constraints contribute no multipliers. The bound \eqref{e:CV} is an estimate while not being a constraint.

We replace the graph at $j=0$, the initial condition, and the redundant constraint $u^k_0\in U_k$ by the single closed set defined by
\[
\Omega_0:=\bigl\{(x^k_0,u^k_0,a^k_0,X_0)\bigm|(x^k_0,u^k_0)=(x_0,\ou(0)),\ X_0-f(x_0,a^k_0)\in K_0\bigr\},\qquad
K_0:=\Bigl\{\sum_{i\in I_0}\eta_i\nabla g_i(y_0)\Bigm|\eta_i\ge0\Bigr\}.
\]
Here $x_0$ and $\ou(0)$ are the given data, while $x^k_0$ and $u^k_0$ are variables. The constraint $u^k_0\in U_k$ is redundant, since $\ou(0)\in U\subset U_k$. The feasible set does not change. Indeed, the LICQ gives us $K_0=-N_C(y_0)$. Hence, at the fixed point $(x_0,\ou(0))$, the inclusion $-X_0\in F(x_0,\ou(0),a_0)$ is equivalent to $X_0-f(x_0,a_0)\in K_0$.

{\bf Step~2}. {\em The normal cone to $\Omega_0$.} The cone $K_0$ is finitely generated, and hence it is closed and convex. The mapping $(a,X)\mapsto(a,X-f(x_0,a))$ is a $C^1$ diffeomorphism with the $C^1$ inverse $(a,r)\mapsto(a,r+f(x_0,a))$. It maps the set $G:=\{(a,X)\mid X-f(x_0,a)\in K_0\}$ onto $\R^d\times K_0$. Thus \cite[Theorem~1.17]{Mordukhovich2006I} gives us the normal cone formula
\[
N_G(a,X)=\bigl\{\bigl(-\nabla_af(x_0,a)^*\omega,\ \omega\bigr)\bigm|\omega\in N_{K_0}\bigl(X-f(x_0,a)\bigr)\bigr\}.
\]
The $(x^k_0,u^k_0)$-components of the normals to $\Omega_0$ are arbitrary. No coderivative of the mapping $y\mapsto N_C(y)$ appears at the first cell since the point $y_0$ is fixed.

{\bf Step~3}. {\em The multiplier rule and the relations at $j=0$.} We apply the Fritz John rule for finitely many closed sets to the grouped family. There are two cases. If the normal qualification condition for the family fails, there exist normals to the sets, not all zero, whose sum vanishes. Then we put $\lm^k:=0$. If this condition holds, the intersection rule of \cite[Volume~II, Corollary~5.6]{Mordukhovich2006I} applies, and the stationarity condition for the cost holds with $\lm^k:=1$. In both cases, we obtain a new nonzero collection of multipliers. It is not obtained by changing a collection given by Theorem~\ref{ThDiscNOC}. The relations at the cells $j\ge1$, relation \eqref{d8} at $j=0$, the relation for $a_0$, and the endpoint and free time relations have the same form as in Theorem~\ref{ThDiscNOC}. The $X_0$-component gives us $\omega=h_k\Lm^k_0$. Since $\omega\in N_{K_0}(X_0-f(x_0,a_0))$ and $X_0-f(x_0,a_0)=\sum_{i\in I_0}\eta^k_{0i}\nabla g_i(y_0)$ by \eqref{d4}, we get
\[
\la\nabla g_i(y_0),\Lm^k_0\ra\le0\quad\mbox{and}\quad\eta^k_{0i}\la\nabla g_i(y_0),\Lm^k_0\ra=0\quad\mbox{for all}\ i\in I_0.
\]
This is {\bf(b)} and \eqref{d15} at $j=0$. Now we define $\gg^k_0:=0$, $\psi^{uk}_0:=0$,
\[
p^{xk}_0:=p^{xk}_1-h_k\lm^kw^{xk}_0+h_k\nabla_xf(x_0,a^k_0)^*\Lm^k_0+h_k\xi^k_0,
\]
and $p^{uk}_0$ by {\bf(c)}. Then \eqref{d5} and \eqref{d6} hold at $j=0$. Moreover, $\gg^k_0=0$ is admissible in \eqref{d11} and \eqref{d12}. Indeed, the vector $\beta^k_0:=-\nabla g(y_0)\Lm^k_0$ satisfies $\beta^k_{0i}\ge0$ and $\eta^k_{0i}\beta^k_{0i}=0$ for $i\in I_0$. For such arguments, zero belongs to the coderivative of the normal cone mapping $N_{\R^m_-}$ in \eqref{e:prop-coderiv}. The stationarity with respect to $(x^k_0,u^k_0)$ only determines the free components of the normal to $\Omega_0$.

{\bf Step~4}. {\em Nontriviality.} Suppose that $\lm^k=0$, $p^{xk}_k=0$, and $\psi^{uk}_j=0$ for $j\ge1$. By \eqref{d8}, we have $p^{uk}_j=0$ for $j=1,\ldots,k$. Then the first relation in \eqref{d16} gives us $\sum_i\al^k_i\nabla g_i=0$ with $\al^k_i\ge0$. Hence $\al^k=0$ by {\bf(H2.3)}. For $j\ge1$, relation \eqref{d6} gives us $\xi^k_j+\nabla g^*\gg^k_j=0$, and then \eqref{d5} yields $p^{xk}_j=(I+h_k\nabla_xf^*)p^{xk}_{j+1}$. Going backward from $p^{xk}_k=0$, we get $p^{xk}_j=0$ for all $j\ge1$. Thus $\Lm^k_j=0$ and $\xi^k_j=0$ for all $j$. The LICQ condition at the nodes from Theorem~\ref{Prop1} gives us $\gg^k_j=0$ for $j\ge1$. Hence all graph normals vanish. Moreover, $\omega=h_k\Lm^k_0=0$, and the free components of the normal to $\Omega_0$ vanish by the stationarity with respect to $(x^k_0,u^k_0)$. Since $\lm^k=0$ and $\Lm^k=0$, relation \eqref{d7} gives us $\psi^{ak}=0$ when the last constraint in \eqref{e:Pk-bv} is inactive. If $A$ is convex and this constraint is active, we get only $\psi^{ak}_j+\zeta^k(m_j-m_{j+1})=0$ with the vectors $m_j$ from Remark~\ref{rem:BV-pen}. We first remove $\zeta^k$. Fix $c\in A$. Pairing with $a^k_j-c$, summing over $j$, and using the convexity of $A$ and summation by parts, we get $\zeta^k\sum_{j=1}^{k-1}\|a^k_j-a^k_{j-1}\|\le0$. If $\zeta^k>0$, then the constraint is active, and this sum equals $\Tilde\mu+1>0$, a contradiction. Thus $\zeta^k=0$, and then $\psi^{ak}=0$. The endpoint normals vanish as well. The horizon component vanishes, since its cost and evolution terms are zero and the horizon constraint is inactive. Therefore, the whole collection is zero. This contradiction verifies {\bf(d)}. $\h$}

\section{Necessary Conditions for Sweeping Optimal Solutions}\label{S4}
\setcounter{equation}{0}

In this section, we formulate and partly discuss the main results of the paper, Theorem~\ref{ThNOC} and Proposition~\ref{prop:coderiv-selection}, providing novel necessary optimality conditions for the original free-time sweeping control problem $(P)$. Their proofs are given in the next Section~\ref{S6a}. Important consequences are given in Section~\ref{S6b}, and the endpoint nondegeneracy is studied in Section~\ref{S9b}.
\vspace*{-0.05in}

\begin{theorem}[necessary optimality conditions for controlled sweeping processes]\label{ThNOC}
Let $(\ox(\cdot),\ou(\cdot),\oa(\cdot),\T)$ be an r.i.l.m.\ for problem $(P)$ with $\T>0$ under the fulfillment of the assumptions in Theorem~{\rm\ref{ThDiscNOC}} {$($in particular, the running cost $\ell$ doesn't depend on $t$ near the optimal points, the assumptions of Theorem~{\rm\ref{conver}{\bf(ii)}} hold, and condition \eqref{e:BV-pen-zero} holds for all large $k$ along some sequence of optimal solutions to $(P_k)$, where its last condition is not needed if $A$ is convex$)$}. Suppose in addition that:

\begin{enumerate}
\item[\textnormal{\bf(a)}] \textsc{LICQ along the trajectory.} The active gradients $\{\nabla g_i(\ox(t)-\ou(t))\}_{i\in I(\ox(t)-\ou(t))}$ are linearly independent for every point $t\in[0,\T]$.
\item[\textnormal{\bf(b)}] \textsc{{BV velocities.}} The functions $\dot\ox(\cdot),\dot\ou(\cdot),\oa(\cdot)$ are of bounded variation on $[0,\T]$ and satisfy \eqref{e:var}. No further regularity of them is needed; see Remark~\ref{rem:Riemann-integrable}.
\item[\textnormal{\bf(c)}] \textsc{{Uniform} Lipschitz continuity of the running cost.} For some $\ve>0$ and every bounded set $K\subset\R^{3n+d}$, there is a constant $L_K>0$ such that $\ell$ is Lipschitz continuous with constant $L_K$ on the $\ve$-neighborhood of the set of points $(x,u,a,\dot x,\dot u)$ from the feasible velocity set
\begin{equation}\label{feas-vel}
\mathcal V:=\bigl\{(x,u,a,\dot x,\dot u)\in\R^{4n+d}\bigm|\,(-\dot x,u,a)\in F(x,u,a)\times(U+\varepsilon\B)\times(A+\varepsilon\B)\bigr\}
\end{equation}
with $(x,u,a,\dot u)\in K$. Here $\dot x$ is not restricted, i.e., the constant is uniform with respect to the velocity $\dot x$.
\item[\textnormal{\bf(d)}] \textsc{Separated structure of the running cost.} The function $\ell$ admits the decomposition
\begin{equation}\label{e:ell-rep}
\ell(t,x,u,a,\dot x,\dot u)=\ell_1(t,x,u,a,\dot x)+\ell_2(t,\dot u),
\end{equation}
where $\ell_2$ is $C^1$-smooth in $\dot u$ satisfying the estimates
\begin{equation}\label{l2-est}
\|\nabla_{\dot u}\ell_2(t,\dot u)\|\le L_{\ell_2}(1+\|\dot u\|),\qquad
\|\nabla_{\dot u}\ell_2(t,\dot u_1)-\nabla_{\dot u}\ell_2(s,\dot u_2)\|\le L_{\ell_2}\bigl(|t-s|+\|\dot u_1-\dot u_2\|\bigr)
\end{equation}
with some positive constant $L_{\ell_2}$.
\end{enumerate}
Then there exist a cost multiplier $\lm\ge 0$, a terminal multiplier $\alpha\in\R^m_+$, adjoint arcs
\[
p(\cdot)=\bigl(p^x(\cdot),p^u(\cdot)\bigr)\in W^{1,2}\bigl([0,\T];\R^{2n}\bigr),\qquad
q(\cdot)=\bigl(q^x(\cdot),q^u(\cdot)\bigr)\in BV\bigl([0,\T];\R^{2n}\bigr),
\]
a vector-valued Borel signed measure $\gamma\in C^*\bigl([0,\T];\R^m\bigr)$, a nonnegative vector-valued Borel measure $\psi^u\in C^*\bigl([0,\T];\R^s_+\bigr)$, a control multiplier $\psi^a\in L^2\bigl([0,\T];\R^d\bigr)$, and subgradient selections
\[
w(\cdot)\in L^2\bigl([0,\T];\R^{2n+d}\bigr),\qquad v(\cdot)\in L^2\bigl([0,\T];\R^{2n}\bigr)
\]
satisfying the convexified subdifferential inclusion
\begin{equation}\label{NOC:subg}
\bigl(w(t),v(t)\bigr)\in\co\,\partial\ell\bigl(t,\ox(t),\ou(t),\oa(t),\dot{\ox}(t),\dot{\ou}(t)\bigr)
\quad\text{for a.e.\ }t\in[0,\T],
\end{equation}
where $w=(w^x,w^u,w^a)$ and $v=(v^x,v^u)$, such that the following {\sc necessary optimality conditions} hold:
\begin{itemize}
\item[\bf(i)] \textsc{Primal representation and pointwise complementary slackness.} There exists a unique nonnegative multiplier $\eta\in L^2\bigl([0,\T];\R^m_+\bigr)$ such that
\begin{equation}\label{NOC:par}
\dot{\ox}(t)=f\bigl(\ox(t),\oa(t)\bigr)+\sum\limits_{i=1}^m\eta_i(t)\,\nabla g_i\bigl(\ox(t)-\ou(t)\bigr)\quad\text{for a.e.\ }t\in[0,\T],
\end{equation}
and we have for every $i=1,\ldots,m$ and a.e.\ $t\in[0,\T]$ the implications
\begin{equation}\label{NOC:cs}
g_i\bigl(\ox(t)-\ou(t)\bigr)>0\ \Lto\ \eta_i(t)=0,\quad
\eta_i(t)>0\ \Lto\ \la\nabla g_i(\ox(t)-\ou(t)),\,\lm v^x(t)-q^x(t)\ra=0.
\end{equation}
\item[\bf(ii)] \textsc{Adjoint systems.} The absolutely continuous component $p(\cdot)$ satisfies
\begin{equation}\label{NOC:adj-ac}
\dot p^x(t)=\lm w^x(t)-\nabla_x f\bigl(\ox(t),\oa(t)\bigr)^*\!{\bigl(q^x(t)-\lm v^x(t)\bigr)},\qquad
\dot p^u(t)=\lm w^u(t)
\end{equation}
for a.e.\ $t\in[0,\T]$, where $q(\cdot)$ is taken right-continuous on $(0,\T]$, $q(0)$ is the boundary value retained from the discrete approximation $($see Step~{\bf 4f}$)$, and where the BV component $q(\cdot)$ is determined pointwise on its second coordinate by
\begin{equation}\label{NOC:qu}
q^u(t)=\lm\,\nabla_{\dot u}\ell_2\bigl(t,\dot{\ou}(t)\bigr)\quad\text{for a.e.\ }\;t\in[0,\T].
\end{equation}
Writing the Hessian coupling density as
\[
\xi(t):=\Bigl(\sum_{i=1}^m\eta_i(t)\nabla^2 g_i(\ox(t)-\ou(t))\Bigr){\bigl(q^x(t)-\lm v^x(t)\bigr)},
\]
the two $u$-adjoints $(p^u,q^u)$ are coupled by the equality
\begin{equation}\label{NOC:puqu}
p^u(t)=q^u(t)+\int_t^{\T}\xi(s)\,ds+\int_{J_t}\nabla g\bigl(\ox(s)-\ou(s)\bigr)^*\,d\gamma(s)+\int_{J_t}\sum\limits_{j=1}^s\nabla\nu_j\bigl(\ou(s)\bigr)\,d\psi^u_j(s)
\end{equation}
for every $t\in[0,\T]$, while the first adjoint components $(p^x,q^x)$ are coupled by
\begin{equation}\label{NOC:qx}
\begin{aligned}
q^x(t)=&p^x(t)+\int_{J_t}\Bigl[\nabla g\bigl(\ox(s)-\ou(s)\bigr)^{\!*}d\gamma(s)\\[-0.2ex]
&+\Bigl(\sum_{i=1}^m\eta_i(s)\nabla^2 g_i(\ox(s)-\ou(s))\Bigr){\bigl(q^x(s)-\lm v^x(s)\bigr)}ds\Bigr]\quad\mbox{for every }\;t\in[0,\T].
\end{aligned}
\end{equation}
Here $J_t:=(t,\T]$ for $t\in(0,\T]$ and $J_0:=[0,\T]$. Thus an atom of $\gamma$ at $0$ contributes to the boundary value $q^x(0)$, while an atom at $\T$ appears as the jump $q^x(\T-)-p^x(\T)$, and $q^x(\T)=p^x(\T)$.
\item[\bf(iii)] \textsc{Adjoint inclusion for the perturbation control.} For a.e.\ $t\in[0,\T]$, we have
\begin{equation}\label{NOC:psi-a}
\psi^a(t):=-\lm w^a(t)+\nabla_a f\bigl(\ox(t),\oa(t)\bigr)^*\!{\bigl(q^x(t)-\lm v^x(t)\bigr)}\in\co\,N_A\bigl(\oa(t)\bigr).
\end{equation}
\item[\bf(iv)] \textsc{Transversality and free-time conditions.} The terminal multiplier $\alpha\in\R^m_+$ satisfies the relationships
\begin{equation}\label{NOC:alpha}
p^u(\T)=-\sum\limits_{i=1}^m\alpha_i\,\nabla g_i\bigl(\ox(\T)-\ou(\T)\bigr)\in N_C\bigl(\ox(\T)-\ou(\T)\bigr)
\end{equation}
together with the endpoint complementarity condition
$$
g_i(\ox(\T)-\ou(\T))>0\Lto\alpha_i=0\;\mbox{ for }\;i
=1,\ldots,m.
$$
Defining further the {\sc averaged Hamiltonian}
\begin{equation}\label{NOC:Hbar}
\mathcal H:=\frac{1}{\T}\int_0^{\T}\bigl[\la q^x(t),\dot\ox(t)\ra+\la q^u(t),\dot\ou(t)\ra-\lm\,\ell(t,\ox(t),\ou(t),\oa(t),\dot\ox(t),\dot\ou(t))\bigr]\,dt,
\end{equation}
we have the joint endpoint inclusion
\begin{equation}\label{NOC:trans}
\bigl(-p^x(\T)-p^u(\T),\ \mathcal H\bigr)\in\lm\,\partial\varphi\bigl(\ox(\T),\T\bigr)+N_{\Xi_x\times\Xi_T}\bigl(\ox(\T),\T\bigr).
\end{equation}
\item[\bf(v)] \textsc{Support of the sweeping and control constraint measures.} Each scalar component $\gamma^i$ is supported on the contact set $\{t\in[0,\T]:g_i(\ox(t)-\ou(t))=0\}$, and each $\psi^u_i$ is supported on $\{t\in[0,\T]:\nu_i(\ou(t))=0\}$.
\item[\bf(vi)] \textsc{Nontriviality conditions.} We always have the {\sc standard nontriviality condition}
\begin{equation}\label{NOC:NT}
\lm+\|q^u(0)\|+\|p(\T)\|+\|\psi^u\|_{TV}>0
\end{equation}
and the {\sc Arutyunov form of nontriviality} 
\begin{equation}\label{NOC:NT-meas}
\lm+\|p(\T)\|+\|\gamma\|_{TV}+\|\psi^u\|_{TV}>0.
\end{equation}
If in addition $g_i(x_0-\ou(0))>0$ for every $i=1,\ldots,m$, then the {\sc enhanced nontriviality}
\begin{equation}\label{NOC:NT-enh}
\lm+\|p(\T)\|+\|\psi^u\|_{TV}>0
\end{equation}
holds with dropping the measure term from \eqref{NOC:NT-meas}.
\end{itemize}
\end{theorem}
A complete proof of Theorem~\ref{ThNOC} will be given in Section~\ref{S6a}. In the remainder of this section, we discuss some aspects of the main result and present auxiliary statements used in the subsequent proof.

Observe first that the support conditions in {\bf(v)} admit a {\em pointwise} refinement coupling the density behavior of the sweeping measure $\gamma$ to the limiting coderivative of $N_{\R^m_-}$, i.e., Mordukhovich's second-order subdifferential from \eqref{2nd},  along the optimal arc. Unlike conclusions {\bf(i)--(vi)}, such a refinement is \emph{conditional}. Indeed, its proof requires a stability property of the discrete coderivative multipliers of Theorem~\ref{ThDiscNOC} that is not delivered by the discrete approximation itself, since the dual pairings below involve the velocity subgradients $v^{xk}$ converging only weakly. To formulate the required condition, we extend the discrete multipliers $\gg^k_j$, $\eta^k_j$, $\Lm^k_j$ of Theorem~\ref{ThDiscNOC} as piecewise constant functions $\gg^k(\cdot)$, $\eta^k(\cdot)$, $\Lm^k(\cdot)$ on the mesh cells $[t^k_j,t^k_{j+1})$ (the convention formalized in \eqref{e:ADE} below), let $d\gamma^k:=\gg^k(t)\,dt$ be the associated discrete sweeping measures from \eqref{e:DM}, and introduce the discrete and limiting dual pairings by
\begin{equation}\label{e:beta-def}
\beta^k_i(t):=-\bigl\langle\nabla g_i\bigl(\ox^k(\vartheta^k(t))-\ou^k(\vartheta^k(t))\bigr),\,\Lm^k(t)\bigr\rangle,\qquad
\beta_i(t):=\bigl\langle\nabla g_i\bigl(\ox(t)-\ou(t)\bigr),\,\lm v^x(t)-q^x(t)\bigr\rangle,
\end{equation}
for $i=1,\ldots,m$, where $\vartheta^k(t)$ denotes the left endpoint of the mesh cell containing $t$.

\begin{definition}[multiplier stability]\label{def:MS}
The discrete multipliers of Theorem~\ref{ThDiscNOC} are said to be {\sc stable along the optimal arc} provided that, along the subsequence fixed in the proof of Theorem~\ref{ThNOC}, we have the properties:
\begin{enumerate}\itemsep0.2ex
\item[\textnormal{\bf(MS1)}] \textsc{Sign stability:} for each $i=1,\ldots,m$, it holds that $\beta^k_i\to\beta_i$ as $k\to\infty$ in Lebesgue measure on the set $\{t\in[0,\T]:\beta_i(t)\neq 0\}$.
\item[\textnormal{\bf(MS2)}] \textsc{No concentration:} the densities $\{\gg^k(\cdot)\}_{k\in\N}$, extended by zero outside $[0,\T^k]$, are uniformly integrable on $[0,\T+\ve]$, which is equivalent to saying that the discrete sweeping measures $\{d\gamma^k\}$ are uniformly absolutely continuous with respect to the Lebesgue measure.
\end{enumerate}
\end{definition}

The introduced multiplier stability is used in the following proposition the verification of which is given in Step~8 of the proof of Theorem~\ref{ThNOC} in Section~\ref{S6a}. 

\begin{proposition}[conditional coderivative selection]\label{prop:coderiv-selection}
In addition to the assumption of Theorem~\ref{ThNOC}, suppose that the discrete multipliers are stable in the sense of Definition~\ref{def:MS}. Denote by
\begin{equation}\label{NOC:dens}
D_\gamma(t):=\Limsup_{\substack{S\ni t\\ |S|\downarrow 0}}\frac{\gamma(S)}{|S|},\qquad t\in[0,\T],
\end{equation}
the symmetric Radon--Nikodym-type derivative of $\gamma$ defined as the set of all $y\in\R^m$ for which there exists a sequence of intervals $S_n\ni t$ with $|S_n|\downarrow 0$ and $\gamma(S_n)/|S_n|\to y$. Then $\gamma$ is absolutely continuous with respect to the Lebesgue measure, $D_\gamma(t)=\bigl\{\tfrac{d\gamma}{dt}(t)\bigr\}$ for a.e.\ $t\in[0,\T]$, and the {\sc coderivative selection condition}
\begin{equation}\label{NOC:coderiv}
D_\gamma(t)\cap D^*N_{\R^m_-}\!\bigl(-g(\ox(t)-\ou(t)),\,\eta(t)\bigr)\bigl(\nabla g(\ox(t)-\ou(t)){(\lm v^x(t)-q^x(t))}\bigr)\neq\varnothing
\end{equation}
is satisfied for a.e.\ $t\in[0,\T]$.
\end{proposition}

\begin{remark}[sufficient conditions for multiplier stability]\label{rem:MS-suff}
Neither {\bf(MS1)} nor {\bf(MS2)} follows from the discrete approximation alone, but both are implied by the following verifiable structural conditions.

\smallskip
{\bf(a)} {\em Sign stability.} Along the subsequence in the proof, we have that $\nabla g_i(\ox^k(\vartheta^k(\cdot))-\ou^k(\vartheta^k(\cdot)))\to\nabla g_i(\ox-\ou)$ uniformly, $q^{xk}\to q^x$ pointwise on $[0,\T]$ by Helly's theorem, and $h_k^{-1}\th^{Xk}\to 0$ strongly in $L^2$. Thus the only obstruction to {\bf(MS1)} is the velocity subgradients $v^{xk}$, which in general converge only weakly in $L^2$. Consequently, {\bf(MS1)} holds whenever $\lm^kv^{xk}\to\lm v^x$ in measure. We consider the two verifiable cases:

\textnormal{\bf(a$_1$)} \emph{Abnormal regime $\lm=0$.} Lemma~\ref{lem:subgradient-bound} gives us $\sup_k\|v^{xk}\|_{L^\infty}<\infty$, and hence $\|\lm^kv^{xk}\|_{L^\infty}\to0$. 

\textnormal{\bf(a$_2$)} \emph{Strict differentiability in $\dot x$.} Suppose that $\ell$ is $C^1$-smooth in $\dot x$ near the reference data, with $\nabla_{\dot x}\ell$ jointly continuous in all the arguments there. Since $\ell$ is locally Lipschitzian, the convexified joint limiting subdifferential in \eqref{NOC:subg} reduces to the Clarke subdifferential. By the gradient formula \cite[Theorem~2.5.1]{Clarke1983}, the latter is the convex hull of all limits of gradients of $\ell$. The $\dot x$-components of these limits equal $\nabla_{\dot x}\ell$ due to the joint continuity of $\nabla_{\dot x}\ell$. Hence the projection of the Clarke subdifferential onto the $\dot x$-slot is the singleton $\{\nabla_{\dot x}\ell\}$. Hence $v^{xk}(t)=\nabla_{\dot x}\ell$ evaluated at the discrete data, and the joint continuity of $\nabla_{\dot x}\ell$ together with the a.e.\ convergence of the discrete arguments tells us that $v^{xk}\to v^x$ a.e.\ We stress that merely pointwise differentiability of $\ell$ in $\dot x$ would not suffice: it neither pins the $\dot x$-component of a \emph{joint} limiting subgradient nor stabilizes it under convergence of the arguments. A Mayer cost ($\ell\equiv0$){,} as in the polyhedral reduction of Remark~\ref{rem:polyhedral}{, and a running cost independent of $\dot x$, as in the minimum time problem of Section~\ref{sec:obstacle}, are special cases} of {\bf(a$_2$)}.

\smallskip
{\bf(b)} {\em No concentration.} {\bf(MS2)} holds under a uniform $L^p([0,\T+\ve];\R^m)$ bound on $\{\gg^k\}$ with $p>1$. In particular, any uniform $L^\infty$ bound on the discrete sweeping densities implies {\bf(MS2)}. We emphasize that the polyhedrality of $C$ alone yields \emph{neither} condition: it fixes the gradients $\nabla g_i$ but controls neither the velocity subgradients nor a possible concentration of the discrete densities. Multiplier stability is automatic only in the structured specializations considered in this paper --- Mayer or velocity differentiable running costs combined with bounded multiplier densities. Conversely, {\bf(MS2)} excludes impulsive regimes{, in which the sweeping measure $\gamma$ has atoms}.

\smallskip
{\bf(c)} {\em Role of the assumptions.} Without {\bf (MS1)} and {\bf(MS2)}, the cellwise sign patterns of $\gg^k$ may oscillate in $k$ on sets of positive measure while the discrete masses concentrate on the oscillation set. The cell averages of $\gg^k$ then need not respect the coderivative value at the limiting data, and the passage of the cellwise inclusion to \eqref{NOC:coderiv} breaks down. On the other hand, the support relationships of Theorem~\ref{ThNOC}(v) require no stability and constitute the unconditional content of the limiting second-order information.
\end{remark}\vspace*{-0.05in}

Now we turn to the other nonstandard object in Theorem~\ref{ThNOC}, which is the {\em averaged Hamiltonian} appearing in the free-time transversality conditions.

\begin{remark}[on the averaged Hamiltonian]\label{rem:Hbar}
The three features of the Hamiltonian \eqref{NOC:Hbar} merit comments.

\smallskip
{\bf(a)} {\em Time averaging is intrinsic to the BV passage.} The discrete approximation framework provides the averaged discrete Hamiltonian $\mathcal H^k$ defined by \eqref{e:Hk} as a discrete Ces\`aro average over mesh slots. Moreover, Step~{\bf 6b} in the proof of Theorem~\ref{ThNOC} establishes the convergence $\mathcal H^k\to\mathcal H$ as $k\to\infty$. A pointwise Hamiltonian identity along the optimal arc of the form available in the classical free-time PMP for smooth differential inclusions, would require pointwise convergence of the Hamiltonian density and a corresponding pointwise transversality slot. Both fail in the present BV/measure setting. The function $\la q^x,\dot{\ox}\ra+\la q^u,\dot{\ou}\ra$ is integrable on $[0,\T]$. However, the limiting procedure does not provide a pointwise Hamiltonian identity or maximization condition for it. The Ces\`aro time average is the level that is kept during the limiting procedure.

\smallskip
{\bf(b)} {\em Free-time content of} \eqref{NOC:trans}. The second component of \eqref{NOC:trans} reduces, after projecting onto the time coordinate, to the following inclusion
\[
\mathcal H\in\lm\,\partial_T\varphi(\ox(\T),\T)+N_{\Xi_T}(\T).
\]
{Here $\partial_T\varphi$ denotes the projection of $\partial\varphi$ onto the time coordinate. Heuristically, $\mathcal H$ plays the role classically taken by $-\partial_T V$, where $V(x_0,T)$ denotes the optimal value of $(P)$ as a function of the initial state $x_0$ and the terminal time $T$. We do not prove such a sensitivity relation here.} When $\Xi_T=\{\T\}$ (fixed terminal time), the condition $N_{\Xi_T}(\T)=\R$ leaves the $\mathcal H$-slot unconstrained, and the free-time identity agrees with the standard fixed-horizon transversality. Conversely, when $\T\in\inter\Xi_T$, the condition $N_{\Xi_T}(\T)=\{0\}$ and the identity reduces to the inclusion $\mathcal H\in\lm\,\partial_T\varphi(\ox(\T),\T)$. When $\varphi$ is smooth, this is the standard free-time Hamiltonian condition.

\smallskip
{\bf(c)} {\em Distinction from the measure modified Hamiltonian.} The averaged Hamiltonian $\mathcal H$ is a scalar associated with the entire optimal arc, coupling the free-time slot of \eqref{NOC:trans} and the terminal time sensitivity of the cost; it carries no pointwise structural information. The two objects play complementary roles in the optimality system. Namely, $\mathcal H$ governs the boundary condition that concerns the optimal time $\T$, while a pointwise maximum condition for the control pair $(\ou(t),\oa(t))$ requires a measure modified Hamiltonian, which is not developed in this paper.
\end{remark}

Let us summarize the new content of Theorem~\ref{ThNOC} and Proposition~\ref{prop:coderiv-selection}. It also gives new information for fixed-time controlled sweeping problems; cf.\ \cite{CaoMordukhovich2019,CaoColomboMordukhovich2021a,CaoColomboMordukhovich2021b,CaoMordukhovich2026,Mordukhovich2026} and the references therein.
\begin{itemize}\itemsep0.15ex
\item[\textnormal{\bf(N1)}] \emph{Free terminal time.} The transversality conclusion {\bf(iv)} carries the averaged Hamiltonian $\mathcal H$ in its second slot. This free-endpoint condition selects the optimal horizon $\T$. At the discrete level, it originates in the term $\lm^k(\T-\T^k)$ of the discrete transversality \cite[Theorem~7.2]{CaoMordukhovich2026}. This term is absent from the fixed-time discretization of \cite{CaoColomboMordukhovich2021a,CaoColomboMordukhovich2021b}.
\item[\textnormal{\bf(N2)}] \emph{Support and coderivative selection.} The support relationships in conclusion~{\bf(v)} are unconditional. The measure-valued coderivative selection of Proposition~\ref{prop:coderiv-selection} couples the sweeping multiplier $\gamma$ with the limiting coderivative of $N_{\R^m_-}$ along the optimal arc. It holds under the multiplier stability of Definition~\ref{def:MS}.
\item[\textnormal{\bf(N3)}] \emph{Nonconvex moving set and relaxation.} Conclusion~{\bf(v)} handles the nonconvex prox-regular set $C=\bigcap_{i=1}^m\{g_i\ge0\}$ under LICQ via the coderivative formula of Theorem~\ref{Prop1}. Its second-order term is absent from the convex polyhedral framework of \cite{ColomboMordukhovichNguyen2020,ColomboMordukhovichNguyen2024}; see Remark~\ref{rem:polyhedral}. Moreover, the r.i.l.m.\ notion of Definition~\ref{Def2} permits a running cost that is nonconvex in the velocity variables.
\item[\textnormal{\bf(N4)}] \emph{Nontriviality.} Besides the standard condition \eqref{NOC:NT}, conclusion~{\bf(vi)} gives the Arutyunov form \eqref{NOC:NT-meas}, which is unconditional, and the interior data form \eqref{NOC:NT-enh}. For the multipliers selected in Lemma~\ref{lem:first-cell}, Section~\ref{S9b} removes the interiority assumption and bounds the endpoint atoms of $\gamma$. To the best of our knowledge, no analogue of such nondegeneracy conditions is available in the literature on sweeping optimal control.
\end{itemize}
Conclusions {\bf(i)}--{\bf(iii)} and the standard nontriviality \eqref{NOC:NT} are structurally inherited from the fixed-time framework, being however carried through the variable-mesh passage that the given proof renders quantitatively.\vspace*{0.05in}

The following remark explains the mechanisms behind {\bf(N1)} and {\bf(N2)} in more detail.

\begin{remark}[mechanisms behind the new content]\label{rem:newcontent}
The two major ingredients of the optimality system in Theorem~\ref{ThNOC} are absent from the prior controlled sweeping analysis; namely, the {\em free-time Hamiltonian transversality} conclusion~{\bf(iv)} and the {\em support relationships} in conclusion~{\bf(v)} together with the {\em conditional coderivative selection} in Proposition~\ref{prop:coderiv-selection}. Neither of them appears in the fixed-horizon nonconvex framework \cite{CaoMordukhovich2019,CaoColomboMordukhovich2021a,CaoColomboMordukhovich2021b}, or in the convex polyhedral framework \cite{ColomboMordukhovichNguyen2020,ColomboMordukhovichNguyen2024}, or in the continuous penalization route \cite{NourZeidan2022,NourZeidan2023a,PinhoEtAl2023}. We comment on each of them in the following discussions.

\smallskip
{\bf(I)}{\em Averaged Hamiltonian relationships.} The second component of the discrete transversality condition \eqref{d16}, isolated at an optimal solution to $(P_k)$, can be rewritten via the {\em discrete averaged Hamiltonian inclusion}
\begin{equation}\label{e:disc-pontryagin-free-time}
\mathcal H^k\in\lm^k\,\partial_T\varphi\bigl(\ox^k_k,\T^k\bigr)+N_{\Xi^k_T}(\T^k)-\lm^k(\T-\T^k)-\lm^k\varrho^k
\end{equation}
with $\mathcal H^k$ taken from \eqref{e:Hk} and with $\partial_T\varphi$ signifying the projection of $\partial\varphi$ onto the time coordinate.

Observe that the correction terms $\lm^k(\T-\T^k)$ and $\lm^k\varrho^k$ in the averaged discrete Hamiltonian inclusion \eqref{e:disc-pontryagin-free-time} vanish in the limit by Theorem~\ref{conver} and Lemma~\ref{lem:free-time-rates}{\bf(i)}, respectively, with the explicit rate $|\varrho^k|=\mathcal O(h_k)$ supplied by Lemma~\ref{lem:free-time-rates}{\bf(ii)} under the residual rate assumption stated therein. In the limit, the discrete inclusion \eqref{e:disc-pontryagin-free-time} converges to the {\em continuous free-time averaged Hamiltonian condition}
\[
\mathcal H\in\lm\,\partial_T\varphi(\ox(\T),\T)+N_{\Xi_T}(\T),
\]
which is the second slot of the endpoint inclusion in \eqref{NOC:trans}.

\smallskip
{\bf(II)} {\em Support and coderivative selection conditions.} Conclusion {\bf(v)} of Theorem~\ref{ThNOC} --- the unconditional support relationships --- together with the conditional coderivative selection \eqref{NOC:coderiv} of Proposition~\ref{prop:coderiv-selection} is the {\em second-order} measure-valued reflection of the discrete coderivative inclusion. The following three structural observations make this claim to be more precise.

{\bf(II.a)} {\em Passage to the limit of measure multipliers.} Writing $z^k(t):=\ox^k(\vartheta^k(t))-\ou^k(\vartheta^k(t))$ and extending $\gg^k,\eta^k,\Lm^k$ as piecewise constant functions on $[t^k_j,t^k_{j+1})$, the discrete relationships in \eqref{d11} and \eqref{d12} together with the coderivative formula of Theorem~\ref{Prop1} yield the cellwise inclusions
\[
\gg^k(t)\in D^*N_{\R^m_-}\!\bigl(-g(z^k(t)),\,\eta^k(t)\bigr)\bigl(-\nabla g(z^k(t))\,\Lm^k(t)\bigr)\quad\text{for each }\;t\in(t^k_j,t^k_{j+1}),\ j=0,\ldots,k-1.
\]
As $k\to\infty$, the discrete multiplier $\gg^k$ accumulates only as a Borel measure $\gamma\in C^*([0,\T];\R^m)$ while not a BV function. The {\em symmetric outer derivative} $D_\gamma(t)$ defined in \eqref{NOC:dens} is the right object for the limit: it captures the cluster points of averaged densities of $\gamma$ on shrinking neighborhoods of $t$, and the intersection condition \eqref{NOC:coderiv} states that at least one such cluster point lies in the limiting coderivative output. This intersection condition is the exact measure-theoretic counterpart of the discrete cellwise inclusion displayed above being a precise mechanism by which the second-order discrete content survives the loss of regularity. The passage to the limit  from discrete approximations is not free. As Proposition~\ref{prop:coderiv-selection} shows, it requires the sign stability of the discrete pairings together with the no 
concentration of the discrete densities in Definition~\ref{def:MS}). The {\em intrinsic density level selection} in \eqref{NOC:coderiv} provides an essential refinement compared to the fixed-horizon setting (see, e.g., \cite[Theorem~4.1]{CaoColomboMordukhovich2021a}), which is achieved under the {\em multiplier stability} formulated in Definition~\ref{def:MS}.

{\bf(II.b)} {\em Active sets in the prox-regular setting}. In the convex polyhedral case \cite{ColomboMordukhovichNguyen2020,ColomboMordukhovichNguyen2024}, we have that $\nabla^2 g_i\equiv 0$, the Hessian coupling vector $\xi$ vanishes identically, the coderivative $D^*N_{\R^m_-}$ reduces to a sign-pattern set, and hence the coderivative selection \eqref{NOC:coderiv} reduces to a familiar componentwise complementarity condition already known in that setting. In the continuous penalization framework of \cite{NourZeidan2022,NourZeidan2023a,PinhoEtAl2023}, the normal cone $N_C$ is replaced by the smooth surrogate condition $\varepsilon^{-1}\nabla[\tfrac{1}{2}\dist^2(\,\cdot\,;C)]$ before passing to the corresponding PMP version. In this case, the limit as $\varepsilon\downarrow 0$ doesn't generate a Radon measure $\gamma$ concentrated on the active contact set, and no analog of the coderivative selection \eqref{NOC:coderiv} appears in the resulting optimality system. Condition \eqref{NOC:coderiv} is therefore an {\em active-set, second-order characteristics} that distinguishes the {\em prox-regular} framework associated with discrete approximations from both the polyhedral and the smooth-penalization alternatives.

{\bf(II.c)} {\em Bridge to the Hamiltonian formalism.} Under the assumptions of Proposition~\ref{prop:coderiv-selection}, the intersection set on the right-hand side of \eqref{NOC:coderiv} provides, via the standard measurable selection results (see, e.g., \cite[Chapter~14]{Rockafellar1998}) applied to the measurable set-valued mapping
\[
t\mapsto D_\gamma(t)\cap D^*N_{\R^m_-}\!\bigl(-g(\ox-\ou),\eta\bigr)\bigl(\nabla g(\ox-\ou){(\lm v^x-q^x)}\bigr),
\]
a measurable selection $\nu(t)\in\R^m$ satisfying the coderivative inclusion pointwise a.e.\ on $[0,\oT]$. The existence of such a selection is the structural prerequisite for any Hamiltonian form refinement of the optimality system that absorbs the sweeping geometry information into a single maximum condition. The modified Hamiltonian framework of \cite{HoangMordukhovich2019} produces one such refinement; more in this direction will be obtained in our future research.
\end{remark}

\section{Proofs of the Main Results}\label{S6a}

Although the proofs of the main necessary optimality conditions of Theorem~\ref{ThNOC} and Proposition~\ref{prop:coderiv-selection}, which are given in this section, use the discrete approximation strategy developed for fixed-horizon controlled sweeping processes, the passage to the limit here differs structurally from the fixed-time argument. The terminal time is an optimization variable, and consequently the discrete mesh $h_k=\T^k/k$, the endpoint localization in the product space for $(x,T)$, the adjoint normalization, and the transversality relationships all depend on the varying horizons $\T^k$. This dependence creates additional limiting terms absent from the fixed-horizon theory, most notably the horizon discrepancy $\lm^k(\T-\T^k)$, the localization residual $\lm^k\varrho^k$, and the averaged discrete Hamiltonian $\mathcal H^k$ whose convergence is established in Lemma~\ref{lem:free-time-rates} and in Step~{\bf 6} of the proof of Theorem~\ref{ThNOC}.\\[1ex]
Let us present three auxiliary results used in the proof. The first one bounds subgradients of the running cost.

\begin{lemma}[uniform bound for running cost subgradients]\label{lem:subgradient-bound}
Under the assumptions in {\bf(b)} and {\bf(c)} of Theorem~{\rm\ref{ThNOC}}, there exists a constant $L_\ell>0$ such that every selected subgradient $(w^k,v^k)$ along these discrete data satisfies the estimate
\begin{equation}\label{vel-est}
\|(w^k(t),v^k(t))\|\le L_\ell\;\mbox{ for a.e.\ }\;t\in[0,\T]\;\mbox{ and large }\;k.
\end{equation}
In particular, the sequence of pairs $(w^k,v^k)$ is bounded in $L^\infty$ and hence weakly compact in $L^2$. Moreover, if $\lm^k\to0$, then $\lm^k v^{xk}\to0$ as $k\to\infty$ in the norm topology of $L^\infty$.
\end{lemma}
{\bf Proof}.
{The discrete velocities $V^x_j$ need not be uniformly bounded, so we do not use a compactness argument for them. The uniform convergence in \eqref{e:discconv} keeps $(\ox^k,\ou^k)$ in a bounded set, and $\oa^k_j\in A$. Moreover, \eqref{e:CV} gives us $\|V^u_j\|\le C_V$. The reference data $(\ox,\ou,\oa,\dot\ou)$ are bounded by \eqref{e31} and assumption {\bf(b)}. Hence the components $(x,u,a,\dot u)$ of all the arguments lie in one bounded set $K$. The discrete arguments belong to $\mathcal V$, since $-V^x_j\in F(\ox^k_j,\ou^k_j,\oa^k_j)$ by \eqref{e:Pk-dyn} and $\ou^k_j\in U+\ve\B$ for large $k$ by \eqref{e:discconv}. Assumption {\bf(c)} provides a Lipschitz constant $L_K$ near the points in question, whatever the size of the $\dot x$-component. The limiting subgradients of a function that is Lipschitz with constant $L_K$ near a point are bounded by $L_K$, and this bound is preserved under the convexification. This verifies \eqref{vel-est} with $L_\ell:=L_K$.} Finally, the last two assertions immediately follow from the above discussions. $\h$
\vspace*{0.05in}

The next two lemmas are a discrete Gr\"onwall inequality, which yields the uniform bound on the discrete adjoint vectors $p^{xk}_j$, and a free-time lemma, which establishes the convergence $\varrho^k\to0$. The convergence $\mathcal H^k\to\mathcal H$ of the averaged Hamiltonians is established in Step~{\bf 6b} below.

\begin{lemma}[discrete Gr\"onwall inequality]\label{lem:gronwall}
Let $h>0$, $L>0$, $k\in\N$, and let $\{a_j\}_{j=0}^{k}\subset[0,\infty)$ and $\{B_j\}_{j=0}^{k-1}\subset[0,\infty)$ satisfy the inequalities
\begin{equation}\label{e:gronwall-hyp}
a_j\le(1+hL)\,a_{j+1}+B_j\qquad\text{for }\;j=0,\ldots,k-1.
\end{equation}
Then for every $j\in\{0,\ldots,k\}$, we have the estimates
\begin{equation}\label{e:gronwall-concl}
a_j\le e^{Lh(k-j)}\,a_k+e^{Lh(k-j)}\!\sum_{i=j}^{k-1}B_i.
\end{equation}
In particular, if $kh\le T$ for some $T>0$, $\sum\limits_{i=0}^{k-1}B_i\le M_1$, and $a_k\le 1$, then $\max_{0\le j\le k}a_j\le e^{LT}(1+M_1)$.
\end{lemma}
{\bf Proof}.
Backward induction on \eqref{e:gronwall-hyp} gives us
\[
a_j\le(1+hL)^{k-j}a_k+\sum_{i=j}^{k-1}(1+hL)^{i-j}B_i\qquad\text{for }j\in\{0,\ldots,k\}.
\]
Then the elementary bound $(1+hL)^{k-j}\le e^{Lh(k-j)}$ yields \eqref{e:gronwall-concl}. The final assertion immediately follows from the assumption $kh\le T$ and the assumed bounds on $a_k$ and $\sum B_i$. $\h$

\begin{lemma}[convergence of localization residuals]\label{lem:free-time-rates}
Let $(\ox(\cdot),\ou(\cdot),\oa(\cdot),\T)$ be an r.i.l.m.\ for the sweeping control problem $(P)$, and let $(\ox^k,\ou^k,\oa^k,\T^k)$ be optimal solutions to the discrete approximation problems $(P_k)$, $k\in\N$, under the assumptions of Theorem~{\rm\ref{conver}}. Then the localization residual \eqref{e:varrho} satisfies $\varrho^k=-\tfrac12R^k\le0$, where
\begin{equation}\label{e:rho-asymp}
R^k:=\frac{1}{\T^k}\int_0^{\T^k}\Bigl\|\bigl(\oa^k(t),\dot\ox^k(t),\dot\ou^k(t)\bigr)-(\oa,\dot\ox,\dot\ou)\bigl(t\T/\T^k\bigr)\Bigr\|^2\,dt.
\end{equation}
Moreover, the following assertions hold:
\begin{enumerate}
\item[\bf(i)] $\varrho^k\to 0$ as $k\to\infty$.
\item[\bf(ii)] If $R^k=\mathcal O(h_k)$, then $|\varrho^k|=\mathcal O(h_k)$ as $k\to\infty$.
\end{enumerate}
\end{lemma}
{\bf Proof}. The identity $\varrho^k=-\tfrac12R^k$ is a restatement of \eqref{e:varrho}, since the discrete velocities and controls are constant on each mesh cell. To prove {\bf(i)}, denote $r:=(\oa,\dot\ox,\dot\ou)$, extended by zero beyond $\T$, and $r^k(t):=r(t\T/\T^k)$. Then 
\[
\bigl(\T^kR^k\bigr)^{1/2}\le\bigl\|(\oa^k,\dot\ox^k,\dot\ou^k)-r\bigr\|_{L^2(0,\T^k)}+\bigl\|r-r^k\bigr\|_{L^2(0,\T^k)}.
\]
The first term tends to zero by the strong convergence \eqref{e:discconv} of Theorem~\ref{conver}. The second term tends to zero since $\T^k\to\T$ and dilations are continuous in $L^2$. Thus $R^k\to0$, and so $\varrho^k\to0$ as $k\to\infty$. Assertion {\bf(ii)} follows directly from $\varrho^k=-\tfrac12R^k$. $\h$

\begin{remark}[on the rescaled localization]\label{rem:Riemann-integrable}
The localization term in \eqref{e:Pk-cost} uses the reference velocities rescaled to the horizon $T^k$. This makes the term linear in $T^k$. Hence its derivative in the free-time condition \eqref{d16} is the penalty itself divided by $T^k$, which tends to zero by the strong convergence of the discrete optimal solutions. Without the rescaling, as in \cite{CaoMordukhovich2026}, the derivative contains nodal values of the reference velocities. These values need not converge when the reference velocities jump inside the mesh cells. Such jumps occur, e.g., in the corridor model of Section~\ref{sec:corridor} at the contact time. With the rescaling, no regularity of the reference velocities beyond the BV condition \eqref{e:var} is needed.
\end{remark}

Finally, we summarize in Table~\ref{tab:disc-to-cont} the discrete-to-continuous correspondence underlying the nine-step limiting procedure in the proof of Theorem~\ref{ThNOC}.
The table reads as a companion to Table~\ref{tab:disc-notation} (which records the meaning of each discrete object) and Remark~\ref{rem:newcontent} (which discusses the structural mechanisms of the passage to the limit from discrete approximations).

\begin{table}[ht]
\centering
\footnotesize
\renewcommand{\arraystretch}{1.30}
\begin{tabular}{@{}p{3.5cm}@{\quad}p{3.8cm}@{\quad}p{4.5cm}@{}}
\hline\noalign{\smallskip}
\textbf{Discrete object (from Theorem~\ref{ThDiscNOC}} & \textbf{Continuous limit (in Theorem~\ref{ThNOC})} 
& \textbf{Limit mechanism} \\
\noalign{\smallskip}\hline\noalign{\smallskip}
$V^x_j,\,V^u_j$ in \eqref{e:disc-vel} & $\dot\ox(t),\,\dot\ou(t)$ & strong $L^2$ convergence \eqref{e:strongconv2} \\
$\lm^k\in[0,1]$ & $\lm\in[0,1]$ & normalization \eqref{e:dnc} + finite-dimensional compactness \\
$\al^k\in\R^m_+$ & $\al\in\R^m_+$ & finite-dimensional convergence under LICQ \\
$p^k_j=(p^{xk}_j,p^{uk}_j)$ & $p(\cdot)=(p^x,p^u)\in W^{1,2}([0,\T];\R^{2n})$ & weak $L^2$-compactness on $\dot p^{xk},\dot p^{uk}$, Mazur, Helly \\
$\Lm^k_j$ in \eqref{d9} & $q^x(t)-\lm\,v^x(t)$ & combination of rows above + $L^2$ vanishing of $h_k^{-1}\theta^{Xk}$ \\
$\eta^k_j\in\R^m_+$ & $\eta(t)\in L^2([0,\T];\R^m_+)$ & LICQ + inverse-triangle \textbf{(H2.3)} \\
$\gg^k_j\in\R^m$ & sweeping measure $\gamma\in C^*([0,\T];\R^m)$ & Banach--Alaoglu weak$^*$ compactness \\
$\xi^k_j=Q_{\eta^k_j}\Lm^k_j$ in \eqref{d9} & 
Hessian coupling density $\xi\in L^1([0,\T];\R^n)$ (measure $\xi(t)\,dt$) & 
product of $\eta^k$ and $\Lm^k$ rows + weak$^*$ compactness \\
$\psi^{uk}_{ji},\,\psi^{ak}_j$ & $\psi^u\in C^*([0,\T];\R^s_+)$, $\psi^a\in L^2([0,\T];\R^d)$ & Banach--Alaoglu (BV side) / Mazur ($L^2$ side) \\
$(\mathcal H^k,\varrho^k,\T^k)$ & $(\mathcal H,\,0,\,\T)$ & Lemma~\ref{lem:free-time-rates}(i), Step~6b + Theorem~\ref{conver} \\
\noalign{\smallskip}\hline
\end{tabular}
\caption{Discrete-to-continuous dictionary for the limiting analysis. Each discrete object produced by Theorem~\ref{ThDiscNOC} is mapped to its continuous-time counterpart in Theorem~\ref{ThNOC}, together with the calculus tool driving the convergence. Used as a navigation aid throughout the nine step proof below.}
\label{tab:disc-to-cont}
\end{table}\vspace*{-0.12in}

{\bf Proof} (of Theorem~\ref{ThNOC}). Our approach extends the discrete approximation method developed in \cite{CaoMordukhovich2019} for the fixed-horizon controlled sweeping processes to the free-time setting. This requires nontrivial developments. The variable horizon $\T^k$ enters at every level of the result and its proof: the discrete mesh $\Delta_k(\T^k)$, the endpoint localization \eqref{e:Pk-loc}, the adjoint normalization \eqref{e:dnc}, and the additional free-time block $(\lm^k(\T-\T^k),\lm^k\varrho^k,\mathcal H^k)$ in the discrete transversality conditions \eqref{d16}. The proof below provides structural ingredients absent from the fixed-horizon proof. Steps~{\bf 1, 2, 3--4, 5, 7}, and {\bf 9a} inherit the template of the fixed-horizon arguments. At the same time, the major elements of the proof are specific for the free-time setting and are established here for the first time. Namely, Steps~{\bf 3--4} using the modified Gr\"onwall inequality of Lemma~\ref{lem:gronwall}), the transversality block (Step~{\bf 6}, via Lemma~\ref{lem:free-time-rates}), the conditional coderivative selection (Step~{\bf 8} proving Proposition~\ref{prop:coderiv-selection} under the multiplier stability of Definition~\ref{def:MS}), the Arutyunov form nontriviality (Step~{\bf 9b}, and the enhanced nontriviality under the interiority assumption (Step~{\bf 9c}). The nine steps structure the proof of the theorem are as follows:
\begin{gather*}
\text{Step~1}\to\eqref{NOC:subg},\quad
\text{Step~2}\to\text{(i)},\quad
\text{Steps~3, 4}\to\text{(ii)},\quad
\text{Step~5}\to\text{(iii)},\\[0.2ex]
\text{Step~6}\to\text{(iv)},\quad
\text{Step~7}\to\text{(v)},\quad
\text{Step~8}\to\text{Proposition~\ref{prop:coderiv-selection}},\quad
\text{Step~9}\to\text{(vi)}.
\end{gather*}

By Theorem~\ref{conver}, the given r.i.l.m.\ $(\ox(\cdot),\ou(\cdot),\oa(\cdot),\T)$ is strongly approximated in the $W^{1,2}([0,\T];\R^{2n})\times L^2([0,\T];\R^d)\times[0,\infty)$ topology by a sequence $\{(\ox^k(\cdot),\ou^k(\cdot),\oa^k(\cdot),\T^k)\}$ of extended optimal solutions to $(P_k)$. We take the sequence along which \eqref{e:BV-pen-zero} holds for all large $k$. This is possible because Theorem~\ref{conver}(ii) applies to an arbitrary sequence of optimal solutions. The piecewise linear and piecewise constant extensions of the discrete arcs are defined on $[0,\T]$ by
\[
\big(\ox^k(t),\ou^k(t),\oa^k(t)\big):=\big(\ox^k(\T^k),\ou^k(\T^k),\oa^k(\T^k)\big)\qquad\text{for }t\in(\T^k,\T]
\]
if $\T^k<\T$, and by the restriction to $[0,\T]$ if $\T^k\ge\T$, where the arcs are given on $[0,\T^k]\supset[0,\T]$ in the latter case. Theorem~\ref{conver} tells us that $\T^k\to\T$, and thus both cases occur only for finitely many indices outside any neighborhood of $\T$, and all the integrals below are taken over $[0,\min\{\T^k,\T\}]$. For each fixed $k\in\N$, Theorem~\ref{ThDiscNOC} furnishes the full discrete dual system \eqref{d1}--\eqref{d17}. We pass to the limit as $k\to\infty$ in the following nine steps.

\medskip
\noindent\textbf{Step~1}. \emph{Subgradient inclusion \eqref{NOC:subg}.}
For all $k\in\N$, extend the discrete subgradients $(w^k_j,v^k_j)$ from \eqref{d2} to piecewise constant functions on $[0,\T]$ by
$$
(w^k(t),v^k(t)):=(w^k_j,v^k_j)\;\mbox{ if }\;t\in[t^k_j,t^k_{j+1})\;\mbox{ and }\;(w^k(t),v^k(t)):=(w^k(\T^k),v^k(\T^k))\;\mbox{ if }\;(\T^k,\T].
$$
Lemma~\ref{lem:subgradient-bound} gives us the uniform subgradient bound and hence ensures the weak compactness of $(w^k,v^k)$ in $L^2([0,\T];\R^{2n+d}\times\R^{2n})$. Passing to a subsequence, we get $(w^k,v^k)\rightharpoonup(w,v)$ in $L^2$. By Mazur's theorem, a sequence of convex combinations of $(w^k,v^k)$ converges strongly in $L^2$ and hence pointwise a.e.\ (again along a subsequence) to $(w,v)$. Since $\ell$ is measurable in $t$ and, by \textbf{(H5)} and assumption~{\bf(c)}, lower semicontinuous and locally Lipschitz in the remaining variables, its limiting subdifferential $\partial\ell(t,\cdot)$ has closed graph for a.e.\ $t$. Combining the latter with the a.e.\ pointwise convergence of the arguments $(\ox^k,\ou^k,\oa^k,\dot\ox^k,\dot\ou^k)\to(\ox,\ou,\oa,\dot\ox,\dot\ou)$ and of the above subgradient selections readily verifies \eqref{NOC:subg}.

\medskip
\noindent\textbf{Step~2}. \emph{Primal equation \eqref{NOC:par} and first complementary slackness condition in \eqref{NOC:cs}.} Define 
\[
\th^{Xk}(t):=\frac{\th^{Xk}_j}{h_k},\qquad
\th^{Uk}(t):=\frac{\th^{Uk}_j}{h_k},\qquad
\th^{ak}(t):=\frac{\th^{ak}_j}{h_k}
\qquad\text{for }t\in[t^k_j,t^k_{j+1})
\]
with $(\th^{Xk}_j,\th^{Uk}_j,\th^{ak}_j)$ as in \eqref{d1}. Direct estimation using \eqref{d1} and the Cauchy--Schwarz inequality yields
\begin{equation}\label{e:theta-L2}
\begin{aligned}
\int_0^{\T^k}\|\th^{Xk}(t)\|^2\,dt
&=\sum_{j=0}^{k-1}\frac{\|\th^{Xk}_j\|^2}{h_k}\\
&\le\sum_{j=0}^{k-1}\int_{t^k_j}^{t^k_{j+1}}
\left\|\frac{\ox^k_{j+1}-\ox^k_j}{h_k}-{\dot\ox(t\T/\T^k)}\right\|^2dt
=\int_0^{\T^k}\|\dot\ox^k(t)-{\dot\ox(t\T/\T^k)}\|^2dt\to0.
\end{aligned}
\end{equation}
Similar estimates hold for $\th^{Uk}$ and $\th^{ak}$. Hence the strong convergence in \eqref{e:discconv}, together with the continuity of dilations in $L^2$ as in the proof of Lemma~\ref{lem:free-time-rates}, ensures that
\begin{equation}\label{e:theta-zero}
\int_0^{\T^k}\|\th^{Xk}(t)\|^2\,dt,\ \int_0^{\T^k}\|\th^{Uk}(t)\|^2\,dt,\ \int_0^{\T^k}\|\th^{ak}(t)\|^2\,dt\;\to\;0,
\end{equation}
and the corresponding subsequences converge pointwise a.e.\ to zero.

Introduce further the piecewise constant multiplier extensions
\[
\eta^k(t):=\eta^k_j\ \text{for }t\in[t^k_j,t^k_{j+1}),\qquad\eta^k(\T^k):=\eta^k_{k-1},
\]
where the multipliers $\eta^k_j$ are taken from  Theorem~\ref{ThDiscNOC}.
The discrete primal relationship in \eqref{d4} reads as
\begin{equation}\label{e:par-dis}
\dot\ox^k(t)-f\bigl(\ox^k(\vartheta^k(t)),\oa^k(t)\bigr)=\sum\limits_{i=1}^m\eta_i^k(t)\,\nabla g_i\bigl(\ox^k(\vartheta^k(t))-\ou^k(\vartheta^k(t))\bigr)\;\mbox{ for every }\;t\in[t^k_j,t^k_{j+1}).
\end{equation}
Since the velocity $\dot\ox^k(t)$ is constant on each slot, \eqref{e:par-dis} is exactly the nodal identity \eqref{d4}.
Applying \textbf{(H2.2)} and \textbf{(H2.3)} to the right-hand side of \eqref{e:par-dis} leads us to the inequalities
\[
\|\dot\ox^k(t)\|+\|f(\ox^k(\vartheta^k(t)),\oa^k(t))\|
\ge\Bigl\|\sum\limits_{i=1}^m\eta_i^k(t)\,\nabla g_i(\ox^k(\vartheta^k(t))-\ou^k(\vartheta^k(t)))\Bigr\|
\ge\frac{M_1}{\beta}\sum\limits_{i=1}^m\eta_i^k(t).
\]
Combining the latter with the linear growth bound \eqref{e11} in \textbf{(H1)} yields
\begin{equation}\label{e:eta-bound}
\|\dot\ox^k(t)\|+M\bigl(1+\|\ox^k(\vartheta^k(t))\|\bigr)\ge\frac{M_1}{\beta}\sum\limits_{i=1}^m\eta_i^k(t).
\end{equation}
The $W^{1,2}$-strong convergence of $\ox^k$ to $\ox$ from {Theorem~\ref{conver}} gives us an $L^2$-bound of $\dot\ox^k$ uniformly in $k$, which being combined with \eqref{e:eta-bound} and \textbf{(H1)} establishes the uniform $L^2$-estimate
\begin{equation}\label{e:eta-L2bd}
\|\eta^k\|_{L^2([0,\T];\R^m_+)}\le C_\eta,\quad\mbox{where }\;C_\eta:=\tfrac{\beta}{M_1}\,\sup_{k\in\N}\bigl(\|\dot\ox^k\|_{L^2}+M\sqrt{\T}+M\|\ox^k\|_{L^2([0,\T])}\bigr)<\infty.
\end{equation}
Observe that the supremum in  \eqref{e:eta-L2bd} is finite due to the $W^{1,2}$-convergence $\ox^k\to\ox$ in {Theorem~\ref{conver}}. The weak compactness of bounded sets in $L^2$ ensures the existence of a cluster point 
$\eta\in L^2([0,\T];\R^m_+)$ of $\{\eta^k\}$. In fact, we get the {\em strong} convergence $\eta^k\to\eta$ in $L^2$. To see this, we use the LICQ in {\bf(a)} and a finite cover of $[0,\T]$. For each $t_l\in[0,\T]$, the gradients $\nabla g_i$ with $i\in I_l:=I(\ox(t_l)-\ou(t_l))$ are linearly independent near $\ox(t_l)-\ou(t_l)$, and the indices $i\notin I_l$ stay inactive nearby. By compactness, finitely many intervals $J_l\ni t_l$ cover $[0,\T]$. On each $J_l$ and for all large $k$, we get from \eqref{d10} that $\eta^k_i(t)=0$ for $i\notin I_l$, while \eqref{e:par-dis} gives us
\[
\eta^k_{I_l}(t)=G^k_l(t)^{-1}\nabla g_{I_l}\bigl(\ox^k(\vartheta^k(t))-\ou^k(\vartheta^k(t))\bigr)\bigl(\dot\ox^k(t)-f(\ox^k(\vartheta^k(t)),\oa^k(t))\bigr).
\]
Here $G^k_l(t)$ is the Gram matrix of $\{\nabla g_i\}_{i\in I_l}$ at the nodal argument. It converges uniformly to an invertible limit on $J_l$. Since the last factor converges strongly in $L^2$, so does $\eta^k$ on each $J_l$, and hence on $[0,\T]$. Thus a subsequence of $\{\eta^k\}$ (no relabeling) exhibits the a.e.\ pointwise convergence $\eta^k(t)\to\eta(t)$. Passing to the limit in \eqref{e:par-dis} by using the continuity of $\nabla g_i$ and the uniform convergence $(\ox^k(\vartheta^k(\cdot)),\ou^k(\vartheta^k(\cdot)))\to(\ox,\ou)$ verifies \eqref{NOC:par}.

To justify now the first relationship in \eqref{NOC:cs}, suppose that $g_i(\ox(t)-\ou(t))>0$. By the continuity of $g_i$ and the pointwise convergence $(\ox^k(t),\ou^k(t))\to(\ox(t),\ou(t))$, we get $g_i(\ox^k(t)-\ou^k(t))>0$ for all large $k$. Then \eqref{d10} forces $\eta_i^k(t)=0$, and hence $\eta_i(t)=0$ in the limit. The uniqueness of $\eta\in L^2\bigl([0,\T];\R^m_+\bigr)$ in representation \eqref{NOC:par} follows directly from the imposed LICQ in ${\bf(a)}$.

\medskip
\noindent\textbf{Step~3}. \emph{Construction of approximating dual elements on $[0,\T]$.} Let $q^k(t):=(q^{xk}(t),q^{uk}(t))$ be the piecewise linear interpolant of the nodal values $q^k(t^k_j):=p^k_j$, $j=0,\ldots,k$. For $t\in[t^k_j,t^k_{j+1})$, $j=0,\ldots,k-1$, define
\begin{equation}\label{e:ADE}
\psi^k(t):=\frac{1}{h_k}\psi^k_j,\quad
\gg^k(t):=\gg^k_j,\quad
\widetilde\gg^k(t):=\sum\limits_{i=1}^m\gg^k_{ji}\,\nabla g_i(\ox^k_j-\ou^k_j),\quad
\Lm^k(t):=\Lm^k_j,\quad
\xi^k(t):=\xi^k_j,
\end{equation}
with the endpoint values $\psi^{uk}(\T^k):=0$, $\psi^{ak}(\T^k):=0$, and $\gg^k(\T^k):=0$. Let $\vartheta^k(t):=\max\{t^k_j:t^k_j\le t\}$ denote the left mesh endpoint containing $t$. Rewriting \eqref{d5}, \eqref{d6} in terms of \eqref{e:ADE} yields 
\begin{equation}\label{e:ADE1}
\begin{aligned}
\dot q^{xk}(t)&=\lm^k w^{xk}(t)-\nabla_xf\bigl(\ox^k(\vartheta^k(t)),\oa^k(\vartheta^k(t))\bigr)^*\Lm^k(t)-\xi^k(t)-\widetilde\gg^k(t),\\
\dot q^{uk}(t)&=\lm^k w^{uk}(t)+\xi^k(t)+\widetilde\gg^k(t)+\sum\limits_{i=1}^s\psi^{uk}_i(t)\,\nabla\nu_i\bigl(\ou^k(\vartheta^k(t))\bigr),\quad t\in[t^k_j,t^k_{j+1}).
\end{aligned}
\end{equation}
Extend now the discrete adjoint arcs to the continuous-time interval by
\begin{equation}\label{e:pk-def}
\begin{aligned}
p^{xk}(t)&:=q^{xk}(t)+\int_{[t,\T^k]}\bigl[-\xi^k(\tau)-\widetilde\gg^k(\tau)\bigr]\,d\tau,\\
p^{uk}(t)&:=q^{uk}(t)+\int_{[t,\T^k]}\Bigl[\xi^k(\tau)+\widetilde\gg^k(\tau)\Bigr]\,d\tau+\int_{[t,\T^k]}\sum\limits_{i=1}^s\nabla\nu_i\bigl(\ou^k(\vartheta^k(\tau))\bigr)\,d\psi^{uk}_i(\tau),
\end{aligned}
\end{equation}
for $t\in[0,\T^k]$. In particular, $p^k(\T^k)=q^k(\T^k)=p^k_k$, i.e., the construction is anchored at the actual terminal node. Beyond $\T^k$, we extend $q^k$ and $p^k$ to $[0,\T+\ve]$ by the constant value $p^k_k$.
Differentiating \eqref{e:pk-def} and combining it with \eqref{e:ADE1} gives us
\begin{equation}\label{e:ADE3}
\dot p^{xk}(t)=\lm^k w^{xk}(t)-\nabla_xf\bigl(\ox^k(\vartheta^k(t)),\oa^k(\vartheta^k(t))\bigr)^*\Lm^k(t),\qquad
\dot p^{uk}(t)=\lm^k w^{uk}(t)\;\mbox{ for a.e. }\;t\in[0,\T^k].
\end{equation}
Define further the vector measures
\begin{equation}\label{e:DM}
\int_S d\gamma^k_i:=\int_S\sum\limits_{j=0}^{k-1}\gg^k_{ji}\chi_{[t^k_j,t^k_{j+1})}(t)\,dt\ \ (i=1,\ldots,m),\quad
\int_S d\psi^{uk}_i:=\int_S\psi^{uk}_i(t)\,dt,\quad
\int_S d\xi^k:=\int_S\xi^k(t)\,dt
\end{equation}
for every Borel subset $S\subset[0,\T+\ve]$, with the piecewise constant densities extended by zero outside $[0,\T^k]$. Observe that the discrete sweeping measures are defined in \eqref{e:DM} with the same density that enters the construction \eqref{e:pk-def} of the absolutely continuous adjoint component, so that the measure limit and the BV-adjoint limit below are taken on the same scale. All the three discrete measures above are regarded as elements of $C^*([0,\T+\ve])$, where $\T^k\in[0,\T+\ve]$ for large $k$ by the localization condition \eqref{e:Pk-bc}. The weak$^*$ passages in what follows are taken in $C^*([0,\T+\ve])$, and the limiting measures are restricted to $[0,\T]$ afterwards. Working on the enlarged interval ensures that any mass the densities placed near $\T^k$ is retained in the limit even when $\T^k>\T$. In particular, a limiting atom of $\psi^u$ at $\T$, if present, arises from concentration of the running densities $\psi^{uk}_i(t)\,dt$ rather than from any discrete terminal coefficient. We now justify the normalization. Assumption {\bf(a)} of Theorem~\ref{ThNOC} and the uniform convergence in \eqref{e:discconv} transfer the LICQ to all mesh nodes for large $k$; see {\bf 4d} below. Hence the enhanced nontriviality \eqref{d17} holds. Suppose that the left-hand side of \eqref{e:dnc} vanishes. Then $\lm^k=0$, $\psi^{uk}=0$, $p^{uk}_0=q^{uk}(0)=0$, and $p^{xk}_k=0$ since $p^k(\T^k)=p^k_k$. This contradicts \eqref{d17}. Thus the left-hand side of \eqref{e:dnc} is positive. By the positive homogeneity of the dual elements, we can impose the normalization condition
\begin{equation}\label{e:dnc}
\lm^k+\|p^k(\T^k)\|+\|q^{uk}(0)\|+\sum\limits_{i=1}^s\|\psi^{uk}_i\|_{TV}=1.
\end{equation}\vspace*{-0.1in}

\noindent\textbf{Step~4}. \emph{Verification of the adjoint dynamic system \eqref{NOC:adj-ac}--\eqref{NOC:qx}.}
First we extract from \eqref{e:dnc} a subsequence (no relabeling) along which $\lm^k\to\lm\ge 0$ as $k\to\infty$.

{\bf 4a}. {\em Uniform bounded variation of $q^{uk}$.} Combining \eqref{d8} with the estimates \eqref{e18} and \eqref{e:BV-pen-zero} for the optimal solutions of $(P_k)$ gives us the relationships
\begin{equation}\label{e:quk-est1}
\begin{aligned}
\sum\limits_{j=0}^{k-1}\|q^{uk}(t^k_{j+1})-q^{uk}(t^k_j)\|
&=\sum\limits_{j=0}^{k-1}\|p^{uk}_{j+1}-p^{uk}_j\|\\
&\le\lm^k(\mu+\widetilde\mu)+\frac{\lm^k}{h_k}\|\th^{Uk}_0\|+\lm^k\sum\limits_{j=1}^{k-1}\|v^{uk}_j-v^{uk}_{j-1}\|+\lm^k\|v^{uk}_0\|+\|p^{uk}_0\|.
\end{aligned}
\end{equation}
It follows from \eqref{d8} that $v^{uk}_j=\nabla_{\dot u}\ell_2(t^k_j,V^u_j)$ with $V^u_j=(\ou^k_{j+1}-\ou^k_j)/h_k$. The Lipschitz estimate in \eqref{l2-est} combined with \eqref{e:BV-pen-zero} and $\T^k\le\T+\ve$ from \eqref{e:Pk-bc} leads us to the inequality
\[
\lm^k\sum\limits_{j=1}^{k-1}\|v^{uk}_j-v^{uk}_{j-1}\|\le\lm^k L_{\ell_2}(\T{+\ve}+\widetilde\mu),
\]
while the growth bound in \eqref{l2-est} together with $\|V^u_0\|\le C_V$ from \eqref{e:CV} tells us that
\[
\lm^k\|v^{uk}_0\|\le\lm^k L_{\ell_2}\bigl(1+\|V^u_0\|\bigr)\le\lm^k L_{\ell_2}(1+{C_V}).
\]
Moreover, \eqref{e:var} gives us $\|\dot\ou(t)\|\le2\mu$ for a.e.\ $t\in[0,\T]$. Indeed, the norm of the average velocity over the first reference mesh cell is at most $\mu$, and its distance from any velocity value is at most $\var(\dot\ou;[0,\T])\le\mu$. This bound and \eqref{e:CV} yield $\lm^k\|\th^{Uk}_0\|/h_k\le\lm^k({2\mu+C_V})$. Substituting this into \eqref{e:quk-est1} and using \eqref{e:dnc}, we arrive at 
\begin{equation}\label{e:quk-est2}
\sum\limits_{j=0}^{k-1}\|q^{uk}(t^k_{j+1})-q^{uk}(t^k_j)\|
\le {3\mu+\widetilde\mu+C_V}+L_{\ell_2}(\T{+\ve}+\widetilde\mu)+L_{\ell_2}(1+{C_V})+1,
\end{equation}
which verifies that the sequence $\{q^{uk}\}$ is 
of bounded variation uniformly on $[0,\T]$.

{\bf 4b}. {\em Boundedness of $\{p^{xk}_j\}$ via the discrete Gr\"onwall inequality.} Here we derive a recursion for $\|p^{xk}_j\|$ and then apply the discrete Gr\"onwall inequality from Lemma~\ref{lem:gronwall}.

We start with the \emph{derivation of the recursion.} Summing up the discrete adjoint relations \eqref{d5} and \eqref{d6} and using definition \eqref{d9} of $\Lm^k_j$ (which contains $p^{xk}_{j+1}$) bring us to
\begin{equation}\label{e:pxk-puk}
p^{xk}_{j+1}-p^{xk}_j=\lm^kh_k\bigl(w^{xk}_j+w^{uk}_j\bigr)+\sum\limits_{i=1}^s\psi^{uk}_{ji}\,\nabla\nu_i(\ou^k_j)-h_k\nabla_xf(\ox^k_j,\oa^k_j)^*\Lm^k_j-(p^{uk}_{j+1}-p^{uk}_j).
\end{equation}
Isolating the term $-h_k\nabla_xf^*p^{xk}_{j+1}$ from $\Lm^k_j$, moving $p^{xk}_{j+1}$ to the left, and using the triangle inequality yield
\begin{equation}\label{e:pxk-rec}
\|p^{xk}_j\|\le(1+h_kL_f)\,\|p^{xk}_{j+1}\|+B^k_j,
\end{equation}
where the numbers $B^k_j$ for $j=0,\ldots, k-1$ and $k\in\N$ are defined by
\begin{equation}\label{e:Bkj-def}
\begin{aligned}
B^k_j:={}&\lm^kh_k\bigl(\|w^{xk}_j\|+\|w^{uk}_j\|\bigr)+L_\nu\sum\limits_{i=1}^s|\psi^{uk}_{ji}|\\
&+\lm^kh_kL_f\bigl(\|v^{xk}_j\|+\|\th^{Xk}(t^k_j)\|\bigr)+\|q^{uk}(t^k_{j+1})-q^{uk}(t^k_j)\|.
\end{aligned}
\end{equation}
Next we obtain a \emph{bound on $\sum\limits_{j=0}^{k-1}B^k_j$} considering the four summand in \eqref{e:Bkj-def} separately.\vspace*{-0.1in}

\begin{enumerate}
\item[\bf(i)] \emph{Subgradient terms.} Since $(w^{xk}_j,w^{uk}_j)\in\partial_{x,u}\ell(t^k_j,\cdot)$ and Lemma~\ref{lem:subgradient-bound} provides a uniform bound $L_\ell$, the Riemann sum estimate tells us that
\[
\sum\limits_{j=0}^{k-1}h_k\|w^{xk}_j\|\le (\T+\ve)\,L_\ell=:L_{\text{cum}},
\]
and similarly $\sum\limits_j h_k\|w^{uk}_j\|\le L_{\text{cum}}$. Using $\lm^k\le 1$ from \eqref{e:dnc} ensures  the estimate
\[
\sum\limits_{j=0}^{k-1}\lm^kh_k\bigl(\|w^{xk}_j\|+\|w^{uk}_j\|\bigr)\le 2L_{\text{cum}}.
\]
\item[\bf(ii)] \emph{Control constraint multiplier terms.} The normalization condition \eqref{e:dnc} yields
\[
\sum\limits_{j=0}^{k-1}L_\nu\sum\limits_{i=1}^s|\psi^{uk}_{ji}|=L_\nu\sum\limits_{i=1}^s\int_0^{\T^k}|\psi^{uk}_i(t)|\,dt\le L_\nu.
\]
\item[\bf(iii)] \emph{Discrete velocity subgradient terms.} Since the subgradient sequence  $v^{xk}_j\in\partial_{X}\ell(\cdot)$ is a.e.\ bounded by $L_\ell$ and since $\sum\limits_j h_k\|\th^{Xk}(t^k_j)\|\to 0$ by \eqref{e:theta-zero}, we deduce from the Cauchy--Schwarz inequality that
\[
\sum\limits_{j=0}^{k-1}\lm^kh_kL_f\bigl(\|v^{xk}_j\|+\|\th^{Xk}(t^k_j)\|\bigr)\le L_f\Bigl(2L_{\text{cum}}+\sqrt{\T}\,\|\th^{Xk}\|_{L^2}\Bigr),
\]
which provides a uniform estimate with respect to $k\in\N$.
\item[\bf(iv)] \emph{$q^{uk}$ increments.} The estimate in \eqref{e:quk-est2} implies that
\[
\sum\limits_{j=0}^{k-1}\|q^{uk}(t^k_{j+1})-q^{uk}(t^k_j)\|\le {3\mu+\widetilde\mu+C_V}+L_{\ell_2}(\T{+\ve}+\widetilde\mu)+L_{\ell_2}(1+{C_V})+1.
\]
\end{enumerate}
Summing up {\bf(i)--(iv)} gives us a constant $\widetilde M_1>0$ depending only on $L_{\text{cum}}$, $L_\nu$, $L_f$, $L_{\ell_2}$, $\T$, $\ve$, $\mu$, and $\widetilde\mu$ together with the prefactor in \eqref{e:theta-zero} ensuring the desired upper bound
\begin{equation}\label{e:Bkj-sum}
\sum\limits_{j=0}^{k-1}B^k_j\le\widetilde M_1\qquad\text{for every }\;k\in\N.
\end{equation}
Observe further that the normalization condition \eqref{e:dnc} provides $\|p^{xk}_k\|\le 1$. Apply now  Lemma~\ref{lem:gronwall} with the data  $a_j:=\|p^{xk}_j\|$, $h:=h_k$, $L:=L_f$, $B_j:=B^k_j$, and the bound $kh_k=\T^k\le\T+\varepsilon$ from \eqref{e:Pk-bc} together with \eqref{e:Bkj-sum}. This leads us to the uniform estimate
\begin{equation}\label{e:pxk-bdd}
\max_{0\le j\le k}\|p^{xk}_j\|\le e^{L_f(\T+\varepsilon)}\bigl(1+\widetilde M_1\bigr)=:\widetilde M_2,\quad k\in\N.
\end{equation}
Replacing $\T^k$ by the upper horizon $\T+\varepsilon$ in the exponent makes $\widetilde M_2$ uniform over the entire range $\T^k\in[0,\T+\varepsilon]$, which is what the limiting arguments of Step~{\bf 6} require.

{\bf 4c}. {\em Uniform bounded variations of $q^{xk}$.} From \eqref{e:pxk-puk} and \eqref{e:pxk-bdd}, we get
\[
\sum\limits_{j=0}^{k-1}\|q^{xk}(t^k_{j+1})-q^{xk}(t^k_j)\|
=\sum\limits_{j=0}^{k-1}\|p^{xk}_{j+1}-p^{xk}_j\|
\le h_kL_f\sum\limits_{j=0}^{k-1}\|p^{xk}_{j+1}\|+\sum\limits_{j=0}^{k-1}B^k_j\le\T L_f\widetilde M_2+\widetilde M_1.
\]
Combined with $\|q^{xk}(\T^k)\|=\|p^{xk}_k\|\le 1$, this yields the uniform boundedness and the uniform BV property of the sequence $\{q^k(\cdot)\}$ on $[0,\T+\ve]$. Then Helly's selection theorem ensures the (subsequence) convergence $q^k(t)\to q(t)$ at every continuity point of $q\in BV([0,\T];\R^{2n})$ (the complement of this set is at most countable) as well as the weak$^*$ convergence of the associated vector measures $dq^k$  to some $dq\in C^*([0,\T];\R^{2n})$.

{\bf 4d}. {\em $L^1$-boundedness of $\{\xi^k\}$.} Using the classical H\"older's inequality together with \eqref{e12} and \eqref{e:eta-bound} yields
\[
\int_0^{\T^k}\|\xi^k(t)\|\,dt
\le\frac{M_3\beta}{M_1}\sqrt{\int_0^{\T^k}\bigl[M+M\|\ox^k(t)\|+\|\dot\ox^k(t)\|\bigr]^2dt}\cdot\sqrt{\int_0^{\T^k}\|\Lm^k(t)\|^2dt}.
\]
The first integral is bounded due to the $W^{1,2}$-convergence of $\ox^k$ to $\ox$. For the second integral, we get
\[
\int_0^{\T^k}\|\Lm^k(t)\|^2dt\le 3\left[\int_0^{\T^k}\|v^{xk}(t)\|^2dt+\int_0^{\T^k}\|\th^{Xk}(t)\|^2dt+\int_0^{\T^k}\|q^{xk}(t)\|^2dt\right],
\]
where all the three summands are bounded by {\bf 4b, 4c}, and \eqref{e:theta-zero}. Hence $\{\xi^k\}$ is uniformly bounded in $L^1$, and so is $\{\widetilde\gg^k\}$ by the first line of \eqref{e:ADE1}. The total variation bound on $\gamma^k$ then follows from LICQ along the trajectory. Indeed, since $\widetilde\gg^k=\sum\limits_{i=1}^m\gg^k_{ji}\nabla g_i$ is bounded in $L^1$, and the discrete coderivative description in \eqref{e:prop-DNR} and \eqref{e17} forces $\gg^k_{ji}=0$ for every inactive index $i\notin I(\ox^k_j-\ou^k_j)$, we have that effectively
$$
\widetilde\gg^k=\sum\limits_{i\in I(\ox^k_j-\ou^k_j)}\gg^k_{ji}\nabla g_i.
$$
The linear independence of the \emph{active} gradients $\{\nabla g_i(\ox^k_j-\ou^k_j)\}_{i\in I(\ox^k_j-\ou^k_j)}$ under LICQ yields 
$$
\sum\limits_{i\in I(\ox^k_j-\ou^k_j)}|\gg^k_{ji}|\le c\,\|\widetilde\gg^k\|,
$$
where the uniformity of the constant $c>0$ in $j,k$ readily follows from the standard compactness device and the imposed LICQ. Therefore, the sequence $\gamma^k$ has uniformly bounded total variation. The total variation bounds on $\psi^{uk}_i$ and $\xi^k$ follow from the normalization \eqref{e:dnc} and from the established $L^1$-bound, respectively. The complementarity  conditions in \eqref{d11} and \eqref{d12} ensure that, in the weak$^*$ limit, the supports of $\gamma$ are confined to the active set. Passing to a further subsequence, we obtain the weak$^*$ convergence
\[
\gg^k\stackrel{w^*}{\to}\gamma,\qquad \psi^{uk}_i\stackrel{w^*}{\to}\psi^u_i,\qquad \xi^k\stackrel{w^*}{\to}\xi
\]
in $C^*([0,\T+\ve];\R^m)$, $C^*([0,\T+\ve];\R_+)$, and $C^*([0,\T+\ve];\R^n)$, respectively, and then restricted to $[0,\T]$.

Moreover, the weak$^*$ limit of $\xi^k\,dt$ is \emph{absolutely continuous} with the explicit density used in \eqref{NOC:qx}. Indeed, the LICQ pseudoinverse representation of the sweeping residual in \eqref{e:par-dis} expresses $\eta^k$ as a uniformly convergent linear image of $(V^x_j-f(\ox^k_j,\oa^k_j))$, which converges strongly in $L^2$ by \eqref{e:discconv}. Hence the convergence $\eta^k\to\eta$ is strong in $L^2$, which ensures in turn the strong $L^2$-convergence of  
$$
\sum_i\eta^k_i\nabla^2 g_i\big(\ox^k(\vartheta^k(\cdot))-\ou^k(\vartheta^k(\cdot))\big)\to\sum_i\eta_i\nabla^2 g_i(\ox-\ou)
$$
since the Hessians are bounded by \textbf{(H2.2)} and their arguments converge uniformly. On the other hand, we clearly have the weak convergence
$$
\Lm^k(t)=q^{xk}(\vartheta^k(t)+h_k)-\lm^k(\th^{Xk}(t)+v^{xk}(t))\rightharpoonup q^x-\lm v^x
$$
in $L^2$ due to $\lm^k\to\lm$, the weak $L^2$-convergence $v^{xk}\rightharpoonup v^x$, the strong vanishing $\th^{Xk}\to 0$ from \eqref{e:theta-zero}, and the strong $L^2$-convergence $q^{xk}\to q^x$. The last convergence is spelled out as follows: the jump set of the BV limit $q^x$ is at most countable, so Helly's theorem in {\bf 4c} gives us $q^{xk}(t)\to q^x(t)$ for a.e.\ $t\in[0,\T]$, the uniform BV bound yields a uniform $L^\infty$-bound on $\{q^{xk}\}$, and the bounded convergence theorem then upgrades the a.e.\ convergence to the strong one $q^{xk}\to q^x$ in $L^2([0,\T];\R^n)$. We know that the product of a strongly $L^2$-convergent factor with a weakly $L^2$-convergent one converges weakly in $L^1$, and thus we arrive at 
\begin{equation}\label{e:xi-density}
\xi^k\ \rightharpoonup\ \Bigl(\sum_{i=1}^m\eta_i\,\nabla^2 g_i(\ox-\ou)\Bigr){\bigl(q^x-\lm v^x\bigr)}\qquad\text{weakly in }L^1([0,\T];\R^n).
\end{equation}
By the uniqueness of limits, the measure $\xi$ agrees with this $L^1$-density times $dt$; this is exactly the absolutely continuous form of $\xi$ entering \eqref{NOC:qx}. Moreover, $\xi^k$ carries no mass on $(\T,\T^k]$ in the limit. Indeed, $\|\xi^k\|_{L^1(\T,\T^k)}\le C\|\eta^k\|_{L^2(\T,\T^k)}\|\Lm^k\|_{L^2}\to0$, since $\eta^k$ is controlled there by $\dot\ox^k$ and $f$ through \eqref{e:eta-bound}, while $\dot\ox^k\to0$ in $L^2(\T,\T^k)$ by \eqref{e:discconv} with the zero extension of the reference velocity.

{\bf 4e}. {\em Boundedness of $\{\psi^{ak}\}$ in $L^2$.} {Here we assume the last condition in \eqref{e:BV-pen-zero}. When $A$ is convex and this condition is dropped, this step is not used; see Step~{\bf 5}.} Denoting $\psi^{ak}(t):=\psi^{ak}_j/h_k$ for $t\in[t^k_j,t^k_{j+1})$, we get by \eqref{d7} the equality
\[
\psi^{ak}(t)=\nabla_a f\bigl(\ox^k(\vartheta^k(t)),\oa^k(\vartheta^k(t))\bigr)^{\!*}\Lm^k(t)-\lm^k w^{ak}(t)-\lm^k\,\th^{ak}(t).
\]
It follows from \eqref{e10} in {\bf(H1)} combined with $\lm^k\le1$ that
\begin{equation}\label{psi-bound}
\|\psi^{ak}\|_{L^2}\le L_f\,\|\Lm^k\|_{L^2}+\|w^{ak}\|_{L^2}+\|\th^{ak}\|_{L^2}.
\end{equation}
The right-hand side of \eqref{psi-bound} is uniformly bounded. Indeed, $\|\Lm^k\|_{L^2}$ by~{\bf 4b} and {\bf 4c}, $\|w^{ak}\|_{L^2}$ by the subgradient bound of Step~{\bf 1}, and $\|\th^{ak}\|_{L^2}\to 0$ by \eqref{e:theta-zero}. Hence $\{\psi^{ak}\}$ is bounded in $L^2$.

{\bf 4f}. {\em Weak convergence and passage to the limit.} We deduce from \eqref{e:dnc}{, \eqref{e:ADE3},} and the bounds in {\bf 4b--4{d}} that the sequence $\{p^k\}$ is bounded in $W^{1,2}([0,\T+\ve];\R^{2n})$. Mazur's theorem produces convex combinations of $\dot p^k$ {converging a.e.\ pointwise to $\dot p\in L^2$. Under the assumption of Step~{\bf 4e}, it also produces convex combinations of $\psi^{ak}$ converging a.e.\ to $\psi^a\in L^2$. Otherwise, the bounded sequence is the combined force on the right-hand side of the formula in Step~{\bf 4e}, and $\psi^a$ is obtained in Step~{\bf 5}}. Passing to the limit in \eqref{e:ADE3} yields \eqref{NOC:adj-ac}. Construction \eqref{e:pk-def} combined with \eqref{d8} tells us that
\begin{equation}\label{e:quk-ptw}
q^{uk}(\vartheta^k(t)+h_k)=\lm^k\bigl(v^{uk}(t)+\th^{Uk}(t)\bigr).
\end{equation}
Passing to the limit with the usage of \eqref{e:theta-zero}, \eqref{NOC:subg}, and the structure in \eqref{e:ell-rep} verifies \eqref{NOC:qu}. To obtain the second implication in \eqref{NOC:cs}, observe that \eqref{d15} yields $\eta^k_i(t)\la\nabla g_i(\ox^k(\vartheta^k(t))-\ou^k(\vartheta^k(t))),\Lm^k(t)\ra=0$. Here $\eta^k_i\to\eta_i$ strongly in $L^2$ by Step~{\bf 2}, while the pairing converges weakly in $L^2$ to $\la\nabla g_i(\ox-\ou),q^x-\lm v^x\ra$ by {\bf 4d}. Hence the products converge weakly in $L^1$, and we get $\eta_i\la\nabla g_i(\ox-\ou),q^x-\lm v^x\ra=0$ a.e., which verifies \eqref{NOC:cs}.

To get \eqref{NOC:qx}, we pass to the limit in \eqref{e:pk-def}. The measure weak$^*$ convergence of $\gg^k$ to $\gamma$ combined with the uniform convergence of $\nabla g_i(\ox^k(\vartheta^k(\cdot))-\ou^k(\vartheta^k(\cdot)))$ to $\nabla g_i(\ox(\cdot)-\ou(\cdot))$ imply that
\[
\int_t^{\T^k}\sum\limits_{i=1}^m\gg^k_i(s)\,\nabla g_i(\ox^k(\vartheta^k(s))-\ou^k(\vartheta^k(s)))\,ds\to\int_{(t,\T]}\sum\limits_{i=1}^m\nabla g_i(\ox(s)-\ou(s))\,d\gamma_i(s)
\]
for all $t\in[0,\T]$ except a countable subset. The analogous passage holds for the $\psi^{uk}_i$ term due to the weak$^*$ convergence of $\psi^{uk}_i$ to $\psi^u_i$ and the uniform convergence of $\nabla\nu_i(\ou^k(\vartheta^k(\cdot)))$ to $\nabla\nu_i(\ou(\cdot)$), as well as for the term $\xi^k$ due to the weak $L^1$-convergence \eqref{e:xi-density} tested against the indicator $\mathbf 1_{(t,\T]}\in L^\infty$. These three passages, together with \eqref{e:pk-def}, justify \eqref{NOC:qx}. Applied instead to the second line of \eqref{e:pk-def}, they establish the $u$-side identity \eqref{NOC:puqu} of conclusion {\bf(ii)} valid for every $t\in[0,\T]$. In more detail, at the continuity points $t\in(0,\T)$ of the limit, the integrals over $[t,\T^k]$ converge to the integrals over $(t,\T]$, and we pass to the right-continuous representative on $(0,\T]$. At $t=0$, the integrals over $[0,\T^k]$ are integrals against the constant test function $1$ on $[0,\T+\ve]$, since the densities vanish beyond $\T^k$. Hence they converge to the integrals over $J_0=[0,\T]$. The boundary value $q(0)$ is the limit of the bounded vectors $q^k(0)$ along the subsequence fixed above. At $t=\T$, we put $q(\T):=q(\T+)=p(\T)$, the right limit on $[0,\T+\ve]$, since $q=p$ on $(\T,\T+\ve]$. This completes Step~{\bf 4} of the proof.

\medskip
\noindent\textbf{Step~5}. \emph{Adjoint inclusion for the perturbation control \eqref{NOC:psi-a} and maximization conditions.}
Rearranging \eqref{d7} and using definition \eqref{d9} of $\Lm^k_j$ tell us, for $t\in[t^k_j,t^k_{j+1})$, that
\begin{equation*}
\begin{aligned}
\psi^{ak}(t):=\frac{\psi^{ak}_j}{h_k}=&-\lm^k w^{ak}(t)-\lm^k\th^{ak}(t)\\
&+\nabla_af\bigl(\ox^k(\vartheta^k(t)),\oa^k(\vartheta^k(t))\bigr)^*{\bigl(q^{xk}(h_k+\vartheta^k(t))-\lm^k(v^{xk}(t)+\th^{Xk}(t))\bigr)}.
\end{aligned}
\end{equation*}
Mazur's theorem applied to the weak $L^2$-compact sequences $\{w^{ak}\}$ and $\{v^{xk}\}$, vanishing \eqref{e:theta-zero} for $\th^{ak}$ and $\th^{Xk}$, the strong convergence $\oa^k\to\oa$, the a.e.\ pointwise convergence $q^{xk}(h_k+\vartheta^k(t))\to q^x(t)$, and the robustness (outer semicontinuity) of the normal cone mapping $N_A$, together with the discrete inclusion $\psi^{ak}_j\in N_A(\oa^k_j)$ from \eqref{e:psia}, readily verify \eqref{NOC:psi-a} via the passage to the limit as $k\to\infty$.

{{\em The case of convex $A$ without the last condition in \eqref{e:BV-pen-zero}.} Then \eqref{d7} contains the additional term $\zeta^k(m_j-m_{j+1})/h_k$ from Remark~\ref{rem:BV-pen}, so Step~{\bf 4e} and the argument above do not apply. We use a variational inequality instead. Let $g^k(t)$ denote the right-hand side of the formula for $\psi^{ak}(t)$ above. By \eqref{d7}, we have $\psi^{ak}_j+\zeta^k(m_j-m_{j+1})=h_kg^k_j$. Take any $b=(b_0,\ldots,b_{k-1})\in A^k$ with $\sum_{j=1}^{k-1}\|b_j-b_{j-1}\|\le\Tilde\mu+1$. Since $A$ is convex, $\la\psi^{ak}_j,b_j-\oa^k_j\ra\le0$. Summation by parts with $m_0=m_k=0$ gives us
\[
\sum_{j=0}^{k-1}\la m_j-m_{j+1},b_j-\oa^k_j\ra=\sum_{j=1}^{k-1}\la m_j,(b_j-b_{j-1})-(\oa^k_j-\oa^k_{j-1})\ra\le\sum_{j=1}^{k-1}\|b_j-b_{j-1}\|-\sum_{j=1}^{k-1}\|\oa^k_j-\oa^k_{j-1}\|.
\]
If $\zeta^k>0$, the last constraint in \eqref{e:Pk-bv} is active, and the right-hand side is not positive. In both cases, we arrive at the discrete variational inequality
\[
\int_0^{\T^k}\la g^k(t),b^k(t)-\oa^k(t)\ra\,dt\le0,
\]
where $b^k$ is the piecewise constant function with the values $b_j$. Now let $b\colon[0,\T]\to A$ be of bounded variation with $\var(b;[0,\T])\le\Tilde\mu+1$, extended by $b(\T)$ beyond $\T$. Its values at the rescaled nodes $t^k_j\T/\T^k$ give admissible vectors $b$ above, and the corresponding $b^k$ converge to $b$ in $L^2$. The bound on $\Lm^k$ in $L^2$ from Steps~{\bf 4a--4c}, the weak convergence of $\Lm^k$ and $w^{ak}$, the vanishing of $\th^{ak}$ by \eqref{e:theta-zero}, and the strong convergence $\oa^k\to\oa$ provide
\begin{equation}\label{e:VI-a}\tag{7.45$'$}
\int_0^{\T}\la\psi^a(t),b(t)-\oa(t)\ra\,dt\le0
\end{equation}
with $\psi^a$ defined in \eqref{NOC:psi-a}. Fix $c\in A$, $\tau\in(0,\T)$, $\delta>0$, and $s\in(0,1]$. Put $b:=\oa+s\chi_{[\tau,\tau+\delta)}(c-\oa)$. Its values lie in $A$ by convexity, and $\var(b;[0,\T])\le\var(\oa;[0,\T])+2s\,{\rm diam}\,A\le\mu+2s\,{\rm diam}\,A$. This is at most $\Tilde\mu+1$ for small $s>0$, since $\Tilde\mu\ge3\mu+1$ by \eqref{e34}. Inserting $b$ into \eqref{e:VI-a}, dividing by $s\delta$, and letting $\delta\dn0$ at the Lebesgue points of $\psi^a$ and $\oa$, we get $\la\psi^a(\tau),c-\oa(\tau)\ra\le0$. Taking $c$ in a countable dense subset of $A$, we arrive at $\psi^a(\tau)\in N_A(\oa(\tau))=\co\,N_A(\oa(\tau))$ for a.e.\ $\tau$, which is \eqref{NOC:psi-a}.}

\medskip
\noindent\textbf{Step~6}. \emph{Transversality and endpoint complementary slackness in \eqref{NOC:alpha}--\eqref{NOC:trans}.}
Substep~{\bf 6a} transports the endpoint normal cone inclusion \eqref{NOC:alpha} to the limit, substep~{\bf 6b} establishes the averaged Hamiltonian convergence $\mathcal H^k\to\mathcal H$; and substep~{\bf 6c} passes the free-time block of \eqref{d16} to \eqref{NOC:trans}.

{\bf 6a}. {\em Endpoint normal cone passage.} The first relationship in \eqref{d16} combined with \eqref{e:dnc} gives us
\[\left\|\sum\limits_{i=1}^m\al^k_i\nabla g_i(\ox^k_k-\ou^k_k)\right\|=\|p^{uk}_k\|\le \|p^k(\T^k)\|\le 1.
\]
The inverse triangle inequality \eqref{e13} from \textbf{(H2.3)} yields
\[
\al^k_i\le\frac{1}{M_1}\!\!\sum\limits_{j\in I(\ox^k_k-\ou^k_k)}\!\!\al^k_j\|\nabla g_j(\ox^k_k-\ou^k_k)\|
\le\frac{\beta}{M_1}\Bigl\|\sum\limits_{i=1}^m\al^k_i\nabla g_i(\ox^k_k-\ou^k_k)\Bigr\|\le\frac{\beta}{M_1}.
\]
Thus we get, along a subsequence, that $\al^k_i\to\al_i\ge 0$ for each $i$. Indeed, for inactive indices $i\notin I(\ox(\T)-\ou(\T))$, it follows from \eqref{d14} and the continuity of $g_i$ that $\al^k_i=0$ for large $k$, and so the equalities $\al_i=0$ hold trivially. For active indices $i$, the uniqueness of the limit follows from the linear independence of the active gradients $\{\nabla g_i(\ox(\T)-\ou(\T))\}_{i\in I(\ox(\T)-\ou(\T))}$ under the imposed LICQ.
Combining $\ou^k_k\to\ou(\T)$ and $\al^k_i\to\al_i$, the passage to the limit in the first relationship of \eqref{d16} produces \eqref{NOC:alpha}, where the endpoint complementarity conditions $g_i(\ox(\T)-\ou(\T))>0\Lto\al_i=0$ follow directly from \eqref{d14} by the continuity of $g_i$.

{\bf 6b}. {\em Averaged Hamiltonian convergence $\mathcal H^k\to\mathcal H$.} This convergence requires the usage of Helly's theorem for the convergence $q^{xk}\to q^x$ and $q^{uk}\to q^u$ of the BV adjoints, which is available from Step~{\bf 4c}. We form the Ces\`aro average \eqref{e:Hk} relating $\mathcal H^k$ to these BV adjoints. Namely, by \eqref{d8} and \eqref{d9}, the forward-node values $p^{uk}_{j+1}$ and $p^{xk}_{j+1}=\Lm^k_j+\lm^k(h_k^{-1}\th^{Xk}_j+v^{xk}_j)$ carry the Step~{\bf 4} limits $q^u$ and $q^x$ at a.e.\ interior node, the single terminal node having the Ces\`aro weight $1/k\to0$, and the resulting average agrees in the limit with \eqref{e:Hk}. Set
\[
I^k(t):=\la q^{xk}(\vartheta^k(t)+h_k),\dot\ox^k(t)\ra+\la q^{uk}(\vartheta^k(t)+h_k),\dot\ou^k(t)\ra-\lm^k\,\ell\bigl(t,\ox^k,\ou^k,\oa^k,\dot\ox^k,\dot\ou^k\bigr)(t),
\]
and let $I(t)$ be obtained by deleting all superscripts $k$. Decomposing
\begin{equation}\label{hamil}
\mathcal H^k-\mathcal H=\Bigl(\tfrac{1}{\T^k}-\tfrac{1}{\T}\Bigr)\int_0^{\T^k}I^k\,dt+\tfrac{1}{\T}\int_0^{\T^k}(I^k-I)\,dt-\tfrac{1}{\T}\int_{\T^k}^{\T}I\,dt,
\end{equation}
observe that the three summands in \eqref{hamil} tend to zero separately. The first is bounded by $|\T-\T^k|/(\T\,\T^k)$ times the uniform $L^1$-bound on $I^k$ (from the $W^{1,2}$-bound on $p^k$ in {\bf 4b} and {\bf 4c}, the $L^2$-bound on the subgradient selections $v^k$, and the Lipschitz continuity of $\ell$) with the convergence $1/\T^k\to 1/\T$ and $|\T-\T^k|\to 0$ by Theorem~\ref{conver}. The second term in \eqref{hamil} is bounded by the $L^2$-to-$L^1$ Cauchy--Schwarz inequality applied to $(\dot\ox^k-\dot\ox$, $\dot\ou^k-\dot\ou$, $\ell^k-\ell)$, the pointwise convergence $q^{xk}\to q^x$ from Helly's selection theorem (Step~{\bf 4c}), and the strong convergence \eqref{e:discconv}. The third term in \eqref{hamil}, $\tfrac1\T\int_{\T^k}^{\T}I\,dt$, tends to zero by the absolute continuity of the Lebesgue integral. Namely, the integrand $I=\la q^x,\dot\ox\ra+\la q^u,\dot\ou\ra-\lm\ell$ is a sum of products of the bounded BV arcs $q^x,q^u$ with the $L^2$ velocities $\dot\ox,\dot\ou$ and the integrable running cost, hence $I\in L^1([0,\T])$. Therefore, $\int_E I\,dt\to 0$ as the Lebesgue measure $|E|\to 0$, where $E$ is the interval between $\T^k$ and $\T$ of length $|\T-\T^k|\to 0$ by Theorem~\ref{conver}. Summing this up yields $\mathcal H^k\to\mathcal H$ as $k\to\infty$.

{\bf~6c}. {\em Passage to the limit of the free-time transversality block.} It remains to pass the limit in the second relationship of \eqref{d16}, i.e., the discrete free-time transversality inclusion
\begin{equation}\label{e:disc-trans-full}
\Bigl(-p^{xk}_k+\textstyle\sum\limits_{i=1}^m\al^k_i\nabla g_i(\ox^k_k-\ou^k_k),\ \mathcal H^k+\lm^k(\T-\T^k)+\lm^k\varrho^k\Bigr)\in\lm^k\,\partial\varphi(\ox^k_k,\T^k)+N_{\Xi^k_x\times\Xi^k_T}(\ox^k_k,\T^k),
\end{equation}
to the continuous transversality condition \eqref{NOC:trans}. We split our consideration into three parts.

\emph{Time slot.} The free-time triple converges. Indeed, $\lm^k\in[0,1]$ by \eqref{e:dnc}, $|\lm^k(\T-\T^k)|\le|\T-\T^k|\to 0$ by Theorem~\ref{conver}, $|\lm^k\varrho^k|\le|\varrho^k|\to 0$ by Lemma~\ref{lem:free-time-rates}{\bf(i)}, and $\mathcal H^k\to\mathcal H$ by {\bf 6b}. Hence the second slot of the left-hand side of \eqref{e:disc-trans-full} converges to $\mathcal H$ as $k\to\infty$. 

\emph{State slot.} The endpoint convergence $\ox^k_k\to\ox(\T)$ follows from the uniform convergence $\ox^k\to\ox$ together with $\T^k\to\T$. The state slot of the left-hand side of \eqref{e:disc-trans-full} converges by {\bf 6a}, i.e.,
\[
-p^{xk}_k+\sum_{i=1}^m\al^k_i\nabla g_i(\ox^k_k-\ou^k_k)=-p^{xk}_k-p^{uk}_k\longrightarrow -p^x(\T)-p^u(\T)          
\]
by using the convergence $\ou^k_k\to\ou(\T)$ and the uniform convergence $p^k\to p$.

\emph{Outer limit passage.} The sets $\Xi_x,\Xi_T$ are closed by \textbf{(H5)}. Since $\mu^x_k\downarrow 0$ for the localization radius 
in \eqref{e:Pk-bc}, the perturbed endpoint constraint sets $\Xi^k_x:=\Xi_x+\mu^x_k \overline{B}(0,1)$ and $\Xi^k_T:=\Xi_T+\mu^x_k[-1,1]$ are closed. We claim that
$$\Limsup_{k\to\infty}N_{\Xi^k_x\times\Xi^k_T}(\ox^k_k,\T^k)\subset N_{\Xi_x\times\Xi_T}(\ox(\T),\T).
$$
Since the normal cone to a product is the product of normal cones, it suffices to consider $\Xi^k_x=\{z\mid\dist(z;\Xi_x)\le\mu^x_k\}$; the set $\Xi^k_T$ is treated in the same way. Let $v^k\in N_{\Xi^k_x}(z^k)$ with $z^k\to\ox(\T)$ and $v^k\to v$. If $\dist(z^k;\Xi_x)<\mu^x_k$, then $v^k=0$. Otherwise, $\dist(z^k;\Xi_x)=\mu^x_k>0$. The subgradients of the distance function at $z^k\notin\Xi_x$ are the unit vectors $(z^k-w)/\mu^x_k$ with $w\in\Pi(z^k;\Xi_x)$. Hence $0\notin\partial\dist(z^k;\Xi_x)$, and the normal cone to the sublevel set gives us $v^k=\al_k(z^k-w^k)$ with $\al_k\ge0$ and $w^k\in\Pi(z^k;\Xi_x)$; see \cite[Chapters~8 and~10]{Rockafellar1998}. Since $z^k\to\ox(\T)$, the definition \eqref{lnc} of the limiting normal cone yields $v\in N_{\Xi_x}(\ox(\T))$.
For the subdifferential term, we recall that the terminal cost $\varphi$ is locally Lipschitz around the optimal endpoint $(\ox(\T),\T)$, which readily ensures the convergence of limiting subgradients; see \cite[Theorem~1.17]{Mordukhovich2006I}. Therefore, the right-hand side of \eqref{e:disc-trans-full} passes to $\lm\,\partial\varphi(\ox(\T),\T)+N_{\Xi_x\times\Xi_T}(\ox(\T),\T)$.

Combining the three slots above, we pass from \eqref{e:disc-trans-full} to \eqref{NOC:trans} and thus verify {\bf(iv)} in Theorem~\ref{ThNOC}.

\medskip
\noindent
\textbf{Step~7}. \emph{Measure support relationships in Theorem~{\rm\ref{ThNOC}}{\bf(v)}.}\\
Fix an index $i$ and a point $t\in[0,\T]$ with $g_i(\ox(t)-\ou(t))>0$. By continuity, there exists a relatively open neighborhood $V_t$ of $t$ in $[0,\T+\ve]$ on which $g_i(\ox(s)-\ou(s))>0$ (with $\ox,\ou$ extended constantly beyond $\T$). The neighborhood is one-sided only at $t=0$. The uniform convergence $(\ox^k,\ou^k)\to(\ox,\ou)$ and the equicontinuity of the discrete arcs, which are bounded in $W^{1,2}$, give us $g_i(\ox^k(t^k_j)-\ou^k(t^k_j))>0$ whenever $t^k_j\in V_t$, for all large $k$. The coderivative descriptions in \eqref{e:prop-DNR} and \eqref{e17} place this inactive index $i$ in $I_1(\beta^k_j)$, and \eqref{d11} yields $\gg^k_{ji}=0$ for such $j$. Therefore, we have $\int_V d\|\gg^k_i\|=0$ for every Borel subset $V$ of $V_t$, and the lower semicontinuity of the total variation on open sets under weak$^*$ convergence gives us $\|\gamma_i\|(V_t)=0$. Consequently, we get
$$
\supp(\gamma_i)\subset\big\{t\;\big|\;g_i(\ox(t)-\ou(t))=0\big\}.
$$
The same device applied with $\nu_i$ in place of $g_i$ ensures the support of $\psi^u_i$ in $\{t\;|\;\nu_i(\ou(t))=0\}$.

\medskip
\noindent\textbf{Step~8}. \emph{Conditional coderivative selection: proof of Proposition~\ref{prop:coderiv-selection}.}
Assume the multiplier stability in {\bf(MS1)} and {\bf(MS2)} of Definition~\ref{def:MS}. Rewrite the discrete relationships in \eqref{d11} and \eqref{d12} through the coderivative formula of Theorem~\ref{Prop1} as the cellwise inclusion
\[
\gg^k(t)\in D^*N_{\R^m_-}\bigl(-g(z^k(t)),\,\eta^k(t)\bigr)\bigl(\beta^k(t)\bigr)
\quad\;\mbox{with}\quad z^k(t):=\ox^k\big(\vartheta^k(t))-\ou^k(\vartheta^k(t)\big)
\]
for $t\in(t^k_j,t^k_{j+1})$, $j=0,\ldots,k-1$, where $\beta^k=(\beta^k_1,\ldots,\beta^k_m)$ is the discrete pairing \eqref{e:beta-def}. Due to \eqref{e:prop-DNR} and \eqref{e17}, this inclusion is written componentwise as
\begin{equation}\label{e:sign-law}
\gg^k_i(t)=0\ \text{ if }\ g_i(z^k(t))>0\ \text{ or }\ \bigl[\eta^k_i(t)=0,\ \beta^k_i(t)<0\bigr],\qquad
\gg^k_i(t)\ge 0\ \text{ if }\ g_i(z^k(t))=0,\ \eta^k_i(t)=0,\ \beta^k_i(t)>0
\end{equation}
with $\gg^k_i(t)$ being unrestricted in the remaining cases. Note that $g_i(z^k(t))\ge 0$ throughout by the dynamics \eqref{e:Pk-dyn} since $N_C(\ox^k_j-\ou^k_j)=\emptyset$ outside $C$. Moreover, it follows from \eqref{d15} that $\eta^k_i(t)>0$ yields $\beta^k_i(t)=0$. Hence, wherever $\beta^k_i(t)\neq 0$, the alternative $\eta^k_i(t)=0$ in \eqref{e:sign-law} is automatic.

{\bf 8a}. {\em Absolute continuity of $\gamma$ and identification of $D_\gamma$.} By {\bf(MS2)}, there exists a modulus $\epsilon(\cdot)$ with $\epsilon(\delta)\downarrow0$ as $\delta\downarrow0$ such that $\sup_k|\gamma^k|(B)\le\epsilon(|B|)$ for every Borel set $B\subset{[0,\T+\ve]}$. The total variation is weak$^*$ lower semicontinuous on open sets, which readily yields the inequalities
$$
|\gamma|(G)\le\liminf_k|\gamma^k|(G)\le\epsilon(|G|)
$$ 
for every relatively open set $G{\subset[0,\T+\ve]}$. For a Borel set $B\subset[0,\T]$ and any $\rho>0$, the outer regularity of $|\gamma|$ provides an open set $G\supset B$ with $|G|\le|B|+\rho$, and hence $|\gamma|(B)\le\epsilon(|B|+\rho)$. Since $\epsilon(\delta)\downarrow0$ as $\delta\downarrow0$, it follows that $|\gamma|(B)\to 0$ as $|B|\to0$. Therefore, $\gamma\ll\mathrm{Leb}$ (Lebesgue measure) with the density $h:=\frac{d\gamma}{dt}\in L^1([0,\T];\R^m)$), and at each Lebesgue point $t$ of $h$ (i.e., almost everywhere in $[0,\T]$) every interval sequence $S_p\ni t$ with $|S_p|\downarrow 0$ satisfies the condition $\gamma(S_p)/|S_p|\to h(t)$ as $p\to\infty$. Thus we get $D_\gamma(t)=\{h(t)\}$ for a.e.\ $t\in[0,\T]$.

{\bf 8b}. {\em Weak convergence of measures in $L^1$.} By {\bf (MS2)} and the Dunford--Pettis theorem, the densities $\{\gg^k\}$ are relatively weakly compact in $L^1([0,\T];\R^m)$. Along a subsequence (no relabeling), we get that $\gg^k\rightharpoonup\widetilde h$ weakly in $L^1$. This allows us, by taking into account the weak$^*$ convergence $\gamma^k\to\gamma$, to identify $\widetilde h=h$ a.e. Using {\bf(MS1)}, we have without loss of generality that $\beta^k_i\to\beta_i$ a.e.\ on $\{\beta_i\ne 0\}$ for every $i$. By Mazur's theorem, for every $N\in\N$ there exists a convex combination $v^N=\sum_{k\in F_N}c^N_k\,\gg^k$ with $F_N\subset\{N,N+1,\ldots\}$ and $\|v^N-h\|_{L^1}\le N^{-1}$. This ensures that $v^{N_j}(t)\to h(t)$ a.e.\ on $[0,\T]$ along a subsequence $N_j\to\infty$.

{\bf 8c}. {\em Componentwise sign passage.} Fix $i\in\{1,\ldots,m\}$, and let $t\in\{\beta_i>0\}$ be a point at which $\beta^k_i(t)\to\beta_i(t)$ and $v^{N_j}(t)\to h(t)$ for a.e.\ $t\in[0,\T]$. For all $k$ beyond a threshold $K(t)$, we have $\beta^k_i(t)>0$ enduring that $\eta^k_i(t)=0$ by \eqref{d15}. Moreover, the sign law \eqref{e:sign-law} tells us that $\gg^k_i(t)\ge0$. Hence $v^N_i(t)\ge 0$ for every $N\ge K(t)$, and in the limit $h_i(t)$  is nonnegative a.e.\ on $\{\beta_i>0\}$. The same device applied on $\{\beta_i<0\}$ gives us $\gg^k_i(t)=0$ for all large $k$. Thus both alternatives of \eqref{e:sign-law} yield the zero value therein, and so $h_i(t)=0$ a.e.\ on $\{\beta_i<0\}$. 
Finally, at any $t$ with $g_i(\ox(t)-\ou(t))>0$, the uniform convergence of $z^k\to\ox-\ou$ ensures that $g_i(z^k(\tau))>0$ for all $\tau$ near $t$ and all large $k$. This implies that the sequence $\{\gg^k_i\}$ vanishes in the limit near $t$ and that $h_i(t)=0$, which recovers therefore the support relationship of Step~{\bf 7} at the density level.

{\bf 8d}. {\em Conclusion of Step~{\bf 8}.} It follows from Step~{\bf 8c} that for a.e.\ $t\in[0,\T]$ we get 
\[
h_i(t)=0\quad\text{if }i\in I_1\bigl(\beta(t)\bigr),\qquad h_i(t)\ge0\quad\text{if }i\in I_2\bigl(\beta(t)\bigr)
\]
with the index sets of \eqref{e17} evaluated at $\bar x=-g(\ox(t)-\ou(t))$, $\bar v=\eta(t)$, $\bar y=\beta(t)$. The domain condition of \eqref{e:prop-DNR} holds a.e.\ because the second complementary-slackness relation in \eqref{NOC:cs} yields $\eta_i(t)\,\beta_i(t)=0$ a.e. Hence
\[
h(t)\in D^*N_{\R^m_-}\bigl(-g(\ox(t)-\ou(t)),\,\eta(t)\bigr)\bigl(\beta(t)\bigr)\quad\text{for a.e.\ }t\in[0,\T].
\]
Since $D_\gamma(t)=\{h(t)\}$ a.e.\ by Step~{\bf 8a}, the intersection in \eqref{NOC:coderiv} is nonempty for a.e.\ $t\in[0,\T]$. This verifies Proposition~\ref{prop:coderiv-selection} and thus completes the proof in Step~{\bf 8}.

\medskip
\noindent\textbf{Step~9}. \emph{Nontriviality conditions in \eqref{NOC:NT}, \eqref{NOC:NT-meas}, and \eqref{NOC:NT-enh}.}

{\bf 9a}. {\em Standard nontriviality.} We first show that the limiting multiplier collection $(\lm,p,q,\gamma,\psi^u)$ is nontrivial by passing the discrete normalization \eqref{e:dnc} to the limit 
term-by-term along the subsequence fixed in Steps~{\bf 1--8}. This point requires some care since the 
total variation and BV boundary terms in \eqref{e:dnc} are only weak$^*$ lower semicontinuous, so a priori mass could be lost in the limit. We verify that no term loses mass.

The cost multipliers clearly converges $\lm^k\to\lm$ along a subsequence. The absolutely continuous adjoint endpoint vectors $\|p^k(\T^k)\|\to\|p(\T)\|$ converge by the uniform convergence $p^k\to p$ on $[0,\T+\ve]$ established in Step~{\bf 4f}, the equicontinuity of $\{p^k\}$, and $\T^k\to\T$. The left endpoint values converge, $\|q^{uk}(0)\|\to\|q^u(0)\|$, since $q^u(0)$ is defined as the limit of $q^{uk}(0)$ along the fixed subsequence; see Step~{\bf 4f}. We do not use the right-continuous representative at this point.

For the control constraint mass, we use the \emph{nonnegativity} condition $\psi^{uk}_i\ge 0$, which ensures by \eqref{e:dnc} and the weak$^*$ convergence $\psi^{uk}_i\stackrel{w^*}\to\psi^u_i$that the first and last equalities in
\[
\|\psi^{uk}_i\|_{TV}=\langle\mathbf 1,\psi^{uk}_i\rangle\longrightarrow\langle\mathbf 1,\psi^u_i\rangle=\|\psi^u_i\|_{TV}
\]
hold because the measures are nonnegative. Passing to the limit in \eqref{e:dnc} yields therefore the equality
\begin{equation}\label{e:NT-limit}
\lm+\|p(\T)\|+\|q^u(0)\|+\sum_{i=1}^s\|\psi^u_i\|_{TV}=1
\end{equation}
whose left-hand side is the standard nontriviality functional \eqref{NOC:NT}. Thus the limiting collection is nontrivial.

{\bf 9b}. {\em Arutyunov form  of nontriviality}. Contrary to \eqref{NOC:NT-meas}, suppose that $\lm=0$, $\|p(\T)\|=0$, $\|\gamma\|_{TV}=0$, and $\|\psi^u\|_{TV}=0$. We intend to show that $q^u(0)=0$, which contradicts the standard nontriviality \eqref{NOC:NT} of Step~{\bf 9a}.

{\bf(i)} {\em Vanishing $p^u$ on $[0,\T]$.} The condition $\|p(\T)\|=0$ gives us $p^x(\T)=p^u(\T)=0$. The adjoint equation \eqref{NOC:adj-ac} for $p^u$ with $\lm=0$ reduces to $\dot p^u(t)\equiv 0$ a.e.\ on $[0,\T]$, and the absolute continuity $p^u\in W^{1,2}$ together with $p^u(\T)=0$ yields $p^u\equiv 0$. In particular, $p^u(0)=0$.

{\bf(ii)} {\em Vanishing $q^x$ and the Hessian coupling vector $\xi$.} The condition $\|\gamma\|_{TV}=0$ forces $\gamma\equiv 0$ as a vector measure on $[0,\T]$. Substituting $\lm=0$ and $\gamma=0$ into \eqref{NOC:qx} reduces it to the backward Volterra integral equation
\begin{equation}\label{e:9b-volterra}
q^x(t)=p^x(t){+}\int_t^{\T}K(s)\,q^x(s)\,ds,\qquad K(s):=\sum\limits_{i=1}^m\eta_i(s)\,\nabla^2 g_i\bigl(\ox(s)-\ou(s)\bigr),
\end{equation}
while the condition \eqref{NOC:adj-ac} for $p^x$ with $\lm=0$ becomes the equation
\begin{equation}\label{e:9b-pdot}
\dot p^x(t)={-}\nabla_x f\bigl(\ox(t),\oa(t)\bigr)^*\,q^x(t)\quad\text{a.e.\ on }[0,\T],\qquad p^x(\T)=0.
\end{equation}
The gradient bound \eqref{e12} yields $\|K(s)\|\le M_3\sum_i\eta_i(s)$, and Step~2 gives us $\eta\in L^2$, and thus $K\in L^2([0,\T])\subset L^1([0,\T])$. Combining \eqref{e:9b-volterra} and \eqref{e:9b-pdot} with $\|\nabla_x f\|\le L_f$ from \textbf{(H1)} tells us that 
\[
\|p^x(t)\|\le\int_t^{\T}L_f\,\|q^x(s)\|\,ds,\qquad
\|q^x(t)\|\le\|p^x(t)\|+\int_t^{\T}\|K(s)\|\,\|q^x(s)\|\,ds\;\mbox{ for a.e. }\;t\in[0,\T],
\]
which brings us therefore to the estimate
\[
\|q^x(t)\|\le\int_t^{\T}\bigl(L_f+\|K(s)\|\bigr)\,\|q^x(s)\|\,ds.
\]
Since $L_f+\|K\|\in L^1([0,\T])$, the backward Gr\"onwall inequality yields $q^x\equiv 0$, and then \eqref{e:9b-pdot} with $p^x(\T)=0$ gives us $p^x\equiv 0$. Consequently, the Hessian coupling vector
\[
\xi(t):=\Bigl(\sum\limits_{i=1}^m\eta_i(t)\,\nabla^2 g_i(\ox(t)-\ou(t))\Bigr){\bigl(q^x(t)-\lm v^x(t)\bigr)}
\]
vanishes identically on the entire interval $[0,\T]$.

{\bf(iii)} {\em Boundary value $q^u(0)$.} Setting $t=0$ in the identity \eqref{NOC:puqu} of {\bf(ii)} in Theorem~\ref{ThNOC} established in Step~{\bf 4} and substituting $p^u(0)=0$ from~{\bf(i)}, $\xi\equiv 0$ from~{\bf(ii)}, $\gamma\equiv 0$, and $\psi^u\equiv 0$ from the contradiction hypothesis yield
\[
q^u(0)=p^u(0)-0-0-0=0.
\]
Combined with $\lm=0$, $\|p(\T)\|=0$, and $\|\psi^u\|_{TV}=0$, the equality $q^u(0)=0$ contradicts the standard nontriviality condition \eqref{NOC:NT} of Step~{\bf 9a}. Therefore, \eqref{NOC:NT-meas} holds.

{\bf 9c.} {\em Enhanced nontriviality.} Assume that $g_i(x_0-\ou(0))>0$ for all $i$, and suppose while arguing by contradiction that $\lm=0$, $p(\T)=0$, and $\|\psi^u\|_{TV}=0$. Using the continuity of $g_i$ gives us $g_i(\ox^k_0-\ou^k_0)>0$ for all large $k$. The discrete recursion \eqref{d6} combined with \eqref{d10} and \eqref{d11} (which force $\eta^k_0=0$ and $\gg^k_{0i}=0$ for all $i$) yields
\[
p^{uk}_1-p^{uk}_0=\lm^k h_k w^{uk}_0+\sum\limits_{i=1}^s\psi^{uk}_{0i}\,\nabla\nu_i(\ou^k_0).
\]
Using now \eqref{d8}, \eqref{e:CV}, and \eqref{e:var} for $\|p^{uk}_1\|$ leads us to the estimates 
\begin{equation*}
\begin{aligned}
\|q^{uk}(0)\|=\|p^{uk}_0\|
&\le\|p^{uk}_1\|+\lm^k h_k\|w^{uk}_0\|+L_\nu\sum\limits_{i=1}^s|\psi^{uk}_{0i}|\\
&\le\lm^k\|v^{uk}_0\|+\lm^k({2}\mu+{C_V})+\lm^k h_k\|w^{uk}_0\|+L_\nu\sum\limits_{i=1}^s\|\psi^{uk}_i\|_{TV},
\end{aligned}
\end{equation*}
where each term on the right-hand side tends to zero under the hypotheses $\lm=0$, $p(\T)=0$, and $\|\psi^u\|_{TV}=0$. Hence  $q^u(0)=\lim_{k\to\infty}q^{uk}(0)=0$, which contradicts \eqref{NOC:NT} and thus justifies \eqref{NOC:NT-enh}.

\medskip
Steps~{\bf 1--7} and~{\bf 9} collectively establish the subgradient inclusion \eqref{NOC:subg}, conclusions {\bf (i)--(vi)} of Theorem~\ref{ThNOC}, and the three nontriviality forms \eqref{NOC:NT}, \eqref{NOC:NT-meas}, and \eqref{NOC:NT-enh}. Step~{\bf 8} verifies the conditional coderivative selection of Proposition~\ref{prop:coderiv-selection}, which is logically separate from the unconditional conclusions. This completes the proof of the necessary optimality condition for the controlled sweeping process under consideration. $\h$

\section{Maximization Conditions and Related Results}\label{S6b}

This section presents Hamiltonian maximization conditions and related results for the controlled sweeping problem $(P)$ that are consequences of the main ones given in Theorem~\ref{ThNOC} and Proposition~\ref{prop:coderiv-selection}. We start with the following sweeping control version of the {\em Pontryagin maximum principle} (PMP) for controls in the perturbed dynamics. In this corollary, the notation $T_A(\oa)$ stands for an arbitrary {\em tangential cone approximation} of the set $A$ at $\oa\in A$ with the only property that $T_A(\oa)$ reduces to the classical tangent cone of convex analysis whenever $A$ is convex. In particular, $T_A(\oa)$ can be taken as the (Bouligand-Severi) {\em contingent cone}; see, e.g., \cite{Mordukhovich2006I}

\begin{corollary}[tangential maximum principle for controls in the dynamics]\label{cor:SweepingMP}
Under the assumptions of 
Theorem~{\rm\ref{ThNOC}}, suppose in addition that the normal cone at $\oa(t)$ is the $($negative$)$ polar of the tangent cone
\begin{equation}\label{polar}
N_A\bigl(\oa(t)\bigr)=\bigl(T_A\bigl(\oa(t)\bigr)\bigr)^{\circ}\;\mbox{ for a.e. }\;t\in[0,\T].
\end{equation}
Then the $a$-control multiplier $\psi^a$ from \eqref{NOC:psi-a} satisfies the {\sc tangential maximization condition}
\begin{equation}\label{cor:Tmax}
\max_{v\in T_A\big(\oa(t)\big)}\la\psi^a(t),v\ra=0\;\mbox{ for a.e. }\;t\in[0,\T].
\end{equation}
If in addition the set $A$ is convex, then the {\sc global maximization condition}
\begin{equation}\label{cor:Gmax}
\la\psi^a(t),\oa(t)\ra=\max_{a\in A}\la\psi^a(t),a\ra\;\text{ for a.e.\ }t\in[0,\T]
\end{equation}
holds, which is equivalently written as $\la\psi^a(t),a-\oa(t)\ra\le 0$ for all $a\in A$.
\end{corollary}
{\bf Proof}. Theorem~\ref{ThNOC}{\bf(iii)} tells us that $\psi^a(t)\in\co\,N_A(\oa(t))$ for a.e.\ $t\in[0,\T]$. Since $(T_A(\oa(t)))^{\circ}$ is a convex cone independently of the convexity of the tangential approximation $T_A(\oa(t))$, the polarity relationship \eqref{polar} makes the convexification unnecessary, i.e., 
$$
\co\,N_A\big(\oa(t)\big)=N_A\big(\oa(t)\big)=\big(T_A(\oa(t))\big)^{\circ}.
$$
This readily ensures that $\psi^a(t)\in(T_A(\oa(t)))^{\circ}$, and so 
$$
\la\psi^a(t),v\ra\le 0\;\mbox{ for all }\;v\in T_A\big(\oa(t)\big).
$$
Since $T_A(\oa(t))$ is a cone containing the origin, the supremum of $v\mapsto\la\psi^a(t),v\ra$ over $T_A(\oa(t))$ equals $0$ and is attained at $v=0$; this is the {\em tangential stationarity} \eqref{cor:Tmax}. When $A$ is convex, we get that $a-\oa(t)\in T_A(\oa(t))$ for every $a\in A$, which verifies  \eqref{cor:Gmax} and thus completes the proof.$\h$

Observe that the global maximum condition \eqref{cor:Gmax} is {\em linear} with respect to the multiplier $\psi^a$. This can be used to upgrade the $a$-maximization condition via the reduced Hamiltonian. To proceed, recall---along with the (lower) subdifferential \eqref{sub}--- its {\em upper subdifferential} (superdifferential) counterpart
\begin{equation}\label{e:superdiff}
\partial^+\phi(\ox):=-\partial(-\phi)(\ox)\;\mbox{ whenever }\;|\phi(\ox)|<\infty
\end{equation}
When $\phi$ is concave, \eqref{e:superdiff} agrees with the superdifferential of convex analysis: $\zeta\in\partial^+\phi(\ox)$ if and only if
\[
\phi(x)\le\phi(\ox)+\la\zeta,\,x-\ox\ra\qquad\text{for all }\;x\in\R^n.
\]
If $\phi$ is strictly differentiable at $\ox$ (in particular, $C^1$-smooth around $\ox$), then $\partial^+\phi(\ox)=\{\nabla\phi(\ox)\}$. For functions of several variables, the partial subdifferential constructions $\partial^+_u,\partial^+_a$ are defined accordingly.

\begin{corollary}[reduced Hamiltonian maximum for controls in the dynamics]\label{cor:aHamMax} Under the assumptions of Theorem~{\rm\ref{ThNOC}}, suppose that $A$ is closed and convex, that the {\sc reduced Hamiltonian}
\begin{equation}\label{e:Ha-def}
\mathcal H_a(t,a):=\la{q^x(t)-\lm v^x(t)},\,f(\ox(t),a)\ra-\lm\,\ell\bigl(t,\ox(t),\ou(t),a,\dot{\ox}(t),\dot{\ou}(t)\bigr)
\end{equation}
is concave in $a$ on $A$ for a.e. $t\in[0,\T]$, and that the selected subgradient of the running cost satisfies the inclusion
\begin{equation}\label{reduce-max}
-\lm w^a(t)\in\partial_a^+\Big[-\lm\,\ell\bigl(t,\ox(t),\ou(t),\cdot,\dot{\ox}(t),\dot{\ou}(t)\bigr)\Big]\bigl(\oa(t)\bigr),
\end{equation}
Both requirements hold if $f$ is affine and $\ell$ is convex in the $a$-variable. 
Then we have the {\sc reduced maximization condition}
\begin{equation}\label{max-hamil}
\oa(t)\in\arg\max_{a\in A}\,\mathcal H_a(t,a)\qquad\text{for a.e.\ }\;t\in[0,\T].
\end{equation}
\end{corollary}
{\bf Proof}. Fix $t\in[0,\oT]$ at which \eqref{NOC:psi-a} holds along with the assumptions imposed above. It follows from the convexity of $A$ that $\psi^a(t)\in\co N_A(\oa(t))=N_A(\oa(t))$. The $f$-summand of \eqref{e:Ha-def} is smooth in $a$ with the gradient $\nabla_af(\ox(t),\oa(t))^*(q^x(t)-\lm v^x(t))$. Then the inclusion in \eqref{reduce-max} combined with the sum rule for the upper subdifferential with a smooth summand leads us to
\[
\psi^a(t)=\nabla_af(\ox(t),\oa(t))^*{\bigl(q^x(t)-\lm v^x(t)\bigr)}-\lm w^a(t)\in\partial^+_a\mathcal H_a\bigl(t,\oa(t)\bigr)\cap N_A(\oa(t)),
\]
{Since $\mathcal H_a(t,\cdot)$ is concave, this upper subgradient is a supergradient in the sense of convex analysis. A supergradient lying in $N_A(\oa(t))$ ensures that $\oa(t)$ maximizes $\mathcal H_a(t,\cdot)$ over the convex set $A$. This verifies} \eqref{max-hamil}. $\h$\vspace*{0.05in}

\section{Major Specifications and Further Relationships}\label{S6c}

This section relates Theorem~\ref{ThNOC} to known results. In appropriate special regimes, the optimality system in Theorem~\ref{ThNOC} agrees with the known results as discussed in the next two remarks; namely, for the convex polyhedral free-time conditions of \cite{ColomboMordukhovichNguyen2024} (see Remark~\ref{rem:polyhedral}) and for the fixed-time convex conditions of \cite{CaoColomboMordukhovich2021a,CaoColomboMordukhovich2021b} (see Remark~\ref{rem:fixed-time-convex}).

\begin{remark}[specializations in the convex polyhedral setting]\label{rem:polyhedral}
The free-time {\em convex polyhedral} sweeping control problem from Colombo--Mordukhovich--Nguyen--Nguyen \cite{ColomboMordukhovichNguyen2020,ColomboMordukhovichNguyen2024} (CMNN), with the moving polyhedron $C(t)=\bigcap_{j=1}^s\{\la x^*_j(t),x\ra\le c_j(t)\}$ and the Mayer cost $\varphi(x(\T),\T)$, contains a subclass that is a polyhedral specialization of $(P)$; namely, the subclass in which the polyhedron moves by translations, $C(t)=C+u(t)$, with $g_i(y)=c_i-\la b_i,y\ra$. The two frameworks are different in:

{\bf(i)} {\em Geometry}. The moving set in CMNN is a convex polyhedron, while Theorem~\ref{ThNOC} allows a nonconvex prox-regular sweeping set of the type $C=\bigcap_i\{g_i\ge0\}$.

{\bf(ii)} {\em Time variation}. CMNN varies facet orientations and offsets directly in $t$, while in Theorem~\ref{ThNOC} the sweeping set is shifted by the control $C_u(t)=C+u(t)$.

{\bf(iii)} {\em Data regularity}. CMNN imposes the Lipschitzian assumptions on $\varphi$ and $f$, while Theorem~\ref{ThNOC} requires $\varphi$ to be locally Lipschitz around the optimal endpoint, $\ell$ to be Lipschitz uniformly in $\dot x$ as in assumption {\bf(c)}, and $f$ to be $C^1$-smooth. When $C$ is a fixed polyhedron ($\nabla^2 g_i\equiv0$), $\ell\equiv0$, and ${u\equiv 0}$, Theorem~\ref{ThNOC} agrees with the results in CMNN, i.e., the Hessian coupling $\xi$ vanishes, the primal, adjoint, and transversality relationships reduce to those in \cite[Theorem~5.1]{ColomboMordukhovichNguyen2024}, and the three nontriviality forms of conclusion~{\bf(vi)} reduce to the single condition $\lm+\|p(\T)\|>0$. Outside this regime, the two ingredients have {\em no CMNN analogue}. They are: the second-order {\em measured coderivative} selection \eqref{NOC:coderiv} in Proposition~\ref{prop:coderiv-selection} carried by the Hessian coupling tensor $Q_\eta$ (which vanishes in the polyhedral case), and the {\em unconditional Arutyunov form of nontriviality} formulated in \eqref{NOC:NT-meas}.
\end{remark}

\begin{remark}[specializations in the fixed-time convex setting]\label{rem:fixed-time-convex}
Imposing $\Xi_T=\{\T\}$ and the convexity of $C$ recovers the fixed-time convex conditions of \cite{CaoColomboMordukhovich2021a,CaoColomboMordukhovich2021b}. In this setting, the free-time slot for $\mathcal H$ in conclusion~{\bf(iv)} of Theorem~\ref{ThNOC} becomes unconstrained and the free-endpoint identity disappears, the sweeping measure $\gamma$ in conclusion~{\bf(v)} is supported on the active set, and the nontriviality condition \eqref{NOC:NT} recovers the form $\lm+\|p(\T)\|+\|\psi^u\|_{TV}>0$ from  \cite{CaoColomboMordukhovich2021b}. The discrete structure of \cite[\S7]{CaoMordukhovich2026} is a natural bridge; namely, fixing $\T^k\equiv\T$ suppresses both the $\lm^k(\T-\T^k)$ term of the discrete transversality \cite[eq.~(7.24)]{CaoMordukhovich2026} and the variable mesh residual $\varrho^k$ giving us the discrete system of \cite{CaoColomboMordukhovich2021a}.
\end{remark}\vspace*{-0.05in}

Let us now consider specifications of the obtained results in the case of {\em polyhedral} sweeping sets $C$, i.e., when $g_i(x)=\la a_i,x\ra-b_i$. {Here and below, the polyhedral case means that the functions $g_i$ describing $C$ are affine. A polyhedral set may also be described by nonaffine functions, e.g., $[0,\infty)=\{z\mid z+\epsilon(1-\cos z)\ge0\}$ with $0<\epsilon<1$, and then $\nabla^2g_i$ need not vanish. In the polyhedral case,} we have
\[
\xi^k_j=\Bigl(\textstyle\sum_i\eta^k_{ji}\nabla^2 g_i\Bigr)\Lm^k_j=0,
\]
so the curvature term in \eqref{NOC:qx} drops from the adjoint system{. In \cite{ColomboMordukhovichNguyen2024}, the sweeping measure satisfies a sign condition on the active set. We do not claim such a sign here. It need not hold for every collection of multipliers satisfying the conclusions of Theorem~\ref{ThNOC}, even in one dimension with an affine $g$. A sign could hold only for a suitable choice of multipliers. In the classical theory of state constraints, endpoint degeneracy is excluded by polyhedral inward pointing conditions; see Arutyunov--Aseev \cite{ArutyunovAseev1997} and Lopes--Fontes--de Pinho \cite{LopesFontesPinho2011}. The endpoint results of Section~\ref{S9b} need neither polyhedrality nor such a condition.} In the general nonconvex prox-regular case, the curvature terms reenter and the measure $\gamma$ becomes sign indefinite.

\section{Endpoint Nondegeneracy}\label{S9b}
\setcounter{equation}{0}

This section is devoted to the nondegeneracy of the optimality conditions at the endpoints of the time interval. Among the three forms of nontriviality in conclusion~{\bf(vi)} of Theorem~\ref{ThNOC}, the enhanced relationship \eqref{NOC:NT-enh} is the strongest. It excludes the degenerate multipliers whose entire nontrivial mass concentrates in a boundary atom of the sweeping measure $\gamma$. However, it holds only under the interior initial data hypothesis $g_i(x_0-\ou(0))>0$. The normal cone dynamics in controlled sweeping processes typically drives trajectories to the boundary of $C_u(t)$, so this hypothesis is restrictive. The results below remove it for suitably selected multipliers.

Without this hypothesis, the following phenomenon may occur. The boundary contact mechanism is a free-time sweeping version of the classical Gamkrelidze--Arutyunov endpoint degeneracy phenomenon (see \cite{ArutyunovAseev1997,Arutyunov2000,Vinter2000}): a {\em Dirac mass} of the sweeping measure $\gamma$ at the active endpoint time $\tau\in\{0,\T\}$ saturates the standard nontriviality condition \eqref{NOC:NT} {\em on its own}, with no information transmitted to the optimal arc on $(0,\T)$.

{In the classical theory of state constraints, this degeneracy is excluded by a {\em controllability type qualification} at the endpoints; see \cite{ArutyunovAseev1997,Arutyunov2000,FerreiraFontesVinter1999,LopesFontesPinho2011,ArutyunovKaramzinPereira2011}. Such a condition asks for an admissible control that induces a velocity pointing strictly into the constraint set along every active gradient. This first-order program has been extended to equality type state constraints in \cite{ArutyunovKaramzin2016} and to higher-order Lie bracket relaxations in \cite{ColomboKhalilRampazzo2022}. Arroud and Colombo \cite{ArroudColombo2018} used such a pointing condition in a maximum principle for a controlled sweeping process with a smooth and strictly convex moving set. Our results do not need such a condition. For the sweeping dynamics with the moving set control, the degeneracy is excluded by a selection of multipliers at the first mesh cell.}

For $\tau\in\{0,\T\}$, let $I_\tau:=I(\ox(\tau)-\ou(\tau))$ stand for the {\em active set at the endpoint} $\tau$.

{Put $y_0:=x_0-\ou(0)$ and $I_0:=I(y_0)$. If $I_0\ne\varnothing$, put
\[
\kappa_0:=\min_{\|z\|_1=1}\Bigl\|\sum_{i\in I_0}z_i\nabla g_i(y_0)\Bigr\|,
\]
which is positive by the LICQ in assumption {\bf(a)} of Theorem~\ref{ThNOC}. The following theorem bounds the boundary value $q^u(0)$ and the initial atom of the sweeping measure for the multipliers selected at the first mesh cell.

\begin{theorem}[initial value and initial atom]\label{thm:initial-atom}
Let the assumptions of Theorem~{\rm\ref{ThNOC}} hold, and build the multipliers in the proof of Theorem~{\rm\ref{ThNOC}} from the discrete multipliers of Lemma~{\rm\ref{lem:first-cell}}. Then all the conclusions of Theorem~{\rm\ref{ThNOC}} hold, and moreover we have
\begin{itemize}
\item[\bf(i)] $\|q^u(0)\|\le C_1\lm$ with $C_1:=L_{\ell_2}(1+C_V)+C_V+2\mu$, where $C_V$ is taken from \eqref{e:CV};
\item[\bf(ii)] if $I_0\ne\varnothing$, then the initial atom of $\gamma$ satisfies the estimate
\begin{equation}\label{e:disc-CQ-bd}
\sum_{i\in I_0}\bigl|\gamma_i(\{0\})\bigr|\le M_0\Bigl[\lm+\sum_{j=1}^s\psi^u_j(\{0\})\Bigr]\quad\mbox{with}\quad M_0:=\kappa_0^{-1}\max\bigl\{C_1+L_{\ell_2}(1+2\mu),\,L_\nu\bigr\},
\end{equation}
and if $I_0=\varnothing$, then $\gamma(\{0\})=0$.
\end{itemize}
\end{theorem}
{\bf Proof}. By Lemma~\ref{lem:first-cell}(d), the left-hand side of \eqref{e:dnc} is positive, so we may normalize as in \eqref{e:dnc}. Then Steps~{\bf 1}--{\bf 9} of the proof of Theorem~\ref{ThNOC} apply to these multipliers. When $A$ is convex and the last constraint in \eqref{e:Pk-bv} is active, Step~{\bf 4e} is replaced by the variational inequality device of Step~{\bf 5}. Moreover, it follows from Step~{\bf 4f} that $q^u(0)=\lim_kp^{uk}_0$.

To verify {\bf(i)}, we use Lemma~\ref{lem:first-cell}(c) and \eqref{d8}. This gives us
\[
p^{uk}_0=\lm^k\Bigl(v^{uk}_0+\frac{\th^{Uk}_0}{h_k}\Bigr)-h_k\lm^kw^{uk}_0-h_k\xi^k_0 .
\]
By Step~{\bf 4a}, we have $\|v^{uk}_0\|\le L_{\ell_2}(1+C_V)$ and $\|\th^{Uk}_0\|/h_k\le C_V+2\mu$. Moreover, $h_k\lm^k\|w^{uk}_0\|\to0$. It remains to show that $h_k\xi^k_0\to0$. Put $r_k:=\|\th^{Xk}_0\|/h_k$. Then $h_kr_k^2\le\|\th^{Xk}\|^2_{L^2}\to0$ by \eqref{e:theta-zero}. By \eqref{e:eta-bound}, we have $\sum_i\eta^k_{0i}\le c(1+r_k)$. By \eqref{e:pxk-bdd}, we have $\|\Lm^k_0\|\le c(1+r_k)$. Here $c$ denotes constants independent of $k$. Since the Hessians of $g_i$ are bounded near $y_0$, we obtain the estimate
\[
h_k\|\xi^k_0\|\le c\,h_k(1+r_k)^2\le 2c\bigl(h_k+h_kr_k^2\bigr)\to0 .
\]
Passing to the limit as $k\to\infty$ gives us (i).

To verify {\bf(ii)}, we compare \eqref{NOC:puqu} at $t=0$, where $J_0=[0,\T]$, with its limit as $t\downarrow0$, where $J_t=(t,\T]$. This gives us the trace identity
\[
\sum_{i\in I_0}\gamma_i(\{0\})\nabla g_i(y_0)=q^u(0+)-q^u(0)-\sum_{j=1}^s\psi^u_j(\{0\})\nabla\nu_j(\ou(0)).
\]
By \eqref{NOC:qu}, the bounded variation of $\dot\ou$, and \eqref{l2-est}, we have $\|q^u(0+)\|\le\lm L_{\ell_2}(1+2\mu)$. By the support relationship of Theorem~\ref{ThNOC}{\bf(v)}, we have $\gamma_i(\{0\})=0$ for $i\notin I_0$. Now the definition of $\kappa_0$, assertion {\bf(i)}, and \textbf{(H3.1)} yield {\bf(ii)}. $\h$

Note that the bound \eqref{e:disc-CQ-bd} does not hold for every collection of multipliers satisfying the conclusions of Theorem~\ref{ThNOC}. Indeed, the degenerate multipliers of Example~\ref{ex:CQ0} below violate it. Thus \eqref{e:disc-CQ-bd} is a property of the multipliers selected in Lemma~\ref{lem:first-cell}. Observe also that $\gg^k_0=0$ alone does not exclude an atom at $t=0$, since mass may concentrate on the cells $j=1,2,\ldots$ near $0$. The bound \eqref{e:disc-CQ-bd} comes from the trace identity and assertion (i).}

The following theorem strengthens the Arutyunov form of nontriviality \eqref{NOC:NT-meas} by removing the measure term $\|\gamma\|_{TV}$. Observe that the asymmetry between the two endpoints in the theorem below is intrinsic for the standard nontriviality condition \eqref{NOC:NT}, which carries the BV value $\|q^u(0)\|$ on the initial side and the value $\|p(\T)\|$ on the terminal side. Only the initial slot is degenerable through a sweeping atom at $\tau=0${, so the selection of Lemma~\ref{lem:first-cell} is needed only at the initial endpoint}. The terminal endpoint statement is handled at the level of the sweeping measure itself by the subsequent Theorem~\ref{ThNOCcq-atom}, which bounds the terminal atom $\gamma(\{\T\})$ by the 
nondegenerate endpoint data {without any qualification condition}.

\begin{theorem}[nondegenerate nontriviality]\label{ThNOCcq}
{Let the assumptions of Theorem~{\rm\ref{ThNOC}} be fulfilled. Then the multipliers of Theorem~{\rm\ref{thm:initial-atom}} satisfy} the enhanced nontriviality condition
\begin{equation}\label{e:NTcq}
\lm+\|p(\T)\|+\|\psi^u\|_{TV}>0
\end{equation}
without imposing the interior initial data assumption $g_i(x_0-\ou(0))>0$. 
\end{theorem}
{\bf Proof}. {Passing to the limit in \eqref{e:dnc} gives us $\lm+\|p(\T)\|+\|q^u(0)\|+\sum_{j=1}^s\|\psi^u_j\|_{TV}=1$, as in Step~{\bf 9a} of the proof of Theorem~\ref{ThNOC}. Suppose that $\lm+\|p(\T)\|+\|\psi^u\|_{TV}=0$. Then $\lm=0$, and Theorem~\ref{thm:initial-atom}{\bf(i)} yields $q^u(0)=0$. This contradicts the normalization and verifies \eqref{e:NTcq}. $\h$

\begin{example}[inward escape from an active boundary at $t=0$]\label{ex:CQ0}
We illustrate the endpoint degeneracy and its removal on a one-dimensional variable-time sweeping control problem. Take $n=m=d=1$, $C=[0,\infty)$ with $g_1(x)=x$, $U=\R$, $A=[-1,1]$, $f(x,a)=a$, $\Xi_x=\R$, $\Xi_T=[\tfrac{1}{2},2]$, terminal cost $\varphi(x,T)=(x-1)^2$, running cost $\ell\equiv 0$, and the initial data $x_0=0$, $\ou(0)=0$. We have the following.

$\bullet$ {\em Reference r.i.l.m.} The tuple $\ox(t)=t,\ \ou(t)\equiv 0,\ \oa(t)\equiv 1,\ \T=1$ is feasible with $\ox(t)-\ou(t)=t\ge 0$, $\dot{\ox}=\oa=1\in -N_C(\ox-\ou)+\oa$, and the cost value $(\ox(\T)-1)^2=0$ proving the global minimum. Hence $(\ox,\ou,\oa,\T)$ is a global optimal solution and therefore an r.i.l.m. for this problem.

$\bullet$ {\em Active set at $t=0$ and failure of the interiority assumption.} At the initial point, we have $g_1(x_0-\ou(0))=g_1(0)=0$, which shows that $I_0=\{1\}$ and the interiority condition $g_i(x_0-\ou(0))>0$ fails.

$\bullet$ {\em Degenerate multiplier under standard nontriviality.} The standard form \eqref{NOC:NT} admits the degenerate multiplier set given by
\[
\lm=0,\quad \alpha=0,\quad p(\cdot)\equiv 0,\quad \psi^u\equiv 0,\quad \gamma=c\,\delta_{\{0\}}\ \text{for some }c>0,
\]
with $\eta\equiv 0$ a.e.\ on $[0,\T]$. The integral identity \eqref{NOC:puqu} tells us that $q^u\equiv 0$ on $(0,1)$, while
\[
q^u(0)=p^u(0)-\nabla g_1(0)\cdot\gamma(\{0\})=-c\neq 0
\]
on the boundary. The full nontriviality is concentrated in the boundary atom $q^u(0)$. This is the variable-time sweeping counterpart of the Gamkrelidze--Arutyunov endpoint degeneracy phenomenon \cite{ArutyunovAseev1997,Arutyunov2000,Vinter2000}.

$\bullet$ {\em Selected multipliers.} {Theorem~\ref{ThNOCcq} provides multipliers different from this degenerate set. Indeed, the enhanced nontriviality \eqref{e:NTcq} requires that $\lm+\|p(\T)\|+\|\psi^u\|_{TV}>0$, which the degenerate triple $(\lm,p(\T),\psi^u)=(0,0,0)$ violates. In this example, the selection can be seen directly. The sampled reference process has zero discrete cost, and every term of the discrete cost is nonnegative. Hence it is the unique optimal solution of $(P_k)$. Along it, $x^k_j-u^k_j=j/k>0$ for $j\ge1$, so $\gg^k_j=0$ for $j\ge1$. Moreover, $\gg^k_0=0$ by Lemma~\ref{lem:first-cell}. Thus the selected multipliers have $\gamma=0$, and the Dirac multiplier $c\,\delta_{\{0\}}$ does not arise. The optimality conditions deliver a nondegenerate multiplier set even though the trajectory begins on the boundary.}
\end{example}

The next theorem ensures nondegeneracy by excluding the atomic mass on the boundary.

{\begin{theorem}[endpoint atom bounds]\label{ThNOCcq-atom}
Let the assumptions of Theorem~{\rm\ref{ThNOC}} be satisfied. For $\tau\in\{0,\T\}$, consider the estimate
\begin{equation}\label{e:gamma-atom}
\sum_{i=1}^m\bigl|\gamma_i(\{\tau\})\bigr|\ \le\ M_\tau\Bigl[\lm+\bigl\|p(\tau)\bigr\|+\sum_{i=1}^s\psi^u_i(\{\tau\})\Bigr].
\end{equation}
{\bf(i)} {\sc Terminal endpoint.} Every collection of multipliers satisfying the conclusions of Theorem~{\rm\ref{ThNOC}} satisfies \eqref{e:gamma-atom} at $\tau=\T$ with
\[
M_{\T}:=\kappa_{\T}^{-1}\max\bigl\{1,\,L_{\ell_2}(1+2\mu),\,L_\nu\bigr\},\qquad
\kappa_{\T}:=\min_{\|z\|_1=1}\Bigl\|\sum_{i\in I_{\T}}z_i\nabla g_i\bigl(\ox(\T)-\ou(\T)\bigr)\Bigr\|,
\]
provided that $I_{\T}\ne\varnothing$. If $I_{\T}=\varnothing$, then $\gamma(\{\T\})=0$.

{\bf(ii)} {\sc Initial endpoint.} The multipliers of Theorem~{\rm\ref{thm:initial-atom}} satisfy \eqref{e:gamma-atom} at $\tau=0$ with the constant $M_0$ from \eqref{e:disc-CQ-bd}. In particular, the endpoint atom vanishes, i.e., $\gamma(\{\tau\})=0$, whenever $\lm=0$, $p(\tau)=0$, and $\psi^u(\{\tau\})=0$.
\end{theorem}
{\bf Proof}. By the support relationship of Theorem~\ref{ThNOC}{\bf(v)}, we have $\gamma_i(\{\tau\})=0$ for $i\notin I_\tau$. Assertion {\bf(ii)} therefore follows from Theorem~\ref{thm:initial-atom}(ii), since the right-hand side of \eqref{e:disc-CQ-bd} does not exceed that of \eqref{e:gamma-atom}. To verify {\bf(i)}, let $I_{\T}\ne\varnothing$. Then $\kappa_{\T}>0$ by the LICQ in assumption {\bf(a)} of Theorem~\ref{ThNOC}. Since $\dot\ou$ is of bounded variation, it has the left limit $\dot\ou(\T-)$, and $\|\dot\ou(\T-)\|\le2\mu$ by Step~{\bf 4a} of the proof of Theorem~\ref{ThNOC}. By \eqref{NOC:qu} and the Lipschitz estimate in \eqref{l2-est}, we get $q^u(t)\to\lm\nabla_{\dot u}\ell_2(\T,\dot\ou(\T-))$ as $t\uparrow\T$ along a set of full measure. Let us pass to the limit in \eqref{NOC:puqu} along such $t$. The function $p^u$ is continuous, the integral of $\xi$ over $[t,\T]$ tends to zero, and the sets $J_t=(t,\T]$ shrink to $\{\T\}$. Hence the continuity of $\nabla g$, $\nabla\nu_j$, $\ox$, and $\ou$ gives us the identity
\[
\sum_{i\in I_{\T}}\gamma_i(\{\T\})\nabla g_i\bigl(\ox(\T)-\ou(\T)\bigr)=p^u(\T)-\lm\nabla_{\dot u}\ell_2\bigl(\T,\dot\ou(\T-)\bigr)-\sum_{j=1}^s\psi^u_j(\{\T\})\nabla\nu_j\bigl(\ou(\T)\bigr).
\]
The definition of $\kappa_{\T}$, the growth estimate in \eqref{l2-est}, and \textbf{(H3.1)} imply that
\[
\kappa_{\T}\sum_{i=1}^m\bigl|\gamma_i(\{\T\})\bigr|\le\|p^u(\T)\|+\lm L_{\ell_2}(1+2\mu)+L_\nu\sum_{j=1}^s\psi^u_j(\{\T\}).
\]
This is \eqref{e:gamma-atom} at $\tau=\T$. The final assertion follows from \eqref{e:gamma-atom}. Note that neither {\bf(i)} nor {\bf(ii)} uses the polyhedrality of $C$, a sign of $\gamma$, or an inward pointing condition. $\h$}

\begin{remark}[on the quantitative form of \eqref{e:gamma-atom}]\label{rem:atom-bound-form}
{At $\tau=\T$, estimate \eqref{e:gamma-atom} holds for every collection of multipliers. At $\tau=0$, it holds for the multipliers selected in Lemma~\ref{lem:first-cell}. Note that \eqref{e:gamma-atom} is a {\em transfer} bound, not an unconditional elimination. The sweeping measure cannot carry an atom {\em of its own} at $\tau$, but mass {\em borrowed} from the nondegenerate endpoint data $(\lm,p(\tau),\psi^u(\{\tau\}))$ may persist. The atom vanishes when the borrowed mass is absent. {At the initial endpoint, \eqref{e:disc-CQ-bd} shows that this happens when $\lm=0$ and $\psi^u(\{0\})=0$, for instance, when $\ou(0)\in{\rm int}\,U$, since then $\psi^u(\{0\})=0$ by the support condition in Theorem~\ref{ThNOC}{\bf(v)}. At the terminal endpoint, it happens when $\lm=0$, $p(\T)=0$, and $\psi^u(\{\T\})=0$.} We do not claim that $\gamma(\{\tau\})=0$ in general.}
\end{remark}

\begin{corollary}[two-sided endpoint atom control]\label{cor:NOCcq2sided}
{Let the assumptions of Theorem~{\rm\ref{ThNOC}} be satisfied. Then, for the multipliers of Theorem~{\rm\ref{thm:initial-atom}},} the endpoint atoms of the sweeping measure $\gamma\in C^*([0,\T];\R^m)$ are bounded by the nondegenerate endpoint data, i.e.,
\begin{equation}\label{2bound}
\sum_{i=1}^m|\gamma_i(\{\tau\})|\le M\Bigl[\lm+\|p(\tau)\|+\sum_{i=1}^s\psi^u_i(\{\tau\})\Bigr]\quad\text{for }\;\tau\in\{0,\T\}{\ \text{ with }\ M:=\max\{M_0,M_{\T}\}}.
\end{equation}
If in addition the right-hand side of \eqref{2bound} vanishes at both endpoints {$($see Remark~\ref{rem:atom-bound-form}$)$}, then we have $\gamma(\{0\})=\gamma(\{\T\})=0$, which excludes the Gamkrelidze--Arutyunov endpoint degeneracy phenomenon at both endpoints simultaneously.
\end{corollary}
{\bf Proof}. {The bound at $\tau=0$ follows from Theorem~\ref{ThNOCcq-atom}{\bf(ii)}. The bound at $\tau=\T$ follows from Theorem~\ref{ThNOCcq-atom}{\bf(i)}, which holds for every collection of multipliers, in particular for the selected one.} When both right-hand sides vanish, both atoms vanish too. {Hence $|\gamma|(\{0,\T\})=0$. Note that this does not imply $\supp\gamma\subset(0,\T)$, since a support is closed.} $\h$

\begin{remark}[reading the endpoint atom bounds]\label{rem:disc-CQ-sketch}
{The left-hand side of \eqref{e:disc-CQ-bd} is the atomic mass $\gamma(\{0\})$ of the sweeping measure at the active initial endpoint. The right-hand side collects only {\em nondegenerate} quantities, namely, the cost multiplier $\lm$ and the control constraint atom $\psi^u(\{0\})$. Thus the bound expresses a {\em transfer principle}. The sweeping measure cannot carry an atom of its own at $0$. Whatever mass it holds there is {\em borrowed} from the nondegenerate part of the system. This excludes the degenerate Dirac multiplier when the borrowed quantities vanish, as in the abnormal regime of Theorem~\ref{ThNOCcq}.

The constant $M_0$ is proportional to $\kappa_0^{-1}$, which measures the linear independence of the active gradients. It becomes large when the active gradients are nearly dependent. For example, for the gradients $(1,0)$ and $(1,\epsilon)$, the number $\kappa_0^{-1}$ is of order $1/\epsilon$.

Inward pointing conditions are the classical tool for nondegeneracy. The paper by Lopes--Fontes--de Pinho \cite{LopesFontesPinho2011} obtains nondegenerate necessary optimality conditions for polyhedral state constrained problems via such a condition. The paper by Arutyunov--Karamzin--Pereira \cite{ArutyunovKaramzinPereira2011} develops the Gamkrelidze form via a controllability type pointwise condition. In the sweeping setting with the moving set control, Theorem~\ref{thm:initial-atom} replaces these conditions by the selection of multipliers at the first mesh cell. This selection uses the adjoint relation \eqref{d8} of the moving set control.}
\end{remark}

{\smallskip
\noindent{\sc The initial control on the boundary of $U$.} Lemma~\ref{lem:first-cell} and Theorem~\ref{thm:initial-atom} do not use the interiority condition $\ou(0)\in{\rm int}\,U$. When $\ou(0)\in\partial U$, the term $\psi^u(\{0\})$ in \eqref{e:disc-CQ-bd} may be positive, and it may carry the whole nontriviality in \eqref{e:NTcq}. The following joint condition excludes this. Put $I^u_0:=\{j\mid\nu_j(\ou(0))=0\}$.

\begin{definition}[joint positive linear independence at $t=0$]\label{def:JQ}
We say that {\rm(JQ$_0$)} holds if the relations
\[
\sum_{i\in I_0}\gamma_i\nabla g_i(y_0)+\sum_{j\in I^u_0}\psi_j\nabla\nu_j\bigl(\ou(0)\bigr)=0,\qquad \gamma_i\in\R,\ \ \psi_j\ge0,
\]
imply that all the numbers $\gamma_i$ and $\psi_j$ vanish.
\end{definition}

Condition (JQ$_0$) combines the LICQ for the state constraints with the positive linear independence \textbf{(H3.2)} for the control constraints. It reduces to the LICQ at $y_0$ when $I^u_0=\varnothing$, i.e., when $\ou(0)\in{\rm int}\,U$. It reduces to the positive linear independence of $\{\nabla\nu_j(\ou(0))\}_{j\in I^u_0}$ when $I_0=\varnothing$. Under the LICQ at $y_0$, the Motzkin alternative shows that (JQ$_0$) is equivalent to the existence of a vector $w$ such that $\la\nabla g_i(y_0),w\ra=0$ for $i\in I_0$ and $\la\nabla\nu_j(\ou(0)),w\ra<0$ for $j\in I^u_0$. Thus the direction $w$ of the moving set control points strictly into $U$ and keeps the active state constraints to first order. This is a statement about the linearized geometry. It does not construct an admissible controlled trajectory. If $I_0\cup I^u_0=\varnothing$, then both atoms at $0$ vanish by the support relationships. Otherwise, (JQ$_0$) and a compactness argument give us
\[
\kappa^J_0:=\min\Bigl\{\Bigl\|\sum_{i\in I_0}\gamma_i\nabla g_i(y_0)+\sum_{j\in I^u_0}\psi_j\nabla\nu_j(\ou(0))\Bigr\|\Bigm|\sum_{i\in I_0}|\gamma_i|+\sum_{j\in I^u_0}\psi_j=1,\ \psi_j\ge0\Bigr\}>0 .
\]

\begin{corollary}[no initial atoms in the abnormal case]\label{cor:JQ}
Let the assumptions of Theorem~{\rm\ref{thm:initial-atom}} hold together with {\rm(JQ$_0$)}, and let $I_0\cup I^u_0\ne\varnothing$. Then the multipliers of Theorem~{\rm\ref{thm:initial-atom}} satisfy the estimate
\[
\sum_{i\in I_0}\bigl|\gamma_i(\{0\})\bigr|+\sum_{j\in I^u_0}\psi^u_j(\{0\})\le(\kappa^J_0)^{-1}\bigl(C_1+L_{\ell_2}(1+2\mu)\bigr)\lm .
\]
In particular, if $\lm=0$, then $\gamma(\{0\})=0$ and $\psi^u(\{0\})=0$. Moreover, we have $\lm+\|p(\T)\|+\|\psi^u\|_{TV((0,\T])}>0$.
\end{corollary}
{\bf Proof}. By the support relationships of Theorem~\ref{ThNOC}{\bf(v)}, we have $\gamma_i(\{0\})=0$ for $i\notin I_0$ and $\psi^u_j(\{0\})=0$ for $j\notin I^u_0$. The trace identity in the proof of Theorem~\ref{thm:initial-atom} can be written as
\[
\sum_{i\in I_0}\gamma_i(\{0\})\nabla g_i(y_0)+\sum_{j\in I^u_0}\psi^u_j(\{0\})\nabla\nu_j\bigl(\ou(0)\bigr)=q^u(0+)-q^u(0).
\]
The right-hand side has norm at most $\bigl(L_{\ell_2}(1+2\mu)+C_1\bigr)\lm$ by \eqref{NOC:qu} and Theorem~\ref{thm:initial-atom}(i). Since $\psi^u_j(\{0\})\ge0$, the definition of $\kappa^J_0$ gives us the estimate. If $\lm=0$ and $p(\T)=0$, then both atoms at $0$ vanish, and Theorem~\ref{ThNOCcq} yields $\|\psi^u\|_{TV((0,\T])}>0$. $\h$

\begin{remark}[on the role of (JQ$_0$)]\label{rem:JQ}
{\bf(a)} An inward pointing condition pairs its witness velocity only with the gradients $\nabla g_i$. Thus it does not exclude the atom $\psi^u(\{0\})$. Indeed, let $U=(-\infty,0]$ in Example~\ref{ex:CQ0}, i.e., $\nu(u)=u$. Then the reference relative velocity $f(0,1)-\dot\ou(0^+)=1$ points strictly into $C$ with margin one. However, (JQ$_0$) fails, since $\nabla g(0)=\nabla\nu(0)=1$. The collection $\lm=0$, $p\equiv0$, $q^u\equiv0$, $q^x=0$ on $(0,\T]$, $q^x(0)=-\sigma$, $\gamma=-\sigma\delta_{\{0\}}$, and $\psi^u=\sigma\delta_{\{0\}}$ with $\sigma>0$ satisfies all the displayed relations of Theorems~\ref{ThNOC} and~\ref{thm:initial-atom}. However, it is not a limit of the multipliers built from Lemma~\ref{lem:first-cell} in this example. Indeed, the sampled reference process has zero discrete cost, and every term of the discrete cost is nonnegative. Hence it is the unique optimal solution of $(P_k)$. Along it, $x^k_j-u^k_j=j/k>0$ for $j\ge1$, so $\gg^k_j=0$ for $j\ge1$. Moreover, $\gg^k_0=0$ by Lemma~\ref{lem:first-cell}. Thus every selected limit has $\gamma=0$.

{\bf(b)} If (JQ$_0$) fails, then there is always a pair $(\gamma_0,\psi_0)$ with $\psi_0\ge0$, $\psi_0\ne0$, and $\nabla g(y_0)^*\gamma_0+\nabla\nu(\ou(0))^*\psi_0=0$. A collection as in (a) then satisfies all the displayed relations. It is open whether, in some problem where (JQ$_0$) fails, the multipliers built from Lemma~\ref{lem:first-cell} can have this form. The example in (a) is not evidence for this.

{\bf(c)} The proofs of Lemma~\ref{lem:first-cell} and Theorem~\ref{thm:initial-atom} use the adjoint relation \eqref{d8} of the moving set control at the cells $j\ge1$. Three situations should be distinguished: a reference control $\ou\equiv0$; the moving set control fixed at zero in both $(P)$ and $(P_k)$; and the $u$-variables and their adjoint equations removed. The first situation is covered by the results above. In the last two situations, relation \eqref{d8} is not available, and a strengthened statement requires a separate argument. There, the enhanced nontriviality \eqref{NOC:NT-enh} keeps its interiority assumption. This concerns, in particular, the models of Section~\ref{S7}.
\end{remark}}

\section{Applications to Practical Modeling}\label{S7}
\setcounter{equation}{0}

In this section, we consider some models of practical interest and formulate for them sweeping optimal control problems. Using necessary optimality conditions established in Theorem~\ref{ThNOC} allows us to determine the structures of optimal solutions and calculate their parameters under certain simplifications. In both models below, the moving set is fixed, i.e., $u\equiv0$. In this case, we apply Theorem~\ref{ThNOC} with the $u$-components $(p^u,q^u,\psi^u)$ removed; the proof is the same.

\subsection{Minimum-time navigation around a circular keep-out zone}\label{sec:obstacle}

Our first attention is paid to the following dynamic model of steering a vehicle around a circular keep-out zone (a no-fly disk, a protected area). A vehicle of position $x\in\R^2$ is driven by a bounded control $a$, $\|a\|\le1$ ($A=\overline{\B}$) and must stay in the {\em nonconvex prox-regular set}
$$
C=\big\{y\in\R^2\;\big|\;g(y):=\|y\|^2-1\ge0\big\},
$$
which is the exterior of the unit keep-out disk. Here $\nabla g(y)=2y$ and $\nabla^2 g\equiv2I\ne0$, i.e., we have the {\em curved} boundary that makes the sweeping system to be genuinely {\em nonpolyhedral}. Since $\nabla g(y)=2y$ is unbounded, we modify $g$ outside a large ball containing the motion, so that \textbf{(H2.2)} holds. This does not affect the example. This is the specialization of our basic problem $(P)$ with $u\equiv0$ and $f(x,a)=a$. Therefore, the controlled dynamics is
\[
\dot x(t)\in-N_C(x(t))+a(t),\qquad\|a(t)\|\le1
\]
giving us $N_C(x)=\{0\}$ and $\dot x=a$ in the interior of $C$ while on the boundary of $C$, the normal cone removes the inward (infeasible) velocity component and keeps the trajectory in $C$. The objective is to steer the trajectory from $x_0=(-d,0)$ to $x_1=(d,0)$, $d>1$ in {\em minimum time}. This corresponds to $(P)$ with  $\ell\equiv1$, $\varphi\equiv0$, $\Xi_x=\{x_1\}$, and $\Xi_T=[0,\infty)$. We determine the optimal trajectory directly from the necessary optimality conditions of Theorem~\ref{ThNOC} and the maximization condition of Corollary~\ref{cor:SweepingMP}.

\begin{enumerate}
\item[\textnormal{\bf(i)}] \textsc{Control law.} Here $f(x,a)=a$, $w^a=0$, and $v^x=\nabla_{\dot x}\ell=0$. Thus \eqref{NOC:psi-a} reads $q^x(t)\in N_A(\oa(t))$. Since $A$ is the unit ball, we get $\oa(t)=q^x(t)/\|q^x(t)\|$ whenever $q^x(t)\ne0$.

\item[\textnormal{\bf(ii)}] \textsc{Free arcs are straight.} Off the zone we have $\eta=\gamma=\xi=0$. Since $\nabla_xf=0$, it follows from \eqref{NOC:adj-ac}--\eqref{NOC:qx} that $q^x\equiv$const. Thus $\oa$ is constant and the motion is a straight segment at the unit speed.

\item[\textnormal{\bf(iii)}] \textsc{Boundary arcs.} The straight chord $x_0x_1$ pierces the disk, and so the trajectory must contact the boundary of $C$. The constraint acts through the {\em adjoint}. On the arc, the unit speed and $\|\oa\|\le1$ force $\eta=0$, and hence $\xi=0$. Since $\nabla g(\ox)=2\ox$, condition \eqref{NOC:qx} gives us $\dot q^x=-2\dot\gamma\,\ox$. By {\bf(i)}, the costate $q^x$ stays parallel to the unit tangent $\dot\ox$, whose derivative is $-\ox$. This forces $\|q^x\|$ to be constant and the sweeping measure to have the constant density $\dot\gamma\equiv\tfrac12\|q^x\|>0$.

\item[\textnormal{\bf(iv)}] \textsc{Tangency at the junctions.} A velocity jump at an interior junction would lengthen the path. Hence the optimal velocity is continuous while forcing each straight segment to meet the arc {\em tangentially}.

\item[\textnormal{\bf(v)}] \textsc{Free terminal time.} Along the whole path, $\dot\ox$ is a unit vector parallel to $q^x$, and $\|q^x\|$ is constant. Hence $\la q^x,\dot\ox\ra=\|q^x\|$, and the averaged Hamiltonian \eqref{NOC:Hbar} equals $\mathcal H=\|q^x\|-\lm$. Since $\varphi\equiv0$ and $\Xi_T=[0,\infty)$, the free-time component of \eqref{NOC:trans} gives us $\mathcal H=0$. Thus $\|q^x\|=\lm$. The nontriviality condition excludes $\lm=0$, and so we may take $\lm=1$. The unit speed means that the optimal time equals the path length.
\end{enumerate}

The obtained results are illustrated in Figure~\ref{fig:obstacle} with the tangent points $P_\pm=(\pm1/d,\sqrt{1-1/d^2})$, the tangent segments of length $\sqrt{d^2-1}$, and a top arc of angle $\pi-2\arccos(1/d)$. Therefore,
\[
T^*=2\sqrt{d^2-1}+\pi-2\arccos\tfrac1d .
\]
We get $T^*\to\pi$ as $d\downarrow1$ and $T^*=2d+O(d^{-1})$ as $d\to\infty$ with the arc shrinking to a point. 
More generally, for endpoints at distances $\rho_i:=\|x_i\|>1$ with angular separation $\theta$ whose connecting chord pierces the disk, the same synthesis provides the closed-form benchmark
\[
T^*=\sqrt{\rho_0^2-1}+\sqrt{\rho_1^2-1}+\theta-\arccos\tfrac1{\rho_0}-\arccos\tfrac1{\rho_1}.
\]

\begin{figure}[ht]
\centering
\begin{tikzpicture}[scale=1.5,>={Stealth[length=2.3mm,width=1.7mm]}]
\fill[BrickRed!9] (0,0) circle (1);
  \draw[very thick,black!70] (0,0) circle (1);
  \node[BrickRed!68!black,font=\footnotesize] at (0,0) {keep-out};
  \coordinate (x0) at (-2,0); \coordinate (x1) at (2,0);
  \coordinate (Pm) at (-0.5,0.8660); \coordinate (Pp) at (0.5,0.8660);
  \draw[BrickRed!70,dashed,thick] (x0) -- (x1);
  \draw[ultra thick,NavyBlue] (x0) -- (Pm);
  \draw[ultra thick,NavyBlue] (Pm) arc[start angle=120,end angle=60,radius=1];
  \draw[ultra thick,NavyBlue] (Pp) -- (x1);
  \fill[NavyBlue] (x0) circle (0.045) (x1) circle (0.045);
  \fill[BurntOrange] (Pm) circle (0.05) (Pp) circle (0.05);
  \node[font=\small,anchor=north] at (-2,-0.04) {$x_0$};
  \node[font=\small,anchor=north] at (2,-0.04) {$x_1$};
  \node[BurntOrange!80!black,font=\footnotesize,anchor=south east,inner sep=1pt] at (-0.52,0.90) {$P_-$};
  \node[BurntOrange!80!black,font=\footnotesize,anchor=south west,inner sep=1pt] at (0.52,0.90) {$P_+$};
  \node[NavyBlue!72!black,font=\footnotesize,anchor=south] at (0,1.04) {optimal path};
  \node[BrickRed!70!black,font=\footnotesize,anchor=south] at (-1.18,0.02) {forbidden};
\end{tikzpicture}
\caption{Minimum-time synthesis. The conditions force a straight tangent arc--tangent straight geodesic (blue) around the keep-out disk: straight off the zone (constant adjoint), tangent at the junctions $P_\pm$ (continuous adjoint), and a unit speed arc on the boundary of $C$ between them. The direct chord (red dashed) is infeasible. The optimal time is $T^*=2\sqrt{d^2-1}+\pi-2\arccos(1/d)$.}
\label{fig:obstacle}
\end{figure}
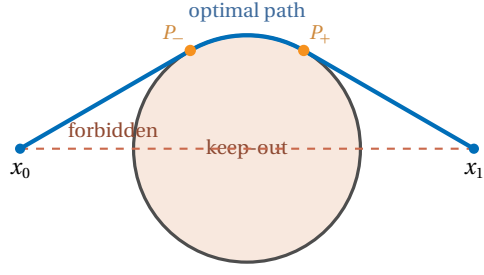\vspace*{-0.1in}

\subsection{Two-agent corridor crowd-motion model}\label{sec:corridor}

Following \cite[\S8]{CaoMordukhovich2026}, we formulate a two-agent corridor crowd motion \emph{benchmark} within the free-time framework $(P)$ and analyze it via the necessary optimality conditions of Theorem~\ref{ThNOC}. The treatment also {\em revisits and refines} the benchmark calculation of \cite[\S8]{CaoMordukhovich2026} by making explicit two simplifying features that are not compatible with the present free-time formulation. {\em First}, the spontaneous speeds therein were set to $s_i=\Lambda_i/T$. We instead keep them as \emph{fixed model parameters}, so that the perturbation remains of the form $f=f(x,a)$ required by the standing hypotheses rather than depending implicitly on the free terminal time $T$. {\em Second}, the absolute values in the $\ell^1$ terminal cost were removed there by the fixed terminal gradient $(-1,0,-1,0,\tau)$, which tacitly placed both agents on the same side of the destination at $T$. We retain here the absolute values and select a relevant terminal gradient from the subdifferential of the $\ell^1$-distance function only after the terminal signs of the candidate trajectory have been identified. The earlier closed form from \cite{CaoMordukhovich2026} becomes thereby a regime-dependent candidate being legitimate after an \emph{a posteriori} check of the terminal signs rather than a universal expression for the global optimal time.

\smallskip
\noindent{\bf The model}. Each agent is a rigid disk in $\R^2$ of radius $L_i>0$ and centre $x_i=(x_i^1,0)$, $i=1,2$, both moving along the horizontal axis toward a common destination $x_d=(x_d^1,0)$ with $x_1^1(0)<x_2^1(0)<x_d^1$. The planar embedding is kept (with the second coordinate frozen at $0$), and so that the colinear corridor appears as the special case of a genuinely planar model to be studied later. Writing $x=(x_1^1,0,x_2^1,0)\in\R^4$, the nonoverlapping requirement $x_2^1-x_1^1\ge L_1+L_2$ defines the admissible configuration set
\begin{equation}\label{e:corridorC}
C=\bigl\{x\in\R^4\;\big|\;g(x):=x_2^1-x_1^1-(L_1+L_2)\ge0\bigr\},\qquad \nabla g=(-1,0,1,0),
\end{equation}
a closed convex polyhedron (hence prox-regular) with $\nabla^2 g\equiv0$.

The {\em nominal/spontaneous velocity} of agent $i$ is $(s_i,0)$, where $s_1,s_2>0$ are \emph{two given constants} independent of the free terminal time $T$ (in contrast to the horizon-dependent choice $s_i=\Lambda_i/T$, which would force $f$ to depend on $T$ and leave the standing form $f=f(x,a)$). Thus the system spontaneous field is $U(x)=(s_1,0,s_2,0)$. Velocities must respect {\em nonpenetration}: they lie in the admissible velocity set
\[
V(x)=\bigl\{v=(v_1^1,0,v_2^1,0)\;\big|\;g(x)=0\Rightarrow\la\nabla g(x),v\ra=v_2^1-v_1^1\ge0\bigr\},
\]
i.e., at contact the front agent may not fall behind the rear. The {\em actual velocity} is the projection $\dot x=\Pi(U(x);V(x))$. Within the horizontal subspace $H:=\{(v_1,0,v_2,0)\}$, the polar of $V(x)$ is $N_C(x)$ at contact. Since $U(x)\in H$, Moreau's decomposition in $H$ gives us $U(x)=\Pi(U(x);V(x))+\Pi(U(x);N_C(x))$, and so we arrive at the {\em sweeping form} $\dot x\in -N_C(x)+U(x)$. Modulating each spontaneous speed by a scalar control $a_i$ through the perturbation $f(x,a)=(a_1 s_1,0,a_2 s_2,0)$ yields the {\em controlled sweeping process}
\begin{equation}\label{e:corridor-dyn}
\dot x(t)\in -N_C(x(t))+f(x(t),a(t)),\qquad x(0)=x_0\in C,\quad (x(\T),\T)\in\R^4\times[0,\infty).
\end{equation}
With target configuration $x^{\mathrm{des}}=(x_d^1,0,x_d^1,0)$, the {\em free-time optimal control problem} is to minimize the cost
\begin{equation}\label{e:corridor-P}
J[x,a,T]=\|x(T)-x^{\mathrm{des}}\|_1+\tau T+\tfrac12\int_0^T\|a(t)\|^2\,dt
\end{equation}
over $a=(a_1,a_2)\in A_1\times A_2$ and trajectories of \eqref{e:corridor-dyn}. This corresponds to the basic problem $(P)$ with $n=4$, $d=2$, the moving set $C_u(t)\equiv C$ ($u\equiv0$), the terminal cost $\varphi(x,T)=\|x-x^{\mathrm{des}}\|_1+\tau T$, and the running cost $\ell=\tfrac12\|a\|^2$. The trade-off parameter $\tau>0$ weighs arrival time against control effort.\vspace*{-0.1in}

\subsubsection*{Necessary optimality conditions via Theorem~\ref{ThNOC}}

Let $(\ox,\oa,\T)$ be an r.i.l.m.\ with $\T>0$. Because $u\equiv0$, the moving set adjoint $q^u$ and the control measure $\psi^u$ are absent. Abbreviating the two active scalar components of $q^x$ by $q^x_1$ and $q^x_2$, Theorem~\ref{ThNOC} gives us a cost multiplier $\lm\ge0$, an endpoint multiplier $\alpha\ge0$ for $x(\T)\in C$, adjoint arcs $p^x\in W^{1,2}([0,\T];\R^4)$ and $q^x\in BV([0,\T];\R^4)$, a scalar Borel measure $\gamma\in C^*([0,\T];\R)$, and subgradient selections
$w^x=v^x=0$ and $w^a=\oa$. The conclusions of the theorem are specified as follows.
\begin{enumerate}
\item[\textnormal{\bf(i)}] \textsc{Primal and complementary slackness conditions} in \eqref{NOC:par}
and \eqref{NOC:cs}. For a.e.\ $t\in[0,\T]$, we have $\dot{\ox}(t)=f(\ox(t),\oa(t))+\eta(t)\,\nabla g$ with $\eta(t)\ge0$, i.e.,
\[
\dot{\ox}_1^1(t)=\oa_1(t)\,s_1-\eta(t),\qquad \dot{\ox}_2^1(t)=\oa_2(t)\,s_2+\eta(t)
\]
together with the complementary slackness relationships $g(\ox(t))>0\Rightarrow\eta(t)=0$ and $\eta(t)>0\Rightarrow\la\nabla g,\lm v^x-q^x(t)\ra=0$, which ensure that $q^x_1(t)=q^x_2(t)$ whenever $\eta(t)>0$.

\item[\textnormal{\bf(ii)}] \textsc{Adjoint systems} in \eqref{NOC:adj-ac} and \eqref{NOC:qx}. Since $\nabla_xf=0$ and $w^x=0$, we have $\dot p^x(t)=0$ for a.e.\ $t$, and so $t\mapsto p^x(t)\equiv p^x$ is a fixed vector on $[0,\T]$. Since $\nabla^2 g\equiv0$, the Hessian coupling vector $\xi$ vanishes, and
we get for every $t\in[0,\T]$ the equality
\[
q^x(t)=p^x+\nabla g\,\gamma(J_t),\qquad\text{hence}\qquad q^x_1(t)+q^x_2(t)=p^x_1+p^x_2\ \text{is constant in }t.
\]

\item[\textnormal{\bf(iii)}] \textsc{Maximization condition} (Corollary~\ref{cor:SweepingMP}).
We take $A_1=A_2=[0,1]$ and assume that $\lm>0$. Then \eqref{NOC:psi-a} gives us $\oa_i(t)=\Pi_{[0,1]}\bigl(s_iq^x_i(t)/\lm\bigr)$. In particular, whenever the control is interior to $A_1\times A_2$, we get
\[
\oa_i(t)=\frac{s_i\,q^x_i(t)}{\lm}\;\mbox{ for a.e. }\;t\in[0,\T],\quad i=1,2.
\]
\item[\textnormal{\bf(iv)}] \textsc{Transversality conditions} in \eqref{NOC:alpha}--\eqref{NOC:trans}. We get $-p^x(\T)\in\lm\,\partial\|x(\T)-x^{\mathrm{des}}\|_1-\alpha\nabla g$, which can be componentwise written as $-p^x(\T)=\lm(\zeta_1,0,\zeta_2,0)-\alpha\nabla g$ with the regime dependent selection
\begin{equation}\label{e:corridor-subgrad}
\zeta_i\in\partial\bigl|\ox^1_i(\T)-x_d^1\bigr|=
\begin{cases}
\{-1\}, & \ox^1_i(\T)<x_d^1,\\[2pt]
[-1,1], & \ox^1_i(\T)=x_d^1,\\[2pt]
\{+1\}, & \ox^1_i(\T)>x_d^1.
\end{cases}
\end{equation}
The subgradient at the terminal point is therefore {\em not fixed} in advance while being selected from the subdifferential of the terminal $\ell^1$-distance only after the terminal signs of the candidate trajectory have been identified, never assumed. Finally, the transversality condition from \eqref{NOC:trans} yields the {\em free time identity} $\mathcal H=\lm\tau$. Wherever $\oa(t)$ is interior to $A_1\times A_2$, the maximization condition in {\bf(iii)} and the slackness relation $\eta\,(q^x_2-q^x_1)=0$ from {\bf(i)} tell us that
$$
\la q^x(t),\dot{\ox}(t)\ra=q^x_1\oa_1 s_1+q^x_2\oa_2 s_2=\lm\|\oa\|^2,
$$
and hence we arrive at the equivalent form
\begin{equation}\label{e:corridor-freetime}
\tfrac12\!\int_0^{\T}\|\oa(t)\|^2\,dt=\tau\T.
\end{equation}
\end{enumerate}

{\bf(v)--(vi)} $\;$ \textsc{Measure support and nontriviality}. The measure $\gamma$ is supported on the contact set

$\qquad\qquad\Sigma:=\{t\in[0,\T]:g(\ox(t))=0\}$, and nontriviality holds. Below we exhibit multipliers with $\lm=1$.\vspace*{0.05in}

Now we make precise the constancy of the control and the velocity used repeatedly below.\vspace*{0.05in}

\noindent\textbf{Claim (piecewise constancy of $\oa$ and $\dot{\ox}$).} \emph{{Let $\lm>0$, and let} $(t_1,t_2)\subset[0,\T]$ be an open interval on which the constraint is inactive, i.e., $g(\ox(t))>0$ for all $t\in(t_1,t_2)$. Then we have that the adjoint arc $q^x$, the control $\oa$, and the velocity $\dot{\ox}$ are constant on $(t_1,t_2)$, and thus each agent is moving in a straight line at constant speed.}\vspace*{0.05in}

\noindent{\bf Proof}. It follows from {\bf(ii)} that $p^x(t)\equiv p^x$ is a fixed vector. By {\bf(v)}, we get that $\operatorname{supp}\gamma\subset\Sigma$, and $\Sigma\cap(t_1,t_2)=\varnothing$, so $\gamma$ assigns no mass to $(t_1,t_2)$. For $t\in(t_1,t_2)$, split the interval $[t,\T]=[t,t_2)\cup[t_2,\T]$. Since $[t,t_2)\subset(t_1,t_2)$ carries no $\gamma$ mass, we have that $\gamma([t,\T])=\gamma([t_2,\T])=:c$ is independent of $t$ on $(t_1,t_2)$, and thus {\bf(ii)} tells us that the arc
\[
q^x(t)=p^x+c\,\nabla g\;\mbox{ for }\;t\in(t_1,t_2)
\]
is constant. The maximization condition from {\bf(iii)} ensures that $\oa_i(t)={\Pi_{[0,1]}(s_i q^x_i(t)/\lm)}$ is constant on $(t_1,t_2)$. Finally, the inactivity yields $\eta(t)=0$ on $(t_1,t_2)$ by the complementary slackness in {\bf(i)}, and so the derivative
\[
\dot{\ox}(t)=f(\ox(t),\oa(t))=(\oa_1 s_1,\,0,\,\oa_2 s_2,\,0)
\]
is a constant vector therein, where each $\ox^1_i$ is affine in $t$ with the constant slope $s_i\oa_i$. $\h$

\smallskip
\noindent It follows from the Claim that each agent moves along a straight line at constant speed on every interval where the agents are apart. We now use Theorem~\ref{ThNOC} to find a candidate r.i.l.m.\ in closed form. Then we verify that this candidate is indeed an r.i.l.m. Note that the necessary conditions alone do not ensure this, since they only select candidates.\vspace*{-0.1in}

\subsubsection*{The candidate from the necessary conditions}

Denote $\Lambda_i:=x_d^1-x_i^1(0)>0$ (distance to destination), $L:=L_1+L_2$, and $S:=\sqrt{s_1^2+s_2^2}$. The initial gap is $g_0:=x_2^1(0)-x_1^1(0)=\Lambda_1-\Lambda_2$. We have $g_0\ge L$ since $x_0\in C$.

\smallskip
\noindent{\sc Terminal cost floor.} The constraint $x_2^1(\T)-x_1^1(\T)\ge L$ implies that
\begin{equation}\label{e:corridor-floor}
\|x(\T)-x^{\mathrm{des}}\|_1=\bigl|x_1^1(\T)-x_d^1\bigr|+\bigl|x_2^1(\T)-x_d^1\bigr|\ \ge\ x_2^1(\T)-x_1^1(\T)\ \ge\ L
\end{equation}
for every admissible terminal configuration,
where an equality holds {\em if and only if} both inequalities are tight. The agents \emph{straddle} the destination, $x_1^1(\T)\le x_d^1\le x_2^1(\T)$ (tightness of the triangle inequality), and they {\em touch}, $x_2^1(\T)-x_1^1(\T)=L$ (tightness of the 
nonoverlap constraint). The two disks cannot both occupy the position $x_d^1$, and so the terminal cost {\em cannot vanish}; its floor is $L$, attained exactly by touching straddles.

\smallskip
\noindent{\sc Structure of the candidate.} Take $\lm=1$, and look for a candidate along which the active components of $q^x$ are equal and constant, i.e., $q^x_1=q^x_2\equiv\kappa$ on $[0,\T]$. By {\bf(ii)}, this holds with $p^x\equiv q^x$ and $\gamma=0$. Then {\bf(iii)} gives us the constant control
\[
\oa(t)\equiv\kappa\,(s_1,s_2),\quad\mbox{provided that}\quad\kappa s_1\le1.
\]
Before contact, the agents move with the velocities $\kappa s_1^2$ and $\kappa s_2^2$. Suppose that $s_1>s_2$. Then the rear agent catches up. The gap reaches $L$ at the contact time
\[
t_c:=\frac{g_0-L}{\kappa(s_1^2-s_2^2)}.
\]
After $t_c$, the agents move together with the common velocity $\kappa S^2/2$. The contact multiplier is $\eta=\kappa(s_1^2-s_2^2)/2>0$, and the slackness condition in {\bf(i)} holds because $q^x_1=q^x_2$. Recall that $\nabla g=(-1,0,1,0)$. Hence the normal cone term cancels in the sum of the two velocities, and every feasible process satisfies
\begin{equation}\label{e:corridor-sum}
\dot x_1^1(t)+\dot x_2^1(t)=s_1a_1(t)+s_2a_2(t)\quad\text{for a.e. }\;t\in[0,T].
\end{equation}
Thus the sum $\sigma:=x_1^1+x_2^1$ grows along the candidate at the rate $\kappa S^2$. The terminal cost floor suggests the touching straddle $x_1^1(\T)=x_d^1-L$, $x_2^1(\T)=x_d^1$. It is reached at $\T=(\Lambda_1+\Lambda_2-L)/(\kappa S^2)$, provided that $t_c\le\T$. A direct calculation shows that $t_c\le\T$ is equivalent to $(\Lambda_1-L)/s_1^2\le\Lambda_2/s_2^2$.

\smallskip
\noindent{\sc Free-time condition.} It remains to find $\kappa$. Along the candidate, we have $\la q^x,\dot\ox\ra=\kappa(\dot\ox_1^1+\dot\ox_2^1)=\kappa^2S^2$ and $\tfrac12\|\oa\|^2=\tfrac12\kappa^2S^2$. Hence $\mathcal H=\tfrac12\kappa^2S^2$, and the free-time identity $\mathcal H=\tau$ gives us $\kappa=\sqrt{2\tau}/S$. The endpoint relations in {\bf(iv)} hold with $\zeta_1=-1$, $\zeta_2=1-2\kappa\in[-1,1]$, $\alpha=1-\kappa\ge0$. Summarizing, we impose the conditions
\begin{equation}\label{e:corridor-regime}
s_1>s_2,\qquad \frac{\Lambda_1-L}{s_1^2}\le\frac{\Lambda_2}{s_2^2},\qquad \sqrt{2\tau}\,\max\{1,s_1\}\le\sqrt{s_1^2+s_2^2}.
\end{equation}
Under these conditions, Theorem~\ref{ThNOC} produces the candidate
\begin{equation}\label{e:Tbar-corridor}
\T=\frac{\Lambda_1+\Lambda_2-L}{\sqrt{2\tau}\,\sqrt{s_1^2+s_2^2}},\qquad \oa\equiv\frac{\sqrt{2\tau}}{\sqrt{s_1^2+s_2^2}}\,(s_1,s_2),
\end{equation}
and the value of the cost functional along this candidate is given in closed form by
\begin{equation}\label{e:Jstar-corridor}
J^\star=J[\ox,\oa,\T]=L+2\tau\T=L+\frac{\sqrt{2\tau}\,(\Lambda_1+\Lambda_2-L)}{\sqrt{s_1^2+s_2^2}}.
\end{equation}
Indeed, the three parts of the cost \eqref{e:corridor-P} are computed directly. The terminal cost equals the floor $L$, since the agents touch and straddle the destination at $\T$. The time cost is $\tau\T$. The control energy is $\tfrac12\|\oa\|^2\T=\tfrac12\kappa^2S^2\T=\tau\T$. Thus the time cost and the energy contribute equally, and their sum is $2\tau\T$. Observe that $J^\star$ coincides with the right-hand side of the lower estimate \eqref{e:corridor-lower} below.
Observe that the rear agent catches up and then pushes the front one. If the second condition in \eqref{e:corridor-regime} holds with equality, then $t_c=\T$, and the contact occurs only at the terminal time. If it fails, the front agent arrives before the rear one catches up.\vspace*{-0.1in}

\subsubsection*{Verification of the r.i.l.m.\ property}

Since $\ell=\tfrac12\|a\|^2$ does not depend on the velocities, the relaxed problem $(R)$ agrees with $(P)$; see Section~\ref{S2b}. So it suffices to show that the candidate is a local minimizer of $(P)$. In fact, we show that it is a global minimizer. The Cauchy--Schwarz inequality and \eqref{e:corridor-sum} give us $\|a\|^2\ge(\dot x_1^1+\dot x_2^1)^2/S^2$. Applying the Cauchy--Schwarz inequality once more, now in time, and then the inequality between the arithmetic and geometric means, we get
\[
\tau T+\tfrac12\int_0^T\|a(t)\|^2\,dt\ \ge\ \tau T+\frac{\bigl(\sigma(T)-\sigma(0)\bigr)^2}{2S^2T}\ \ge\ c\,\bigl|\sigma(T)-\sigma(0)\bigr|\quad\mbox{with}\quad c:=\frac{\sqrt{2\tau}}{S}\le1.
\]
Let $y_i:=x_i^1(T)-x_d^1$. Then $\sigma(T)-\sigma(0)=y_1+y_2+\Lambda_1+\Lambda_2$. Since $y_2-y_1\ge L$, the terminal cost satisfies $|y_1|+|y_2|\ge\max\{L,|y_1+y_2|\}$. Consequently, we have $J[x,a,T]\ge h(y_1+y_2)$ with $h(r):=\max\{L,|r|\}+c\,|r+\Lambda_1+\Lambda_2|$. The function $h$ is nonincreasing on $(-\infty,-L]$ and increasing on $[-L,\infty)$, so it attains its minimum at $r=-L$. In this way, we arrive at the lower estimate
\begin{equation}\label{e:corridor-lower}
J[x,a,T]\ \ge\ L+\frac{\sqrt{2\tau}\,(\Lambda_1+\Lambda_2-L)}{\sqrt{s_1^2+s_2^2}}\quad\mbox{for every feasible process}.
\end{equation}
The candidate attains this bound. Indeed, its control is parallel to $(s_1,s_2)$, its terminal time balances the two terms in the arithmetic--geometric mean inequality, and its terminal configuration is the touching straddle. Therefore, the candidate \eqref{e:Tbar-corridor} is a global minimizer, and hence an r.i.l.m. In particular, $J^\star$ in \eqref{e:Jstar-corridor} is the optimal value of the problem. If, in addition, $\sqrt{2\tau}<\sqrt{s_1^2+s_2^2}$, it is unique. Indeed, the equality cases above then determine the control, the terminal time, and the terminal configuration, while the sweeping process has a unique trajectory for each control.

\smallskip
\noindent{\sc Numerical example}. Consider the above model with the data: the radii $L_1=4$ and $L_2=2$, the initial centres $x_1^1(0)=-62$ and $x_2^1(0)=-24$, the destination $x_d^1=0$, the weight $\tau=\tfrac12$, and the nominal speeds $s_1=7$ and $s_2=4$. Thus $\Lambda_1=62$, $\Lambda_2=24$, $g_0=38$, $S=\sqrt{65}$, and $\kappa=1/\sqrt{65}$. Conditions \eqref{e:corridor-regime} hold, since $7>4$, $\tfrac{56}{49}\le\tfrac{24}{16}$, and $7\le\sqrt{65}$. Formulas \eqref{e:Tbar-corridor} and \eqref{e:Jstar-corridor} give us
\[
\T^\star=\frac{80}{\sqrt{65}}=\frac{16\sqrt{65}}{13}\approx9.923,\qquad J^\star=6+\frac{80}{\sqrt{65}}\approx15.923,\qquad \oa\equiv\frac{(7,4)}{\sqrt{65}}\approx(0.868,\,0.496).
\]
Before contact, the agents move with the velocities $49/\sqrt{65}\approx6.08$ and $16/\sqrt{65}\approx1.98$. The gap closes from $38$ to $6$ at $t_c=32\sqrt{65}/33\approx7.818$. Then the agents move together with the velocity $\sqrt{65}/2\approx4.03$, and the contact multiplier is $\eta=33/(2\sqrt{65})\approx2.05$. For comparison, consider the process with contact only at the terminal time, the constant controls $(\tfrac45,\tfrac35)$, and $T=10$. It also satisfies the necessary conditions of Theorem~\ref{ThNOC}, but it gives the larger value $J=16$. It is not even an r.i.l.m.: touching slightly earlier and then pushing lowers the cost within any localization neighborhood \eqref{local}. Figure~\ref{fig:corridor} shows the configurations of the r.i.l.m.\ before contact and at the terminal time.

\begin{figure}[ht]
\centering
\begin{tikzpicture}[>={Stealth[length=2.6mm,width=2mm]}]
  \draw[black!22] (0,0) -- (9.7,0);
  \fill[NavyBlue!16] (1.7,0) circle (0.8);
  \draw[very thick,NavyBlue!65!black] (1.7,0) circle (0.8);
  \node[NavyBlue!58!black] at (1.7,0) {$1$};
  \fill[BrickRed!16] (5.3,0) circle (0.4);
  \draw[very thick,BrickRed!72!black] (5.3,0) circle (0.4);
  \node[BrickRed!62!black,font=\small] at (5.3,0) {$2$};
  \draw[->,NavyBlue!65!black,very thick] (2.6,0) -- (4.1,0);
  \node[NavyBlue!58!black,font=\footnotesize,anchor=south,inner sep=2pt] at (3.35,0.05) {$\dot x_1^1\approx6.08$};
  \draw[->,BrickRed!72!black,very thick] (5.8,0) -- (6.3,0);
  \node[BrickRed!62!black,font=\footnotesize,anchor=south,inner sep=2pt] at (6.1,0.05) {$\dot x_2^1\approx1.98$};
  \draw[ForestGreen!50!black,dashed,thick] (9.1,-0.95) -- (9.1,0.95);
  \node[ForestGreen!42!black,font=\small,anchor=south] at (9.1,0.96) {$x_d$};
  \draw[black!40] (1.7,-0.95) -- (1.7,-1.22);
  \draw[black!40] (5.3,-0.95) -- (5.3,-1.22);
  \draw[<->,black!40] (1.7,-1.08) -- (5.3,-1.08);
  \node[black!55,font=\footnotesize,anchor=north,inner sep=2pt] at (3.5,-1.16) {$g_0=38$};
\end{tikzpicture}\\[0.6mm]
{\footnotesize (a) Before contact ($t=0$): the agents are $g_0=38$ apart and move right, agent~1 faster.}\\[3.6mm]
\begin{tikzpicture}[>={Stealth[length=2.6mm,width=2mm]}]
  \draw[black!22] (0,0) -- (9.7,0);
  \draw[ForestGreen!50!black,dashed,thick] (5.75,-1.05) -- (5.75,1.4);
  \node[ForestGreen!42!black,font=\small,anchor=south] at (5.75,1.42) {$x_d$};
  \fill[NavyBlue!16] (4.4,0) circle (0.9);
  \draw[very thick,NavyBlue!65!black] (4.4,0) circle (0.9);
  \node[NavyBlue!58!black] at (4.4,0) {$1$};
  \fill[BrickRed!16] (5.75,0) circle (0.45);
  \draw[very thick,BrickRed!72!black] (5.75,0) circle (0.45);
  \node[BrickRed!62!black,font=\small] at (5.75,0) {$2$};
  \fill[black] (5.3,0) circle (0.052);
  \node[black!72,font=\footnotesize,anchor=south,inner sep=1pt] at (5.15,0.72) {contact};
  \draw[black!50] (5.2,0.66) -- (5.29,0.07);
  \draw[black!40] (4.4,-0.95) -- (4.4,-1.22);
  \draw[black!40] (5.75,-0.95) -- (5.75,-1.22);
  \draw[<->,black!40] (4.4,-1.08) -- (5.75,-1.08);
  \node[black!55,font=\footnotesize,anchor=north,inner sep=2pt] at (5.08,-1.16) {$L_1+L_2=6$};
\end{tikzpicture}\\[0.6mm]
{\footnotesize (b) Terminal time $t=\T\approx9.92$: agent~2 at $x_d$, agent~1 at $x_d-L$, disks touching. The agents are in contact since $t_c\approx7.82$.}
\caption{The explicit instance of \S\ref{sec:corridor} (data and closed form in the text): before contact (a) and at the terminal straddle (b). The gap shrinks affinely from $g_0=38$ to $L$ at $t_c\approx7.82$. Then agent~1 pushes agent~2 until $\T\approx9.92$, where the terminal cost equals the floor $L$. Both panels are schematic, not to a common scale.}
\label{fig:corridor}
\end{figure}

\begin{remark}[Algorithmic validation]\label{rem:algorithm}
The discrete problems $(P_k)$ underlying Theorem~\ref{ThDiscNOC} may be solved by the {\em coderivative-based semi-Newton method} developed in \cite{KhanhMordukhovichPhat2026,AragonArtachoMordukhovichPerezAros2025} (cf.\ Algorithm~1 of \cite[\S8]{CaoMordukhovich2026}), which keeps the normal cone geometry of Theorem~\ref{Prop1} and exposes the contact multiplier $\eta$ as a physical contact force. The validation proceeds by: {\bf(i)}~checking convergence $\T^k\to\T$ against \eqref{e:Tbar-corridor}; {\bf(ii)}~verifying that the contact times concentrate as predicted (here, the contact interval $[t_c,\T]$); and {\bf(iii)}~confirming that the terminal configuration realizes the straddle \eqref{e:corridor-floor} rather than the spurious both-overshoot branch. Step~{\bf(iii)} is exactly where the earlier benchmark is corrected. Namely, the numerical solution reported in \cite[\S8]{CaoMordukhovich2026} (with $x_d^1=0$, $x_1^1(0)=-48$, $x_2^1(0)=-24$, $L=6$, $\tau=1$) yields a terminal configuration with $\ox^1_1(\T)<x_d^1<\ox^1_2(\T)$ and terminal contact $\ox^1_2(\T)-\ox^1_1(\T)=L$. Hence terminal cost is exactly $L$ by \eqref{e:corridor-floor} indicating that the front agent has already \emph{crossed} the destination. This straddle is incompatible with both left gradients $(-1,0,-1,0,\tau)$ used there to derive the closed form, and is reproduced consistently by the present regime aware selection \eqref{e:corridor-subgrad}.
\end{remark}

\section{Conclusions and Future Research}\label{S8}
\setcounter{equation}{0}

The paper provides a broad study of optimal control problems of the generalized Bolza type over the constrained sweeping dynamics with free time and prox-regular moving sets. Control actions of different classes enter both dynamic perturbations and moving sets. To the best of our knowledge, sweeping control problems of this generality have never been considered in the literature. By using an appropriate version of the method of discrete approximations, we establish novel necessary optimality conditions explicitly expressed in terms of the problem data. A crucial role in deriving such conditions is played by advanced tools of variational analysis and generalized differentiation married to constructive coderivative calculations for the sweeping dynamics. The obtained results are comprehensively discussed and illustrated by applications to some practical models.

Several questions remain open. First, the technical assumption \eqref{e:BV-pen-zero} concerns the discrete optimal solutions. We would like to replace it by a condition on the data of $(P)$. Second, the endpoint nondegeneracy of Section~\ref{S9b} uses the adjoint relation of the moving set control. Models in which the moving set is fixed require a separate argument. Third, it is not known whether the selected multipliers can degenerate when the joint qualification condition (JQ$_0$) fails.

Our future research will systematically address important improvements of the obtained results and the derivation of novel maximization conditions. For controls in the sweeping dynamics, we plan to significantly relax the tangential conditions in the maximum principle. Our main attention will be paid to controls in the moving sets, which have been largely underinvestigated in the literature. For such problems, we intend to establish maximization conditions in terms of a new construction of the measure modified Hamiltonian. 

The proposed theoretical developments will be applied to various models of practical interest. They include: multiagent corridor and planar crowd motion models, contact mechanics and robotic transport, obstacle avoidance, safety constrained learning, etc.

\end{document}